\documentclass[reqno,12pt]{amsart}

\usepackage[all]{xypic}
\usepackage{enumerate}
\usepackage{hyperref}
\usepackage{amssymb}

\DeclareMathOperator{\Gr}{Gr}
\DeclareMathOperator{\supp}{supp}
\DeclareMathOperator{\ddiv}{div}
\DeclareMathOperator{\pr}{pr}
\DeclareMathOperator{\Emb}{Emb}
\DeclareMathOperator{\vvol}{vol}
\DeclareMathOperator{\ann}{ann}
\DeclareMathOperator{\ad}{ad}
\DeclareMathOperator{\id}{id}

\DeclareMathOperator{\Diff}{Diff}
\DeclareMathOperator{\isomm}{isom}
\DeclareMathOperator{\Ham}{Ham}
\DeclareMathOperator{\img}{img}

\newcommand\isom{\textnormal{lin,\,inj}}
\newcommand\ham{\textnormal{ham}}
\newcommand\Sing{\textnormal{Sing}}
\newcommand\reg{\textnormal{reg}}
\newcommand\KKS{\textnormal{KKS}}
\newcommand\wt{\textnormal{wt}}
\newcommand\contact{\textnormal{contact}}
\newcommand\lin{\textnormal{lin}}
\newcommand\iso{\textnormal{iso}}

\theoremstyle{plain}
  \newtheorem{theorem}{Theorem}[section]
  \newtheorem{corollary}[theorem]{Corollary}
  \newtheorem{lemma}[theorem]{Lemma}
  \newtheorem{proposition}[theorem]{Proposition}
\theoremstyle{definition}
  
  \newtheorem{example}[theorem]{Example}
\theoremstyle{remark}
  \newtheorem{remark}[theorem]{Remark}

\begin{document}

\title{A dual pair for the contact group}

\author{Stefan Haller}

\address{Stefan Haller,
         Department of Mathematics,
         University of Vienna,
         Oskar-Mor\-gen\-stern-Platz 1,
         1090 Vienna,
         Austria.}

\email{stefan.haller@univie.ac.at}

\author{Cornelia Vizman}

\address{Cornelia Vizman,
         Department of Mathematics,
         West University of Timi\c soara,
         Bd. V.P\^arvan 4,
         300223-Timi\c soara, 
         Romania.}

\email{cornelia.vizman@e-uvt.ro}

\begin{abstract}
Generalizing the canonical symplectization of contact manifolds, we construct an infinite dimensional non-linear Stiefel manifold of weighted embeddings into a contact manifold.
This space carries a symplectic structure such that the contact group and the group of reparametrizations act in a Hamiltonian fashion with equivariant moment maps, respectively, giving rise to a dual pair, called the EPContact dual pair.
Via symplectic reduction, this dual pair provides a conceptual identification of non-linear Grassmannians of weighted submanifolds with certain coadjoint orbits of the contact group.
Moreover, the EPContact dual pair gives rise to singular solutions for the geodesic equation on the group of contact diffeomorphisms.
For the projectivized cotangent bundle, the EPContact dual pair is closely related to the EPDiff dual pair due to Holm and Marsden, and leads to a geometric description of some coadjoint orbits of the full diffeomorphism group.
\end{abstract}



\maketitle


\section{Introduction}\label{S:intro}

Every contact manifold gives rise to a symplectic manifold in a canonical way.
If the contact structure is described by a $1$-form $\alpha$ on $P$, then this symplectic manifold can be described as $P\times(\mathbb R\setminus0)$ with the symplectic form $d(t\alpha)$, where $t$ denotes the projection onto the second factor.
Regarding the contact structure as a subbundle of hyperplanes, $\xi\subseteq TP$, and denoting the corresponding line bundle over $P$ by $L:=TP/\xi$, this symplectization can be described more naturally as $M=L^*\setminus P$, with the symplectic form induced from the canonical symplectic form on $T^*P$ via the natural vector bundle inclusion $L^*\subseteq T^*P$.

The group of contact diffeomorphisms, $\Diff(P,\xi)$, acts on $M$ in a natural way, preserving the symplectic structure.
This action is in fact Hamiltonian and admits an equivariant moment map.
This moment map identifies (unions of) connected components of the symplectization $M$ with certain coadjoint orbits of the contact group.

\subsection{The EPContact dual pair}

In this paper we will introduce a natural infinite dimensional generalization $\mathcal M$ of the symplectization $M=L^*\setminus P$ with similar features.
To this end we fix a closed manifold $S$, we denote by $|\Lambda|_S$ its line bundle of densities, and we consider the space $\mathcal M$ of line bundle homomorphisms from $|\Lambda|_S^*\to S$ to $L^*\to P$ which restrict to a linear isomorphism on each fiber.
Every volume density on $S$ provides an identification $\mathcal M\cong C^\infty(S,M)$ and permits to regard elements $\Phi\in\mathcal M$ as pairs consisting of a map $\varphi\colon S\to P$ together with a contact form for $\xi$ along this map.
This space $\mathcal M$ can be equipped with the structure of a Fr\'echet manifold in a natural way, and admits a canonical (weakly non-degenerate) symplectic form.
The symplectization $M$ can be recovered by choosing $S$ to be a single point.

The contact group acts on $\mathcal M$ in a natural way, preserving the symplectic structure.
This action is Hamiltonian and admits an equivariant moment map, see Proposition~\ref{P:momentmaps}.
Furthermore, the group of reparametrizations, $\Diff(S)$, acts on $\mathcal M$ in a Hamiltonian fashion, also admitting an equivariant moment map.
On the non-linear Stiefel manifold of weighted embeddings, $\mathcal E\subseteq\mathcal M$, the latter action is free.
We show that the restrictions of these moment maps to $\mathcal E$,
\begin{equation}\label{epc}
\mathfrak X(P,\xi)^*\xleftarrow{\quad J_L^{\mathcal E}\quad}\mathcal E\xrightarrow{\quad J_R^{\mathcal E}\quad}\Omega^1(S,|\Lambda|_S)\subseteq\mathfrak X(S)^*,
\end{equation}
constitute a symplectic dual pair in the sense of Weinstein \cite{W83}, see Theorem~\ref{T:dp}.
Here $\mathfrak X(P,\xi)$ denotes the Lie algebra of contact vector fields on $P$, $\mathfrak X(S)$ denotes the Lie algebra of all vector fields on $S$, and $\Omega^1(S,|\Lambda|_S)$ denotes the space of smooth $1$-form densities on $S$.
The moment maps are given by $\langle J_L^{\mathcal E}(\Phi),X\rangle=\int_S\Phi(X\circ\varphi)$ for all $X\in\mathfrak X(P,\xi)$, and $\langle J_R^{\mathcal E}(\Phi),Z\rangle=\int_S\Phi(T\varphi\circ Z)$ for all $Z\in\mathfrak X(S)$.

Actually, we will show a stronger statement: The group $\Diff(S)$ acts freely and transitively on the fibers of $J_L^{\mathcal E}$, and the group $\Diff(P,\xi)$ acts locally transitive on the level sets of $J_R^{\mathcal E}$, see Proposition~\ref{jleg} and Theorem~\ref{T:Erho}.
Moreover, we will see that the level sets of both moment maps are smooth submanifolds of $\mathcal E$.
The dual pair in \eqref{epc} will be referred to as the \emph{EPContact dual pair}, because the left leg provides singular solutions of the EPContact equation, i.e., the Euler--Poincar\'e equation associated with the group of contact diffeomorphisms.

Recall that the projectivized cotangent bundle of a manifold $Q$ admits a canonical contact structure.
The EPContact dual pair corresponding to the projectivized cotangent bundle of $Q$ is closely related to the EPDiff dual pair, due to Holm--Marsden \cite{HM}, associated to the action of $\Diff(Q)$ and $\Diff(S)$ on $T^*\Emb(S,Q)$, the cotangent bundle of embeddings of $S$ into $Q$, see Section~\ref{S:EPDiff}.
This comparison leads to a geometric interpretation of some coadjoint orbits of $\Diff(Q)$, see Corollary~\ref{C:GisoQ}.

\subsection{Coadjoint orbits of the contact group}

The EPContact dual pair will be used to identify coadjoint orbits of the contact group via symplectic reduction for the reparametrization action, following the general principle: Symplectic reduction on one leg of a dual pair of moment maps leads to coadjoint orbits of the other group.
The same principle was used in \cite{GBV17}, where symplectic reduction on the right leg of the ideal fluid dual pair due to Marsden and Weinstein \cite{MW} led to coadjoint orbits of the Hamiltonian group consisting of weighted isotropic submanifolds of the symplectic manifold \cite{W90,Lee}.

To make this more precise, consider the non-linear Grassmannian of weighted submanifolds, $\mathcal G=\mathcal E/\Diff(S)$, consisting of pairs $(N,\gamma)$ where $N$ is a submanifold of type $S$ in $P$ and $\gamma\colon|\Lambda|^*_N\to L|_N^*$ is an isomorphism of line bundles which may be regarded as being akin to a trivialization of the contact structure along $N$.
This space $\mathcal G$ is a Fr\'echet manifold in a natural way and the projection $\mathcal E\to\mathcal G$ is a smooth principal bundle with structure group $\Diff(S)$.
The moment map $J_L^{\mathcal E}$ descends to a $\Diff(P,\xi)$-equivariant injective immersion $\mathcal G\to\mathfrak X(P,\xi)^*$, which permits to identify orbits of the contact group in $\mathcal G$ with coadjoint orbits.
Each $1$-form density $\rho\in\Omega^1(S,|\Lambda|_S)$ gives rise to a reduced space $\mathcal G^\rho\subseteq\mathcal G$ given by 
$$
\mathcal G^\rho=(J_R^{\mathcal E})^{-1}(\mathcal O_\rho)/\Diff(S)=(J^{\mathcal E}_R)^{-1}(\rho)/\Diff(S,\rho),
$$ 
where $\mathcal O_\rho$ denotes the $\Diff(S)$-orbit through $\rho$, and $\Diff(S,\rho)$ is the isotropy group of $\rho$.

Reduction works best for the zero level.
The corresponding reduced space $\mathcal G^0$ coincides with the subset of weighted isotropic submanifolds, $\mathcal G^\iso\subseteq\mathcal G$.
We will see that $\mathcal G^\iso$ is a smooth submanifold of $\mathcal G$ and that the action of the contact group on $\mathcal G^\iso$ admits local smooth sections.
In particular, this action is locally transitive.
Hence, the restriction of the moment map, $\mathcal G^\iso\to\mathfrak X(P,\xi)^*$, identifies (unions of) connected components of $\mathcal G^\iso$ with coadjoint orbits of the contact group.
Moreover, this identification intertwines the Kostant--Kirillov--Souriau symplectic form with the reduced symplectic form on $\mathcal G^\iso$.
These facts are summarized in Theorem~\ref{T:Giso}.

The situation is more delicate with regard to reduction at more general levels.
In this case the reduced spaces are more singular subsets of $\mathcal G$ and it is unclear, if the contact group acts locally transitive on them.
If $\rho$ is a contact $1$-form density on $S$, i.e., if $\ker\rho$ is a contact structure on $S$, then the reduced space $\mathcal G^\rho$ consists of certain weighted contact submanifolds of $P$ which are of type $(S,\ker\rho)$.
This is an open condition on the submanifold in view of Gray's stability theorem.
The condition on the weight, however, is rather singular: The space of all admissible (for $\mathcal G^\rho$) weights on a fixed contact submanifold may be identified with the $\Diff(S,\ker\rho)$-orbit of $\rho$.
The situation is tamer if we specialize to $1$-dimensional $S$, see Example~\ref{Ex:S1}.
In particular, (unions of) connected components in the spaces of weighted transverse knots of fixed length in a contact $3$-manifold, may be identified with coadjoint orbits of the contact group.

\subsection{Singular solutions of the Euler--Poincar\'e equation}\label{1.3}

Another motivation for studying the EPContact dual pair is the construction of singular solutions of the geodesic equation on the group of contact diffeomorphisms equipped with a right invariant Riemannian metric. 
This works analogous to the EPDiff equation, where the EPDiff dual pair has been used by Holm and Marsden \cite{HM} to construct singular solutions for the geodesic equation on the full diffeomorphism group.
Similarly, point vortices in two dimensional ideal fluids, a geodesic equation on the group of volume preserving diffeomorphisms, have been described using a dual pair by Marsden--Weinstein \cite{MW}, see the appendix.
The same kind of argument has been applied for the Vlasov equation in kinetic theory by Holm--Tronci \cite{HT} using the ideal fluid dual pair, and for the Euler--Poincar\'e equations on the group of automorphisms of a principal bundle in \cite{GBTV} using the EPAut dual pair \cite{GBV16}.

In all these cases the singular solutions of the system are obtained, via a moment map, from a collective Hamiltonian dynamics on a symplectic manifold, referred to as Clebsch variables.
This moment map turns out to be the left leg of a dual pair associated to commuting actions on the manifold of embeddings, while the right leg moment map gives conserved quantities by Noether's theorem.
We show that for the group of contact diffeomorphisms the situation is similar.

To describe this in more detail, let us start by briefly reviewing the geodesic equation on a Lie group with respect to a right invariant Riemannian metric.
We write the inner product on the Lie algebra $\mathfrak g$ in the form $(u,v)=\langle Qu,v\rangle$, where the inertia operator $Q\colon\mathfrak g\to\mathfrak g^*$ is symmetric and strictly positive.
Formally, the right trivialized geodesic equation on the Lie algebra $\mathfrak g$ is the Euler--Arnold equation,
\begin{equation}\label{E:EA}
\tfrac{d}{dt}u=-\ad_u^\top u,
\end{equation}
where the adjoint of the adjoint action can be characterized by $(\ad_u^\top v,w):=(v,\ad_uw)$ for all $u,v,w\in\mathfrak g$.
In other words, $\ad_u^\top=Q^{-1}\ad_u^*Q$, where $\ad_u^*\colon\mathfrak g^*\to\mathfrak g^*$ denotes the coadjoint action characterized by $\langle\ad_u^*m,v\rangle=\langle m,\ad_uv\rangle$ for $u,v\in\mathfrak g$ and $m\in\mathfrak g^*$.

Via Legendre transformation, using the momentum $m:=Qu$, the Euler--Arnold equation \eqref{E:EA} becomes the Lie--Poisson equation,
\begin{equation}\label{lp}
\tfrac{d}{dt}m=-\ad_u^*m,
\end{equation}
which is the Hamiltonian equation on the Poisson manifold $\mathfrak g^*$ for the Hamiltonian 
$$
h\colon\mathfrak g^*\to\mathbb R, \quad h(m):=\tfrac12\langle m,Q^{-1}m\rangle.
$$
Its solutions are confined to coadjoint orbits, the symplectic leaves of $\mathfrak g^*$.

Let us now turn to the group of contact diffeomorphisms on a contact manifold $(P,\xi)$.
Recall that its Lie algebra can be canonically identified with the space of contact vector fields, $\mathfrak X(P,\xi)=\Gamma^\infty(L)$, where $L=TP/\xi$.
For simplicity, we will assume $P$ to be closed.
We consider $\mathfrak X(P,\xi)^*=\Gamma^{-\infty}(L^*\otimes|\Lambda|_P)$, the space of distributional sections of $L^*\otimes|\Lambda|_P$, where $|\Lambda|_P$ denotes the bundle of densities on $P$.
We assume that the inertia operator, $Q\colon\Gamma^\infty(L)\to\Gamma^\infty(L^*\otimes|\Lambda|_P)$, is a pseudo-differential operator of real order $s$ which is symmetric, strictly positive, invertible, and its inverse, $Q^{-1}\colon\Gamma^\infty(L^*\otimes|\Lambda|_P)\to\Gamma^\infty(L)$, is a pseudo-differential operator of order $-s$.
Hence, the corresponding inner product, $(u,v)=\langle Qu,v\rangle$, generates the Sobolev $H^{s/2}$ topology on $\Gamma(L)$.
Using elliptic theory, such inertia operators can be easily constructed.
For instance, we may use $Q=\phi(1+\Delta)^{s/2}$, where $\Delta$ is a Laplacian acting on $\Gamma(L)$ which is non-negative and formally self-adjoint with respect to a volume density on $P$ and a fiberwise Euclidean metric on $L$, and $\phi\colon L\to L^*\otimes|\Lambda|_P$ denotes the isomorphism of line bundles provided by these geometric choices.

The Hamiltonian function $h(m)=\tfrac12\langle m,Q^{-1}m\rangle$ is well defined on $\Gamma^{-s/2}(L^*\otimes|\Lambda|_P)$, the space of sections which are of Sobolev class $-s/2$.
Note that the Sobolev space $\Gamma^{-s/2}(L^*\otimes|\Lambda|_P)$ is invariant under the coadjoint action of $\Diff(P,\xi)$.
If $k\in\Gamma^{-\infty}(L\boxtimes L)$ denotes the Schwartz kernel of $Q^{-1}$, then 
$$
h(m)=\tfrac12\langle k,m\boxtimes m\rangle=\frac12\int_{(x,y)\in P\times P}m(x)k(x,y)m(y)
$$
extends continuously (regularization) to $m\in\Gamma^{-s/2}(L^*\otimes|\Lambda|_P)$.
Assuming
\begin{equation}\label{E:scodim}
s>\dim P-\dim S,
\end{equation}
the moment map $J^{\mathcal E}_L\colon\mathcal E\to\mathfrak X(P,\xi)^*$ takes values in $\Gamma^{-s/2}(L^*\otimes|\Lambda|_P)=\Gamma^{s/2}(L)^*$.
Indeed, for $\Phi\in\mathcal E$ the distribution $J_L^{\mathcal E}(\Phi)$ is the push forward of a smooth section on $S$ along a smooth embedding $S\to P$, cf.\ Remark~\ref{R:JLRalpha}.
According to a standard property of the trace map on Sobolev spaces, see for instance \cite[Proposition~1.6 in Chapter~4]{TayI}, it thus provides a continuous functional on $\Gamma^{s/2}(L)$.
The map $J^{\mathcal E}_L$ is actually smooth into $\Gamma^{-s/2}(L^*\otimes|\Lambda|_P)$.
Hence, the pull back of the Hamiltonian $h$ to $\mathcal E$,
$$
H\colon\mathcal E\to\mathbb R,\quad H:=h\circ J^{\mathcal E}_L,
$$
is smooth.
Although the symplectic form on $\mathcal E$ is only weakly non-degenerate, the function $H$ gives rise to a Hamiltonian vector field $X_H$ on (and tangential to) $\mathcal E$, cf.\ the discussion in \cite[Section~4.2.2]{EP}.
Indeed, since $J_L^{\mathcal E}$ is a moment map, we formally have $X_H(\Phi)=\zeta^{\mathcal E}_{Q^{-1}J^{\mathcal E}_L(\Phi)}(\Phi)$ and thus
\begin{equation}\label{E:XH}
X_H(\Phi)=\zeta^{L^*}_{Q^{-1}J^{\mathcal E}_L(\Phi)}\circ\Phi,
\end{equation}
where $\zeta^{\mathcal E}$ and $\zeta^{L^*}$ denote the infinitesimal $\Diff(P,\xi)$-actions on $\mathcal E$ and $L^*$, respectively, cf.\ \eqref{E:dhXM} and \eqref{E:zetaXM} below.
By microlocal regularity, $Q^{-1}J^{\mathcal E}_L(\Phi)$ is smooth along the submanifold $N$ in $P$ determined by $\Phi$, see for instance \cite[Corollary~9.4 in Chapter~7]{TayII} or \cite[Proposition~3.11 in Chapter IV\S3]{GS77}.
Furthermore, since $\zeta^{L^*}\colon\Gamma^\infty(L)\to\Gamma^\infty(TL^*)$ is essentially given by a first order differential operator, it extends to distributional sections, and $\zeta^{L^*}_{Q^{-1}J^{\mathcal E}_L(\Phi)}$ is smooth along $L^*|_N$.
In particular, the latter is smooth along $\Phi$ and thus $X_H(\Phi)$ is a tangent vector to $\mathcal E$ at $\Phi$, cf.~\eqref{E:XH}.

Every solution $\Phi_t\in\mathcal E$ of the Hamilton equation
\begin{equation}\label{E:XHE}
\tfrac{d}{dt}\Phi_t=X_H(\Phi_t)
\end{equation}
provides a singular (peakon) solution $u_t:=Q^{-1}J_L^{\mathcal E}(\Phi_t)\in\Gamma^{s/2}(L)$ of the Euler--Arnold equation \eqref{E:EA} with momentum $m_t:=J_L^{\mathcal E}(\Phi_t)\in\Gamma^{-s/2}(L^*\otimes|\Lambda|_P)$.
The support of the distributional momentum $m_t$ coincides with the smooth submanifold determined by $\Phi_t$, and this also coincides with the singular support of $u_t$.
Due to the dual pair property, each solution $\Phi_t$ of \eqref{E:XHE} remains in a level of the other moment map, $J_R^{\mathcal E}\colon\mathcal E\to\mathfrak X(S)^*$, and is thus confined to a $\Diff(P,\xi)$ orbit in $\mathcal E$.
Hence, its momentum $m_t=J_L^{\mathcal E}(\Phi_t)$ is constrained to a coadjoint orbit.

If $S$ is a single point, then the assumption in~\eqref{E:scodim} implies that the distributional kernel $k$ of $Q^{-1}$ is continuous.
In this case we have $\mathcal E=L^*\setminus P$ and $H$ is given by the (smooth) restriction of $k$ to the diagonal.

The initial value problem for the EPContact equation has been studied by Ebin and Preston in \cite{EP}.
They consider inertia operators of the form $Q=1+\Delta$, where the Laplacian is with respect to a Riemannian metric which is associated with the contact structure.

It appears to be interesting \cite{BD} to replace the class of inertia operators considered above with operators in the Heisenberg calculus \cite{BG88,T84,P08}, a calculus of pseudo-differential operators which is closely linked to the contact geometry on $P$.
Using the Rockland theorem, one can construct pseudo-differential operators $Q\colon\Gamma^\infty(L)\to\Gamma^\infty(L^*\otimes|\Lambda|_P)$ of Heisenberg order $s$ which are symmetric, strictly positive, invertible, and such that the inverse, $Q^{-1}\colon\Gamma^\infty(L^*\otimes|\Lambda|_P)\to\Gamma^\infty(L)$, is of Heisenberg order $-s$.
For instance, we may use $Q=\phi(1+\Delta)^{s/2}$, where $\Delta$ is a subLaplacian.
Everything mentioned above remains valid, provided the Sobolev spaces are being replaced with the corresponding spaces in the Heisenberg Sobolev scale and the assumption \eqref{E:scodim} is replaced by the stronger condition $s/2>\dim P-\dim S$.

\subsection{Structure of the paper} 

The remaining part of the paper is organized as follows.
In Section~\ref{S:St} we construct the EPContact dual pair.
In Section~\ref{S:hs} we show that the level sets of the right moment map are submanifolds on which the contact group acts locally transitive.
In Section~\ref{S:Gr} we study the reduced spaces obtained by factoring out the group of reparametrizations.
In Section~\ref{S:EPDiff} we compare the EPContact dual pair for the projectivized cotangent bundle with the EPDiff dual pair of Holm and Marsden.
In the appendix we provide a comparison with a dual pair due to Marsden and Weinstein for the Euler equation of an ideal fluid.

\subsection{Acknowledgments}
The first author would like to thank the West University of Ti\-mi\-\c soa\-ra for the warm hospitality and Shantanu Dave for a helpful reference.
He gratefully acknowledges the support of the Austrian Science Fund (FWF): project numbers P31663-N35 and Y963-N35.
The second author was partially supported by CNCS UEFISCDI, project number PN-III-P4-ID-PCE-2016-0778.

\section{Weighted non-linear Stiefel manifolds}\label{S:St}

The aim of this section is to construct the EPContact dual pair, see Theorem~\ref{T:dp}.

\subsection{Canonical symplectization of contact manifolds}\label{SS:P}

In this section we set up our notation and recall some well known facts about the symplectization of contact manifolds.
We emphasize the structure that will be generalized in the subsequent sections.
For more details we refer to \cite[Appendix~4.E]{A89} and \cite[Section 12.3]{MR}.

Consider a contact manifold $(P,\xi)$ where $\xi\subseteq TP$ denotes the contact subbundle.
We write $L:=TP/\xi$ for the corresponding line bundle.
The vector bundle projection of the dual line bundle will be denoted by $\pi^{L^*}\colon L^*\to P$.
The canonical projection $TP\to L$ permits to regard the dual bundle as a subbundle of the cotangent bundle, $L^*\subseteq T^*P$.
We denote by $\theta^{L^*}\in\Omega^1(L^*)$ the pull back of the canonical $1$-form on $T^*P$.
\footnote{If $\xi=\ker\alpha$ and $L^*\cong P\times\mathbb R$ denotes the trivialization provided by $\alpha$, then $\theta^{L^*}=t(\pi^{L^*})^*\alpha$, where $t$ denotes the projection onto the factor $\mathbb R$.}
Hence, the defining equation for $\theta^{L^*}$ is
\begin{equation}\label{E:thetaLs}
\theta^{L^*}_\beta(V)={\beta(T_\beta\pi^{L^*}\cdot V)},
\end{equation}
where $\beta\in L^*_x$, $x\in P$, and $V\in T_\beta L^*$.
The pairing in \eqref{E:thetaLs} can be viewed either as a pairing between $L^*_x$ and $L_x$ by considering the class of $T_\beta\pi^{L^*}\cdot V$ in $L_x=T_xP/\xi_x$, or as a pairing between $T^*P$ and $TP$ by considering $\beta$ an element of $L^*_x\subseteq T^*_xP$.
It is well known that the closed $2$-form 
$$
\omega^{L^*}:=d\theta^{L^*}\in\Omega^2(L^*)
$$ 
restricts to a symplectic form on $M:=L^*\setminus P$, 
which will be denoted by $\omega^M=d\theta^M$.
The symplectic manifold $(M,\omega^M)$ is called the symplectization of the contact manifold $(P,\xi)$.
Note that both forms are homogeneous of degree one with respect to the fiberwise scalar multiplication $\delta_t\colon L^*\to L^*$, that is $\delta_t^*\theta^{L^*}=t\theta^{L^*}$ and $\delta_t^*\omega^{L^*}=t\omega^{L^*}$ for all $t\in\mathbb R$.

\subsubsection*{The action by the contact group}

Let us write $\Diff(P,\xi)$ for the group of contact diffeomorphisms.
Since contact diffeomorphisms preserve the contact subbundle $\xi$, the $\Diff(P,\xi)$-action on $P$ lifts to an action on the total space of $L^*$.
For $g\in\Diff(P,\xi)$, we let $\Psi_g^{L^*}\in\Diff(L^*)$ denote the corresponding (fiberwise linear) diffeomorphism on $L^*$.
Clearly, $\pi^{L^*}\circ\Psi^{L^*}_g=g\circ\pi^{L^*}$, $\delta_t\circ\Psi^{L^*}_g=\Psi^{L^*}_g\circ\delta_t$, and $\Psi_{g_2g_1}^{L^*}=\Psi_{g_2}^{L^*}\Psi_{g_1}^{L^*}$ for all $g,g_1,g_2\in\Diff(P,\xi)$ and $t\in\mathbb R$.
Moreover, the contact group action preserves $\theta^{L^*}$ and $\omega^{L^*}$, that is $(\Psi_g^{L^*})^*\theta^{L^*}=\theta^{L^*}$ and $(\Psi_g^{L^*})^*\omega^{L^*}=\omega^{L^*}$ for all $g\in\Diff(P,\xi)$.
Noticing that the symplectic piece $M\subseteq L^*$ is invariant under the contact group action, we write $\Psi^M_g$ for the restricted action.

Let $\mathfrak X(P,\xi)$ denote the Lie algebra of contact vector fields.
Via the projection $TP\to L$, every (contact) vector field gives rise to a section of $L$ which may in turn be regarded as a fiberwise linear function on the total space of $L^*$.
This provides canonical identifications,
\begin{equation}\label{E:XPxi}
\mathfrak X(P,\xi)=\Gamma^\infty(L)=C^\infty_\lin(L^*),\qquad X\leftrightarrow\textrm{$X$ mod $\xi$}\leftrightarrow h_X,
\end{equation}
where $h_X\in C^\infty_\lin(L^*)$ is the fiberwise linear function given by $h_X(\beta)=\beta(X_x)$ for $\beta\in L_x^*$ and $x\in P$.
Clearly, this identification is equivariant, i.e.,
\begin{equation}\label{E:hXequi}
(\Psi^{L^*}_g)^*h_X=h_{g^*X}
\end{equation}
for all $g\in\Diff(P,\xi)$ and $X\in\mathfrak X(P,\xi)$.

For $X\in\mathfrak X(P,\xi)$, we denote the corresponding fundamental vector field (infinitesimal action) on the total space of $L^*$ by $\zeta^{L^*}_X\in\mathfrak X(L^*)$.
Clearly, 
\begin{equation}\label{E:zetaL*P}
T\pi^{L^*}\circ\zeta^{L^*}_X=X\circ\pi^{L^*},
\end{equation}
$T\delta_t\circ\zeta^{L^*}_X=\zeta^{L^*}_X\circ\delta_t$, $(\Psi^{L^*}_g)^*\zeta^{L^*}_X=\zeta^{L^*}_{g^*X}$, and $[\zeta^{L^*}_{X_1},\zeta^{L^*}_{X_2}]=\zeta^{L^*}_{[X_1,X_2]}$ for all $X,X_1,X_2\in\mathfrak X(P,\xi)$, $g\in\Diff(P,\xi)$ and $t\in\mathbb R$.
From the definition of $\theta^{L^*}$ in \eqref{E:thetaLs} one immediately gets
\begin{equation}\label{E:hX}
i_{\zeta^{L^*}_X}\theta^{L^*}=h_X
\end{equation}
for $X\in\mathfrak X(P,\xi)$.
Invariance of $\theta^{L^*}$ and $\omega^{L^*}$ yields infinitesimal invariance $L_{\zeta^{L^*}_X}\theta^{L^*}=0$ and $L_{\zeta^{L^*}_X}\omega^{L^*}=0$, respectively, for all $X\in\mathfrak X(P,\xi)$.
Using Cartan's formula and \eqref{E:hX}, we obtain 
\begin{equation}\label{E:dhX}
i_{\zeta^{L^*}_X}\omega^{L^*}=-dh_X
\end{equation}
as well as the following formula for the bracket of contact vector fields,
\begin{equation}\label{E:bracket}
h_{[X,Y]}=\zeta^{L^*}_X\cdot h_Y=-\zeta^{L^*}_Y\cdot h_X=\omega^{L^*}(\zeta^{L^*}_X,\zeta^{L^*}_Y),
\end{equation}
for all $X,Y\in\mathfrak X(P,\xi)$.

Over the symplectic piece $M=L^*\setminus P$ the Hamiltonian vector field corresponding to $h_X^M:=h_X|_M$ coincides with $\zeta^M_X:=\zeta^{L^*}_X|_M$, see \eqref{E:dhX}.
Moreover, \eqref{E:bracket} implies 
\begin{equation}\label{E:Poisson}
h^M_{[X,Y]}=\{h^M_X,h^M_Y\},
\end{equation}
where the right hand side denotes the Poisson bracket on $C^\infty(M)$.
The formulas \eqref{E:dhX} and \eqref{E:hXequi} above imply that the action of $\Diff(P,\xi)$ on $M$ is Hamiltonian with equivariant moment map 
\begin{equation}\label{E:JLs} 
J^M\colon M\to\mathfrak X(P,\xi)^*,\qquad\langle J^M(\beta),X\rangle:=h_X^M(\beta)=\beta(X),
\end{equation}
where $\beta\in M$ and $X\in\mathfrak X(P,\xi)$.

\begin{remark}
A slightly more explicit, yet less natural description is possible if the contact structure is described by a contact form $\alpha\in\Omega^1(P)$, that is, if $\xi=\ker\alpha$.
Such a contact form provides a trivialization $P\times\mathbb R\cong L^*\subseteq T^*P$, $(x,t)\leftrightarrow t\alpha_x$.
Via this identification we have $\theta^{L^*}=t(\pi^{L^*})^*\alpha$, and the fiberwise linear function $h_X$ from \eqref{E:XPxi} becomes $h_X(x,t)=t(i_X\alpha)(x)$ where $x\in P$ and $t\in\mathbb R$.
A diffeomorphism $g$ of $P$ is a contact diffeomorphism iff it preserves the contact form up to a conformal factor, i.e., iff there exists a (nowhere vanishing) function $a_g$ on $P$ such that $g^*\alpha=a_g\alpha$.
Similarly, a vector field $X$ on $P$ is a contact vector field iff it satisfies $L_X\alpha=\lambda_X\alpha$, for a conformal factor $\lambda_X\in C^\infty(P)$.
Both, the group action of $\Diff(P,\xi)$ and the Lie algebra action of $\mathfrak X(P,\xi)$ on $L^*$, written in the trivialization $L^*\cong P\times\mathbb R$, involve the conformal factors.
More explicitly, we have $\Psi^{L^*}_g(x,t)=(g(x),ta_g(x))$ and $\zeta_X^{L^*}(x,t)=(X(x),t\lambda_X(x)\partial_t)$.
\end{remark}

\subsubsection*{Coadjoint orbits}

It is well known that each connected component of a symplectic manifold is equivariantly symplectomorphic to a coadjoint orbit of its Hamiltonian group, see for instance \cite{GBV17}.
We will now formulate a similar statement for the group $\Diff_c(P,\xi)$ of compactly supported contact diffeomorphisms which can be considered as a special case of Theorem~\ref{T:Giso} below.

For $\beta\in M$, the isotropy subgroup $\Diff_c(P,\xi;\beta)$ is a closed Lie subgroup of $\Diff_c(P,\xi)$.
Moreover, the map provided by the action, $\Diff_c(P,\xi)\to M$, $g\mapsto\Psi^M_g(\beta)$, admits a local smooth right inverse defined in a neighborhood of $\beta$.
In particular, the group $\Diff_c(P,\xi)$ acts locally and infinitesimally transitive on $M$, and the $\Diff_c(P,\xi)$-orbit through $\beta$ is open and closed in $M$.
Denoting this orbit by $M_\beta$, the map $\Diff_c(P,\xi)\to M_\beta$ is a smooth principal bundle with structure group $\Diff_c(P,\xi;\beta)$.
Hence, 
$$
M_\beta=\Diff_c(P,\xi)/\Diff_c(P,\xi;\beta)
$$ 
may be regarded as a homogeneous space.
The moment map \eqref{E:JLs} induces an equivariant diffeomorphism between $M_\beta$ and the coadjoint orbit of $\Diff_c(P,\xi)$ through $J^M(\beta)\in\mathfrak X(P,\xi)^*$.
By infinitesimal equivariance of $J^M$ and \eqref{E:bracket}, this diffeomorphism intertwines the Kostant--Kirillov--Souriau symplectic form $\omega^\KKS$ with $\omega^M$. 
Indeed, for $\beta\in M$ and $X,Y\in\mathfrak X(P,\xi)$, we get
\begin{multline*}
((J^M)^*\omega^\KKS)(\zeta_X^M(\beta),\zeta_Y^M(\beta))
\\=\omega^\KKS\left(\zeta_X^{\mathfrak X(P,\xi)^*}(J^M(\beta)),\zeta_Y^{\mathfrak X(P,\xi)^*}(J^M(\beta))\right)
\\=\langle J^M(\beta),[X,Y]\rangle
\stackrel{\eqref{E:JLs}}{=}h^M_{[X,Y]}(\beta)
\stackrel{\eqref{E:bracket}}{=}\omega^M(\zeta_X^M(\beta),\zeta_Y^M(\beta)),
\end{multline*}
whence $(J^M)^*\omega^\KKS=\omega^M$.

In particular, each connected component of $M$ is equivariantly symplectomorphic to a coadjoint orbit of the identity component in $\Diff_c(P,\xi)$.
If $P$ connected and the contact structure is not coorientable, then $M$ is connected, hence a coadjoint orbit of $\Diff_c(P,\xi)$.

\subsection{Moment maps on a manifold of weighted maps}\label{SS:M}

In this section we introduce an infinite dimensional generalization $\mathcal L$ of $L^*$ that also carries a canonical $1$-form $\theta^{\mathcal L}$ which is invariant under a natural $\Diff(P,\xi)$-action.

To this end, we fix a closed manifold $S$.
We let $|\Lambda|_S$ denote the line bundle of densities \cite[Chapter~16]{L} on $S$, and we write $\pi^{|\Lambda|_S}\colon|\Lambda|_S\to S$ for the corresponding vector bundle projection.
Recall that sections of $|\Lambda|_S$ can be integrated over $S$ in a natural way.
Every orientation of $S$ provides an isomorphism of line bundles $|\Lambda|_S\cong\Lambda^{\dim(S)}T^*S$.
A nowhere vanishing density, i.e., a section in $\Gamma^\infty(|\Lambda|_S\setminus S)$, will be referred to as a \emph{volume density.}

We denote the space of line bundle homomorphisms from $|\Lambda|_S^*\to S$ to $L^*\to P$ by
$$
\mathcal L
:=C^\infty_\lin(|\Lambda|_S^*,L^*)
:=\left\{\Phi\in C^\infty(|\Lambda|_S^*,L^*)\middle|\forall t\in\mathbb R:\Phi\circ\delta^{|\Lambda|_S^*}_t=\delta_t^{L^*}\circ\Phi\right\}.
$$
There is a canonical map $\pi^{\mathcal L}\colon\mathcal L\to C^\infty(S,P)$, characterized by 
\begin{equation}\label{E:piLL}
\pi^{L^*}\circ\Phi=\pi^{\mathcal L}(\Phi)\circ\pi^{|\Lambda|_S^*},
\end{equation} 
for all $\Phi\in\mathcal L$.
For the fiber over $\varphi\in C^\infty(S,P)$ we have a canonical identification,
\begin{equation}\label{E:Lphi}
\mathcal L_\varphi:=(\pi^{\mathcal L})^{-1}(\varphi)=\Gamma^\infty(|\Lambda|_S\otimes\varphi^*L^*).
\end{equation}

The contact group $\Diff(P,\xi)$ acts from the left on $\mathcal L$, and the reparametrization group $\Diff(S)$ acts on $\mathcal L$ from the right in an obvious way.
More explicitly, these actions are given by
\begin{equation}\label{boxaction}
\Psi^{\mathcal L}_g(\Phi):=\Psi_g^{L^*}\circ\Phi
\qquad\text{and}\qquad
\psi_f^{\mathcal L}(\Phi):=\Phi\circ\psi^{|\Lambda|_S^*}_f,
\end{equation}
where $\Phi\in\mathcal L$, $g\in\Diff(P,\xi)$, $f\in\Diff(S)$, and $\psi^{|\Lambda|_S^*}_f\in\Diff(|\Lambda|_S^*)$ denotes the induced (fiberwise linear) action 
of $\Diff(S)$ on the total space of $|\Lambda|_S^*$.
The two actions on $\mathcal L$ commute, and the map $\pi^{\mathcal L}$ intertwines them with the corresponding actions on $C^\infty(S,P)$ given by 
$$
\Psi^{C^\infty(S,P)}_g(\varphi)=g\circ\varphi\qquad\text{and}\qquad\psi^{C^\infty(S,P)}_f(\varphi)=\varphi\circ f,
$$ 
where $g\in\Diff(P,\xi)$, $f\in\Diff(S)$, and $\varphi\in C^\infty(S,P)$.
More explicitly, we have
$\Psi^{\mathcal L}_{g_2g_1}=\Psi^{\mathcal L}_{g_2}\circ\Psi^{\mathcal L}_{g_1}$,
$\pi^{\mathcal L}\circ\Psi_g^{\mathcal L}=\Psi_g^{C^\infty(S,P)}\circ\pi^{\mathcal L}$,
$\psi^{\mathcal L}_{f_1f_2}=\psi^{\mathcal L}_{f_2}\circ\psi^{\mathcal L}_{f_1}$,
$\pi^{\mathcal L}\circ\psi_f^{\mathcal L}=\psi_f^{C^\infty(S,P)}\circ\pi^{\mathcal L}$, and
$\Psi^{\mathcal L}_g\circ\psi^{\mathcal L}_f=\psi^{\mathcal L}_f\circ\Psi^{\mathcal L}_g$,
for $g,g_1,g_2\in\Diff(P,\xi)$ and $f,f_1,f_2\in\Diff(S)$.

\begin{remark}\label{R:mutriv}
Let $\mu\in\Gamma^\infty(|\Lambda|_S\setminus S)$ be a volume density on $S$, i.e., a nowhere vanishing smooth section of $|\Lambda|_S$.
Such a volume density provides an identification
\begin{equation*}
\mathcal L\cong C^\infty(S,L^*),\qquad \Phi\leftrightarrow\phi=\Phi\circ\hat\mu,
\end{equation*}
where $\hat\mu\in\Gamma^\infty(|\Lambda|_S^*)$ denotes the section dual to $\mu$, that is $\hat\mu(\mu)=1$.
In this picture the actions on $\mathcal L$ take the form 
$$
\Psi^{\mathcal L}_g(\phi)=\Psi^{L^*}_g\circ\phi
\qquad\text{and}\qquad
\psi^{\mathcal L}_f(\phi)=\frac{f^*\mu}\mu\cdot(\phi\circ f),
$$
where $\phi\in C^\infty(S,L^*)$, $g\in\Diff(P,\xi)$ and $f\in\Diff(S)$.
\end{remark}

The space $\mathcal L$ can be equipped with the structure of a smooth Fr\'echet manifold such that the identification $\mathcal L\cong C^\infty(S,L^*)$ in Remark~\ref{R:mutriv} becomes a diffeomorphism, for each choice of volume density $\mu$.
The map $\pi^{\mathcal L}\colon\mathcal L\to C^\infty(S,P)$ is a smooth vector bundle.
The tangent space at $\Phi\in\mathcal L$ can be canonically identified as
\begin{equation}\label{spiral}
T_\Phi\mathcal L
=\left\{\eta\in C^\infty(|\Lambda|_S^*,TL^*)\left|\begin{array}{l}\pi^{TL^*}\circ\eta=\Phi\quad \textrm{and}\\\forall t\in\mathbb R:\eta\circ\delta^{|\Lambda|_S^*}_t=T\delta_t^{L^*}\circ\eta\end{array}\right\}\right..
\end{equation}
The actions of $\Diff(P,\xi)$ and $\Diff(S)$ on $\mathcal L$ are smooth.
For $X\in\mathfrak X(P,\xi)$ and $Z\in\mathfrak X(S)$, the corresponding fundamental vector fields are
\begin{equation}\label{E:zetaXM}
\zeta_X^{\mathcal L}(\Phi)=\zeta_X^{L^*}\circ\Phi
\qquad\text{and}\qquad
\zeta_Z^{\mathcal L}(\Phi)=T\Phi\circ\zeta^{|\Lambda|_S^*}_Z
\end{equation}
where $\Phi\in\mathcal L$ and $\zeta^{|\Lambda|_S^*}_Z\in\mathfrak X(|\Lambda|_S^*)$ denotes the fundamental vector field of the $\Diff(S)$-action on the total space of $|\Lambda|_S^*$.
Note that
\begin{equation}\label{E:zetaLambda*S}
T\pi^{|\Lambda|_S^*}\circ\zeta_Z^{|\Lambda|_S^*}=Z\circ\pi^{|\Lambda|_S^*}\qquad\text{and}\qquad\bigl(\delta_t^{|\Lambda|_S^*}\bigr)^*\zeta_Z^{|\Lambda|_S^*}=\zeta_Z^{|\Lambda|_S^*}.
\end{equation}
Clearly, $(\Psi^{\mathcal L}_g)^*\zeta^{\mathcal L}_X=\zeta^{\mathcal L}_{g^*X}$,
$\zeta^{\mathcal L}_{[X_1,X_2]}=[\zeta^{\mathcal L}_{X_1},\zeta^{\mathcal L}_{X_2}]$,
$T\pi^{\mathcal L}\circ\zeta^{\mathcal L}_X=\zeta^{C^\infty(S,P)}_X\circ\pi^{\mathcal L}$,
$(\psi^{\mathcal L}_f)^*\zeta^{\mathcal L}_Z=\zeta^{\mathcal L}_{f_*Z}$,
$\zeta^{\mathcal L}_{[Z_1,Z_2]}=-[\zeta^{\mathcal L}_{Z_1},\zeta^{\mathcal L}_{Z_2}]$,
$T\pi^{\mathcal L}\circ\zeta^{\mathcal L}_Z=\zeta^{C^\infty(S,P)}_Z\circ\pi^{\mathcal L}$, and
$[\zeta^{\mathcal L}_X,\zeta^{\mathcal L}_Z]=0$,
where $g\in\Diff(P,\xi)$, $X,X_1,X_2\in\mathfrak X(P,\xi)$, $f\in\Diff(S)$, $Z,Z_1,Z_2\in\mathfrak X(S)$.

\subsubsection*{The canonical $1$-form}

Consider the $1$-form $\theta^{\mathcal L}$ on $\mathcal L$ defined by
\begin{equation}\label{E:thetaL}
\theta^{\mathcal L}(\eta):=\int_S\theta^{L^*}(\eta),
\end{equation}
where $\eta\in T_\Phi\mathcal L$ and $\Phi\in\mathcal L$. 
Note here that, because of \eqref{spiral}, inserting $\eta$ into $\theta^{L^*}$ leads to a fiberwise linear map $\theta^{L^*}(\eta)\colon|\Lambda|_S^*\to\mathbb R$ which, when regarded as a section of $|\Lambda|_S$, may be integrated over $S$.
By invariance of $\theta^{L^*}$, the $1$-form $\theta^{\mathcal L}$ is invariant under both actions, i.e., we have 
$$
(\Psi_g^{\mathcal L})^*\theta^{\mathcal L}=\theta^{\mathcal L}\qquad\textrm{and}\qquad(\psi_f^{\mathcal L})^*\theta^{\mathcal L}=\theta^{\mathcal L}
$$ 
for all $g\in\Diff(P,\xi)$ and $f\in\Diff(S)$.
The corresponding infinitesimal invariance reads 
$$
L_{\zeta^{\mathcal L}_X}\theta^{\mathcal L}=0\qquad\textrm{and}\qquad L_{\zeta^{\mathcal L}_Z}\theta^{\mathcal L}=0,
$$
where $X\in\mathfrak X(P,\xi)$ and $Z\in\mathfrak X(S)$.

Moreover, we introduce the $2$-form $\omega^{\mathcal L}:=d\theta^{\mathcal L}$ on $\mathcal L$.
By invariance of $\theta^{\mathcal L}$, this $2$-form is invariant under both actions too.
More explicitly, we have $(\Psi_g^{\mathcal L})^*\omega^{\mathcal L}=\omega^{\mathcal L}$ and $(\psi_f^{\mathcal L})^*\omega^{\mathcal L}=\omega^{\mathcal L}$ for $g\in\Diff(P,\xi)$ and $f\in\Diff(S)$, as well as infinitesimal invariance $L_{\zeta^{\mathcal L}_X}\omega^{\mathcal L}=0$ and $L_{\zeta^{\mathcal L}_Z}\omega^{\mathcal L}=0$ for $X\in\mathfrak X(P,\xi)$ and $Z\in\mathfrak X(S)$.
Clearly, see \cite{V11,GBV12},
\begin{equation}\label{E:omegaM}
\omega^{\mathcal L}(\eta_1,\eta_2)=\int_S\omega^{L^*}(\eta_1,\eta_2)
\end{equation}
where $\eta_1,\eta_2\in T_\Phi\mathcal L$ and $\Phi\in\mathcal L$.
As before, the fiberwise linear function $\omega^{L^*}(\eta_1,\eta_2)$ on $|\Lambda|_S^*$ may be regarded as a section of $|\Lambda|_S$ which can be integrated over $S$.

The exact $2$-form $\omega^{\mathcal L}=d\theta^{\mathcal L}$ is not (weakly) non-degenerate, because $\omega^{L^*}$ is not symplectic on all of $L^*$.
In the subsequent section, we will restrict to an invariant open subset of $\mathcal L$ on which $\omega^{\mathcal L}$ is (weakly) symplectic.
On this symplectic part, both actions are Hamiltonian with equivariant moment map.
This is a well known formal consequence of the fact that these actions preserve the $1$-form $\theta^{\mathcal L}$, see for instance \cite[Section~12.3]{MR}.
The corresponding Hamiltonian functions and moment maps are given by contraction of the fundamental vector fields with the canonical $1$-form.
However, these geometric objects make sense on all of $\mathcal L$.
Hence, we will now formulate their fundamental relations on $\mathcal L$.

\subsubsection*{The left moment map}

For $X\in\mathfrak X(P,\xi)$, consider the function $h^{\mathcal L}_X\colon\mathcal L\to\mathbb R$ defined by
\begin{equation}\label{E:hXM}
i_{\zeta^{\mathcal L}_X}\theta^{\mathcal L}=:h^{\mathcal L}_X.
\end{equation}
Using the infinitesimal invariance, $L_{\zeta^{\mathcal L}_X}\theta^{\mathcal L}=0$, we obtain
\begin{equation}\label{E:dhXM}
i_{\zeta^{\mathcal L}_X}\omega^{\mathcal L}=-dh^{\mathcal L}_X
\end{equation}
analogous to \eqref{E:dhX}, as well as
\begin{equation}\label{E:bracketM}
h^{\mathcal L}_{[X,Y]}=\zeta^{\mathcal L}_X\cdot h^{\mathcal L}_Y
=-\zeta^{\mathcal L}_Y\cdot h^{\mathcal L}_X
=\omega^{\mathcal L}(\zeta^{\mathcal L}_X,\zeta^{\mathcal L}_Y),
\end{equation}
for all $X,Y\in\mathfrak X(P,\xi)$, cf.~\eqref{E:bracket}.
From the invariance of $\theta^{\mathcal L}$ we obtain, cf.~\eqref{E:hXequi}
\begin{equation}\label{E:PghXM}
(\Psi_g^{\mathcal L})^*h^{\mathcal L}_X=h^{\mathcal L}_{g^*X}
\qquad\text{and}\qquad
(\psi_f^{\mathcal L})^*h^{\mathcal L}_X=h^{\mathcal L}_X
\end{equation}
for all $f\in\Diff(S)$, $g\in\Diff(P,\xi)$, and $X\in\mathfrak X(P,\xi)$.
We introduce a smooth map
\begin{equation}\label{plus}
J^{\mathcal L}_L\colon\mathcal L\to\mathfrak X(P,\xi)^*
\end{equation}
by putting $\langle J_L^{\mathcal L},X\rangle:=h^{\mathcal L}_X$, that is,
\begin{equation}\label{E:JLdef}
\langle J_L^{\mathcal L}(\Phi),X\rangle:=h^{\mathcal L}_X(\Phi)=\theta^{\mathcal L}(\zeta^{\mathcal L}_X(\Phi)),
\end{equation}
where $\Phi\in\mathcal L$ and $X\in\mathfrak X(P,\xi)$.
The equations in \eqref{E:PghXM} may be written in the form
\begin{equation}\label{E:PghXMJ}
\langle J_L^{\mathcal L}\circ\Psi_g^{\mathcal L},X\rangle=\langle\Psi_g^{\mathcal L},g^*X\rangle
\qquad\text{and}\qquad
J_L^{\mathcal L}\circ\psi^{\mathcal L}_f=J_L^{\mathcal L},
\end{equation}
where $g\in\Diff(P,\xi)$, $X\in\mathfrak X(P,\xi)$, and $f\in\Diff(S)$.
Combining \eqref{E:hXM}, \eqref{E:zetaXM}, \eqref{E:thetaL}, \eqref{E:thetaLs}, \eqref{E:zetaL*P}, and \eqref{E:piLL}, we obtain
\begin{equation}\label{E:hXMint}
h_X^{\mathcal L}(\Phi)
=\int_S\Phi(X\circ\varphi),
\end{equation}
where $\varphi\in C^\infty(S,P)$, $\Phi\in\mathcal L_\varphi=\Gamma^\infty(|\Lambda|_S\otimes\varphi^*L^*)$, and $X\in\mathfrak X(P,\xi)=\Gamma^\infty(L)$, cf.~\eqref{E:Lphi} and \eqref{E:XPxi}.
Here we use the canonical contraction between $L^*\subseteq T^*P$ and $TP$ to obtain the density $\Phi(X\circ\varphi)\in\Gamma^\infty(|\Lambda|_S)$. 
More explicitly, the verification of \eqref{E:hXMint} reads:
\begin{multline*}
h_X^{\mathcal L}(\Phi)
\stackrel{\eqref{E:hXM}}{=}\theta^{\mathcal L}(\zeta_X^{\mathcal L}(\Phi))
\stackrel{\eqref{E:zetaXM}}{=}\theta^{\mathcal L}(\zeta_X^{L^*}\circ\Phi)
\stackrel{\eqref{E:thetaL}}{=}\int_S\theta^{L^*}(\zeta_X^{L^*}\circ\Phi)
\\\stackrel{\eqref{E:thetaLs}}{=}\int_S\Phi(T\pi^{L^*}\circ\zeta_X^{L^*}\circ\Phi)
\stackrel{\eqref{E:zetaL*P}}{=}\int_S\Phi(X\circ\pi^{L^*}\circ\Phi)
\\\stackrel{\eqref{E:piLL}}{=}\int_S\Phi(X\circ\varphi\circ\pi^{|\Lambda|_S^*})
=\int_S\Phi(X\circ\varphi).
\end{multline*}

\subsubsection*{The right moment map}

For $Z\in\mathfrak X(S)$, consider the function $k^{\mathcal L}_Z\colon\mathcal L\to\mathbb R$ defined by
\begin{equation}\label{E:kZM}
i_{\zeta^{\mathcal L}_Z}\theta^{\mathcal L}=:k^{\mathcal L}_Z.
\end{equation}
Using the infinitesimal invariance, $L_{\zeta^{\mathcal L}_Z}\theta^{\mathcal L}=0$, we obtain
\begin{equation}\label{E:dkZM}
i_{\zeta^{\mathcal L}_Z}\omega^{\mathcal L}=-dk^{\mathcal L}_Z
\end{equation}
as well as
\begin{equation}\label{E:bracketMM}
k^{\mathcal L}_{[Z_1,Z_2]}=\zeta^{\mathcal L}_{Z_1}\cdot k^{\mathcal L}_{Z_2}
=-\zeta^{\mathcal L}_{Z_2}\cdot k^{\mathcal L}_{Z_1}
=\omega^{\mathcal L}(\zeta^{\mathcal L}_{Z_1},\zeta^{\mathcal L}_{Z_2}),
\end{equation}
for all $Z,Z_1,Z_2\in\mathfrak X(S)$.
From the invariance of $\theta^{\mathcal L}$ we obtain
\begin{equation}\label{E:PgkZM}
(\Psi_g^{\mathcal L})^*k^{\mathcal L}_Z=k^{\mathcal L}_Z
\qquad\text{and}\qquad
(\psi_f^{\mathcal L})^*k^{\mathcal L}_Z=k^{\mathcal L}_{f_*Z}
\end{equation}
for all $g\in\Diff(P,\xi)$, $f\in\Diff(S)$, and $Z\in\mathfrak X(S)$.
We introduce a smooth map
\begin{equation}\label{E:JRLdp}
J_R^{\mathcal L}\colon\mathcal L\to\Omega^1(S,|\Lambda|_S)\subseteq\mathfrak X(S)^*
\end{equation}
by putting $\langle J_R^{\mathcal L},Z\rangle:=k^{\mathcal L}_Z$, that is,
\begin{equation}\label{E:JRdef}
\langle J_R^{\mathcal L}(\Phi),Z\rangle:=k^{\mathcal L}_Z(\Phi)=\theta^{\mathcal L}(\zeta^{\mathcal L}_Z(\Phi)),
\end{equation}
where $\Phi\in\mathcal L$ and $Z\in\mathfrak X(S)$.
The equations in \eqref{E:PgkZM} may be written in the form
\begin{equation}\label{E:PgkZMJ}
\langle J_R^{\mathcal L}\circ\psi_f^{\mathcal L},Z\rangle=\langle\psi_f^{\mathcal L},f_*Z\rangle
\qquad\text{and}\qquad
J_R^{\mathcal L}\circ\Psi^{\mathcal L}_g=J_R^{\mathcal L},
\end{equation}
where $f\in\Diff(S)$, $Z\in\mathfrak X(S)$, and $g\in\Diff(P,\xi)$.
In view of \eqref{E:kZM}, \eqref{E:zetaXM}, \eqref{E:thetaL}, \eqref{E:thetaLs}, \eqref{E:piLL}, and \eqref{E:zetaLambda*S}, we have
\begin{equation}\label{E:kZMint}
k^{\mathcal L}_Z(\Phi)=\int_S\Phi(T\varphi\circ Z),
\end{equation}
where $\varphi\in C^\infty(S,P)$, $\Phi\in\mathcal L_\varphi=\Gamma^\infty(|\Lambda|_S\otimes\varphi^*L^*)$, and $Z\in\mathfrak X(S)$, cf.~\eqref{E:Lphi}.
As before, we use the canonical contraction between $L^*\subseteq T^*P$ and $TP$ to obtain a density $\Phi(T\varphi\circ Z)\in\Gamma^\infty(|\Lambda|_S)$.
More explicitly, the verification of \eqref{E:kZMint} reads:
\begin{multline*}
k^{\mathcal L}_Z(\Phi)
\stackrel{\eqref{E:kZM}}{=}\theta^{\mathcal L}(\zeta_Z^{\mathcal L}(\Phi))
\stackrel{\eqref{E:zetaXM}}{=}\theta^{\mathcal L}(T\Phi\circ\zeta_Z^{|\Lambda|_S^*})
\stackrel{\eqref{E:thetaL}}{=}\int_S\theta^{L^*}(T\Phi\circ\zeta_Z^{|\Lambda|_S^*})
\\\stackrel{\eqref{E:thetaLs}}{=}\int_S\Phi(T\pi^{L^*}\circ T\Phi\circ\zeta_Z^{|\Lambda|_S^*})
\stackrel{\eqref{E:piLL}}{=}\int_S\Phi(T\varphi\circ T\pi^{|\Lambda|_S^*}\circ\zeta_Z^{|\Lambda|_S^*})
\\\stackrel{\eqref{E:zetaLambda*S}}{=}\int_S\Phi(T\varphi\circ Z\circ\pi^{|\Lambda|_S^*})
=\int_S\Phi(T\varphi\circ Z).
\end{multline*}

It follows from \eqref{E:JRdef} and \eqref{E:kZMint} that $J_R^{\mathcal L}(\Phi)$ is indeed a smooth $1$-form density as indicated in \eqref{E:JRLdp}, i.e., $J_R^{\mathcal L}(\Phi)\in\Omega^1(S,|\Lambda|_S)$.
More precisely, we have
\begin{equation}\label{E:JRpw}
\lambda\bigl(J_R^{\mathcal L}(\Phi)(Z)\bigr)=(\Phi\circ\lambda)(T\varphi\circ Z)
\end{equation}
for $Z\in\mathfrak X(S)$ and $\lambda\in\Gamma^\infty(|\Lambda|_S^*)$.
Note that $J_R^{\mathcal L}(\Phi)$ can also be characterized as the smooth $1$-form density on $S$ corresponding to the $1$-homogeneous vertical $1$-form $\Phi^*\theta^{L^*}$ on the total space of $|\Lambda|_S^*$.
More explicitly, we have
\begin{equation}\label{E:JRpwv}
J_R^{\mathcal L}(\Phi)=\Phi^*\theta^{L^*}
\end{equation}
via the canonical identification
\begin{equation}\label{E:Omega1Lambda}
\Omega^1(S,|\Lambda|_S)=\left\{\beta\in\Omega^1(|\Lambda|_S^*)\middle|\textrm{$\beta$ is vertical and $\bigl(\delta_t^{|\Lambda|_S^*}\bigr)^*\beta=t\beta$ for all $t\in\mathbb R$}\right\}.
\end{equation}
Here $\rho\in\Omega^1(S,|\Lambda|_S)$ corresponds to $\beta\in\Omega^1(|\Lambda|_S^*)$ given by $\beta(W)=w(\rho(T_w\pi^{|\Lambda|_S^*}\cdot W))$ where $w\in|\Lambda|_S^*$ and $W\in T_w|\Lambda|_S^*$.

\begin{remark}\label{R:mutriv2}
Using a volume density $\mu$ on $S$ to identify $\mathcal L\cong C^\infty(S,L^*)$ as in Remark~\ref{R:mutriv}, the differential forms $\theta^{\mathcal L}$ and $\omega^{\mathcal L}$ become, see \eqref{E:thetaL} and \eqref{E:omegaM},
\begin{equation}\label{E:omegaCSL}
\theta^{\mathcal L}(\eta)=\int_S\theta^{L^*}(\eta)\mu
\qquad\text{and}\qquad
\omega^{\mathcal L}(\eta_1,\eta_2)=\int_S\omega^{L^*}(\eta_1,\eta_2)\mu,
\end{equation}
where $\phi\in C^\infty(S,L^*)$ and
$\eta,\eta_1,\eta_2\in T_\phi C^\infty(S,L^*)=\{\eta\in C^\infty(S,TL^*):\pi^{TL^*}\circ\eta=\phi\}$.
For $X\in\mathfrak X(P,\xi)$ and $Z\in\mathfrak X(S)$, the fundamental vector fields $\zeta_X^{\mathcal L}$ and $\zeta_Z^{\mathcal L}$ identify to 
\begin{equation}\label{E:zetaCSM}
\zeta_X^{\mathcal L}(\phi)=\zeta_X^{L^*}\circ\phi
\qquad\text{and}\qquad
	\zeta_Z^{\mathcal L}(\phi)=T\phi\circ Z+\ddiv_\mu(Z)\cdot(R\circ\phi),
\end{equation}
where $\ddiv_\mu(Z):=\frac{L_Z\mu}\mu$ denotes the $\mu$-divergence, and $R:=\frac\partial{\partial t}|_{t=1}\delta_t^{L^*}\in\mathfrak X(L^*)$ denotes the Euler vector field of $L^*$.
The functions $h^{\mathcal L}_X$ and $k^{\mathcal L}_Z$ become, see \eqref{E:hXMint} and \eqref{E:kZMint},
\begin{equation}\label{E:hCSL}
h^{\mathcal L}_X(\phi)=\int_S(\phi^*h_X)\mu
\qquad\text{and}\qquad
k^{\mathcal L}_Z(\phi)=\int_S(\phi^*\theta^{L^*})(Z)\mu.
\end{equation}
Hence, the maps $J_L^{\mathcal L}$ and $J_R^{\mathcal L}$ identify to 
\begin{align}\label{E:JLCSL}
J_L^{\mathcal L}\colon C^\infty(S,L^*)&\to C^\infty(L^*)^*\to\mathfrak X(P,\xi)^*,
&J_L^{\mathcal L}(\phi)&=\phi_*\mu
\\\label{E:JRCSL}
J_R^{\mathcal L}\colon C^\infty(S,L^*)&\to\Omega^1(S,|\Lambda|_S) \subseteq\mathfrak X(S)^*,
&J_R^{\mathcal L}(\phi)&=\phi^*\theta^{L^*}\otimes\mu
\end{align}
where $\phi\in C^\infty(S,L^*)$ and we use the inclusion $\mathfrak X(P,\xi)=C^\infty_\lin(L^*)\subseteq C^\infty(L^*)$, see \eqref{E:XPxi}.
\end{remark}

\subsection{The symplectic part}

Let $\mathcal M\subseteq\mathcal L=C^\infty_\lin(|\Lambda|_S^*,L^*)$ denote the open subset of line bundle homomorphisms $|\Lambda|_S^*\to L^*$ which restrict to a linear isomorphism on each fiber,
\begin{equation}\label{eme}
\mathcal M:=C^\infty_\isom(|\Lambda|_S^*,L^*).
\end{equation}
\footnote{Using a volume density on $S$ to identify $\mathcal L\cong C^\infty(S,L^*)$ as in Remark~\ref{R:mutriv}, the space $\mathcal M$ corresponds to $C^\infty(S,L^*\setminus P)$. When $\xi=\ker\alpha$ for a contact form $\alpha$, then the corresponding trivialization $L^*\cong P\times\mathbb R$ yields a further identification $\mathcal M\cong C^\infty(S,P)\times C^\infty(S,\mathbb R^\times)$.}
We will denote the restriction to $\mathcal M$ of any action, function, form, or vector field on $\mathcal L$ considered above, by replacing the superscript $\mathcal L$ with $\mathcal M$.
Because $L^*\setminus P$ is symplectic, the $2$-form $\omega^{\mathcal M}=d\theta^{\mathcal M}$ is (weakly) non-degenerate, whence symplectic, cf.~\eqref{E:omegaM}.

The map $\pi^{\mathcal M}\colon\mathcal M\to C^\infty(S,P)$ is a principal fiber bundle with structure group $C^\infty(S,\mathbb R^\times)$, provided we restrict to the connected components of $C^\infty(S,P)$ in the image of $\pi^{\mathcal M}$.
If $\varphi$ is in one of these components, then the fiber $\mathcal M_\varphi:=(\pi^{\mathcal M})^{-1}(\varphi)$ may be canonically identified with the space of nowhere vanishing sections of
the line bundle $|\Lambda|_S\otimes\varphi^*L^*$, cf.~\eqref{E:Lphi}.
Thus, disregarding the density part, $\mathcal M_\varphi$ may be considered as the space of contact forms for $\xi$ along the map $\varphi\colon S\to P$.

Clearly, $\mathcal M$ is invariant under the action of the groups $\Diff(P,\xi)$ and $\Diff(S)$.
Since both actions preserve the $1$-form $\theta^{\mathcal M}$, they are Hamiltonian with equivariant moment maps obtained by contraction of the $1$-form with the infinitesimal generators, see for instance \cite[Section~12.3]{MR}.
We summarize these facts in the following proposition.

\begin{proposition}\label{P:momentmaps}
(a)
The action of the group $\Diff(P,\xi)$ on $\mathcal M$ is Hamiltonian with an equivariant moment map $J_L^{\mathcal M}\colon\mathcal M\to\mathfrak X(P,\xi)^*$, given by
\begin{equation}\label{E:jlm}
\langle J_L^{\mathcal M}(\Phi),X\rangle
=(i_{\zeta^{\mathcal M}_X}\theta^{\mathcal M})(\Phi)
=h^{\mathcal M}_X(\Phi)
=\int_S\Phi(X\circ\varphi),
\end{equation}
where $\Phi\in\mathcal M_\varphi$ and $X\in\mathfrak X(P,\xi)$.
Moreover, the moment map $J_L^{\mathcal M}$ is $\Diff(S)$-invariant.
More explicitly, we have $(\Psi_g^{\mathcal M})^*\omega^{\mathcal M}=\omega^{\mathcal M}$, $i_{\zeta^{\mathcal M}_X}\omega^{\mathcal M}=-d\langle J^{\mathcal M}_L,X\rangle$, $\langle J_L^{\mathcal M}\circ\Psi_g^{\mathcal M},X\rangle=\langle\Psi_g^{\mathcal M},g^*X\rangle$, and $J_L^{\mathcal M}\circ\psi^{\mathcal M}_f=J_L^{\mathcal M}$ where $g\in\Diff(P,\xi)$, $X\in\mathfrak X(P,\xi)$, and $f\in\Diff(S)$.

(b)
The action of the group $\Diff(S)$ on $\mathcal M$ is Hamiltonian with an equivariant moment map $J_R^{\mathcal M}\colon\mathcal M\to\Omega^1(S,|\Lambda|_S)\subseteq\mathfrak X(S)^*$, given by
\begin{equation}\label{E:jrm}
\langle J_R^{\mathcal M}(\Phi),Z\rangle
=(i_{\zeta^{\mathcal M}_Z}\theta^{\mathcal M})(\Phi)
=k^{\mathcal M}_Z(\Phi)
=\int_S\Phi(T\varphi\circ Z),
\end{equation}
where $\Phi\in\mathcal M_\varphi$ and $Z\in\mathfrak X(S)$.
Moreover, the moment map $J_R^{\mathcal M}$ is $\Diff(P,\xi)$-invariant.
More explicitly, we have $(\psi_f^{\mathcal M})^*\omega^{\mathcal M}=\omega^{\mathcal M}$, $i_{\zeta^{\mathcal M}_Z}\omega^{\mathcal M}=-d\langle J_R^{\mathcal M},Z\rangle$, $\langle J_R^{\mathcal M}\circ\psi_f^{\mathcal M},Z\rangle=\langle\psi_f^{\mathcal M},f_*Z\rangle$, and $J_R^{\mathcal M}\circ\Psi^{\mathcal M}_g=J_R^{\mathcal M}$, 
where $f\in\Diff(S)$, $Z\in\mathfrak X(S)$, and $g\in\Diff(P,\xi)$.
\end{proposition}

\begin{proof}
The statements in (a) follow immediately from \eqref{E:dhXM}, \eqref{E:JLdef}, \eqref{E:PghXMJ}, and \eqref{E:hXMint}.
The statements in (b) follow immediately from \eqref{E:dkZM}, \eqref{E:JRdef}, \eqref{E:PgkZMJ}, and \eqref{E:kZMint}.
\end{proof}

\begin{remark}
If $S$ is a single point, then we recover the symplectization discussed in Section~\ref{SS:P}.
More precisely, in this case the canonical volume density on $S$ provides a canonical isomorphism between the line bundles $\pi^{\mathcal L}\colon\mathcal L\to C^\infty(S,P)$ and $\pi^{L^*}\colon L^*\to P$.
Up to this identification, we have $\Psi^{\mathcal L}_g=\Psi^{L^*}_g$, for all $g\in\Diff(P,\xi)$, $\theta^{\mathcal L}=\theta^{L^*}$ and $\omega^{\mathcal L}=\omega^{L^*}$.
Moreover, $\mathcal M=M$ and $J_L^{\mathcal M}=J^M$. 
Clearly, the $\Diff(S)$-action is trivial in this case and $J^{\mathcal M}_R=0$.
\end{remark}

\subsection{A dual pair on the non-linear Stiefel manifold of weighted embeddings}\label{SS:E}

We will now restrict to an open subset of $\mathcal M$ on which the $\Diff(S)$-action is free.
Let 
\begin{equation}\label{eee}
\mathcal E:=\Emb_\lin(|\Lambda|_S^*,L^*)
\end{equation}
denote the open subset of all (fiberwise linear) embeddings in $\mathcal L=C^\infty_\lin(|\Lambda|_S^*,L^*)$.
Elements of $\mathcal E$ are automatically isomorphisms on fibers, so $\mathcal E\subseteq\mathcal M$.
We consider $\mathcal E$ as a non-linear Stiefel manifold of weighted embeddings.
\footnote{Using a volume density $\mu$ on $S$ to identify $\mathcal L\cong C^\infty(S,L^*)$ as in Remark~\ref{R:mutriv}, the subset $\mathcal E$ corresponds to $C^\infty(S,L^*\setminus P)\cap(\pi^{\mathcal L})^{-1}(\Emb(S,P))$. If, moreover, $\xi=\ker\alpha$, we get a further identification $\mathcal E\cong\Emb(S,P)\times C^\infty(S,\mathbb R^\times)$.\label{f1}}

We will denote the restriction to $\mathcal E$ of any action, function, form, or vector field on $\mathcal L$ considered above, by replacing the superscript $\mathcal L$ with $\mathcal E$.
The map $\pi^{\mathcal E}\colon\mathcal E\to\Emb(S,P)$ is a principal fiber bundle with structure group $C^\infty(S,\mathbb R^\times)$, provided we restrict to the connected components of $\Emb(S,P)$ in the image of $\pi^{\mathcal E}$.
Since $\mathcal E$ is open in $\mathcal M$, the symplectic form $\omega^{\mathcal M}$ restricts to a symplectic form $\omega^{\mathcal E}$ on $\mathcal E$.
Hence, $(\mathcal E,\omega^{\mathcal E})$ is a (weakly) symplectic Fr\'echet manifold.

Note that $\mathcal E$ is invariant under the actions of $\Diff(P,\xi)$ and $\Diff(S)$.
In view of Proposition~\ref{P:momentmaps}, the restrictions of $J_L^{\mathcal M}$ and $J_R^{\mathcal M}$ to $\mathcal E$ provide equivariant moment maps
\begin{equation}\label{E:JLRE}
\mathfrak X(P,\xi)^*\xleftarrow{\quad J_L^{\mathcal E}\quad}\mathcal E\xrightarrow{\quad J_R^{\mathcal E}\quad}\Omega^1(S,|\Lambda|_S)\subseteq\mathfrak X(S)^*
\end{equation}
for the actions of $\Diff(P,\xi)$ and $\Diff(S)$ on $\mathcal E$, respectively.

A pair of equivariant moment maps for commuting Hamiltonian actions of (infinite dimensional) Lie groups $G$ and $H$ on an (infinite dimensional) symplectic manifold $Q$,
\begin{equation*}
\mathfrak g^*\xleftarrow{\quad J_L\quad}Q\xrightarrow{\quad J_R\quad}\mathfrak h^*,
\end{equation*}
is called a symplectic dual pair \cite{W83} if the distributions $\ker TJ_L$ and $\ker TJ_R$ are symplectic orthogonal complements of one another: $(\ker TJ_L)^\perp=\ker TJ_R$ and $(\ker TJ_R)^\perp=\ker TJ_L$.
Both identities are needed here, due to the weakness of the symplectic form.
Let $\mathfrak g_Q(x):=\{\zeta_X^Q(x)|X\in\mathfrak g\}$ denote the tangent space to the $G$-orbit at $x\in Q$.
When  
\begin{equation}\label{mutcomort}
\mathfrak g_Q=\mathfrak h_Q^\perp
\qquad\text{and}\qquad
\mathfrak h_Q=\mathfrak g_Q^\perp,
\end{equation} 
i.e., if the $G$-orbits and $H$-orbits are symplectic orthogonal complements of one another, then the actions are said to be mutually completely orthogonal \cite{LM}. 
Since $\ker TJ_R=\mathfrak h_Q^\perp$, the first identity in \eqref{mutcomort} can be rephrased as the transitivity of the $\mathfrak g$-action on level sets of the moment map $J_R$, and similarly for the second identity.

Mutually completely orthogonality of the actions implies that $J_L$ and $J_R$ form a dual pair. 
The reverse implication is not always true,
due to the weakness of the symplectic form \cite{GBV15}.

\begin{theorem}\label{T:dp}
The moment mappings $J_L^{\mathcal E}$ and $J_R^{\mathcal E}$ in \eqref{E:JLRE} form a symplectic dual pair, called the EPContact dual pair. Moreover, the commuting actions of $\Diff(P,\xi)$ and $\Diff(S)$ on $\mathcal E$
are mutually completely orthogonal, i.e., for each $\Phi\in\mathcal E$ we have
\begin{multline}\label{E:dp1}
\left\{\zeta^{\mathcal E}_X(\Phi)\middle|X\in\mathfrak X(P,\xi)\right\}
=\left\{\zeta^{\mathcal E}_Z(\Phi)\middle|Z\in\mathfrak X(S)\right\}^\perp
\\:=\left\{A\in T_\Phi\mathcal E\middle|\forall Z\in\mathfrak X(S):\omega^{\mathcal E}_\Phi(A,\zeta^{\mathcal E}_Z(\Phi))=0\right\}
\end{multline}
as well as
\begin{multline}\label{E:dp2}
\left\{\zeta^{\mathcal E}_Z(\Phi)\middle|Z\in\mathfrak X(S)\right\}
=\left\{\zeta^{\mathcal E}_X(\Phi)\middle|X\in\mathfrak X(P,\xi)\right\}^\perp
\\:=\left\{B\in T_\Phi\mathcal E\middle|\forall X\in\mathfrak X(P,\xi):\omega^{\mathcal E}_\Phi(\zeta^{\mathcal E}_X(\Phi),B)=0\right\}.
\end{multline}
\end{theorem}

\begin{proof}
Suppose $\Phi\in\mathcal E$.
The inclusion
$$
\left\{\zeta^{\mathcal E}_X(\Phi)\middle|X\in\mathfrak X(P,\xi)\right\}
\subseteq\left\{\zeta^{\mathcal E}_Z(\Phi)\middle|Z\in\mathfrak X(S)\right\}^\perp
$$
follows immediately from \eqref{E:PghXM} and \eqref{E:dhXM}.
To show the converse inclusion, suppose $A\in\left\{\zeta^{\mathcal E}_Z(\Phi)\middle|Z\in\mathfrak X(S)\right\}^\perp$.
The $1$-form $\beta:=\Phi^*i_A\omega^{L^*}\in\Omega^1(|\Lambda|_S^*)$, given by $\beta(V)=\omega^{L^*}_{\Phi(y)}(A(y),T_y\Phi(V))$ for all $V\in T_y|\Lambda|_S^*$, satisfies $\bigl(\delta_t^{|\Lambda|_S^*}\bigr)^*\beta=t\beta$, by homogeneity of $\Phi$, $A$, and $\omega^{L^*}$.
Thus, for all $Z\in\mathfrak X(S)$,
$$
0=\omega^{\mathcal E}(A,\zeta^{\mathcal E}_Z(\Phi))
\stackrel{\eqref{E:omegaM}}{=}\int_S\omega^{L^*}\bigl(A,T\Phi\circ\zeta_Z^{|\Lambda|_S^*}\bigr){=}\int_S\beta\bigl(\zeta_Z^{|\Lambda|_S^*}\bigr),
$$
where the integrands are fiberwise linear functions on the total space of $|\Lambda|_S^*$, which may be regarded as sections of $|\Lambda|_S$ and integrated over $S$.
By Lemma~\ref{L:div} below, there exists a fiberwise linear function $u\in C^\infty_\lin(|\Lambda|_S^*)$ such that $\beta=du$.

Because $\Phi$ is a fiberwise linear embedding, one can construct  $h\in C^\infty_\lin(L^*)$,
i.e.  $h\circ\delta_t^{L^*}=th$ for all $t\in\mathbb R$, such that $ h\circ\Phi=u$ and $dh\circ\Phi=i_A\omega^{L^*}$.
Indeed, let $\tilde u\in C^\infty_\lin(L^*)$ be any fiberwise linear function with  $\tilde u\circ\Phi=u$ and write $h=\tilde u+h'$.
Hence, it suffices to construct $h'\in C^\infty_\lin(L^*)$ which vanishes along $\Phi$ and has prescribed derivative $dh'\circ\Phi=i_A\omega^{L^*}-(d\tilde u)\circ\Phi$ along $\Phi$.
This is possible, since $\Phi^*(i_A\omega^{L^*})-\Phi^*(d\tilde u)=\beta-du=0$.

According to the identification \eqref{E:XPxi}, there exists a contact vector field $X\in\mathfrak X(P,\xi)$ such that $h=-h_X$, hence
$$i_A\omega^{L^*}=-dh_X\circ\Phi
\stackrel{\eqref{E:dhX}}{=}i_{\zeta_X^{L^*}\circ\Phi}\omega^{L^*}.$$
Since $\omega^{L^*}$ is non-degenerate over $L^*\setminus P$, we conclude $A=\zeta^{L^*}_X\circ\Phi$, and using \eqref{E:zetaXM} we get $A=\zeta^{\mathcal E}_X(\Phi)$, whence \eqref{E:dp1}.

It remains to check the other equality \eqref{E:dp2}.
The inclusion
$$
\left\{\zeta^{\mathcal E}_Z(\Phi)\middle|Z\in\mathfrak X(S)\right\}
\subseteq\left\{\zeta^{\mathcal E}_X(\Phi)\middle|X\in\mathfrak X(P,\xi)\right\}^\perp
$$
follows immediately from \eqref{E:PgkZM} and \eqref{E:dkZM}, or \eqref{E:dp1}.
To show the converse inclusion, suppose that $B\in\left\{\zeta^{\mathcal E}_X(\Phi)\middle|X\in\mathfrak X(P,\xi)\right\}^\perp$.
Hence, for all $X\in\mathfrak X(P,\xi)$,
$$
0=\omega^{\mathcal E}(\zeta^{\mathcal E}_X(\Phi),B)
\stackrel{\eqref{E:zetaXM}}{=}\omega^{\mathcal E}(\zeta^{L^*}_X\circ\Phi,B)
\stackrel{\eqref{E:omegaM}}{=}\int_S\omega^{L^*}(\zeta_X^{L^*}\circ\Phi,B)
\stackrel{\eqref{E:dhX}}{=}-\int_S(dh_X\circ\Phi)(B),
$$
and thus $\int_S(dh\circ\Phi)(B)=0$, for all $h\in C^\infty_\lin(L^*)$, cf.\ \eqref{E:XPxi}.
This implies that $B$ is tangential to $\tilde N:=\Phi(|\Lambda|^*_S)$.
To see this, consider $\gamma\colon|\Lambda|_S^*\to\ann(T\tilde N)\subseteq T^*L^*$ satisfying $\pi^{T^*L^*}\circ\gamma=\Phi$ and $(T\delta_t^{L^*})^*\circ\gamma\circ\delta_t^{|\Lambda|_S^*}=\gamma$ for all $t$.
Since $\Phi$ is a fiberwise linear embedding, there exists $h\in C^\infty_\lin(L^*)$ with $h\circ\Phi=0$ and $\gamma=dh\circ\Phi$, hence $\int_S\gamma(B)=0$ for all such $\gamma$.
We conclude that $B$ is tangential to $\tilde N$.
Consequently, there exists a vector field $W$ on the total space of $|\Lambda|_S^*$ such that $B=T\Phi\circ W$.
Clearly, $\delta_t^*W=W$, for all $t\in\mathbb R$.
Using Lemma~\ref{L:Z} below, we conclude that there exists $Z\in\mathfrak X(S)$ such that $W=\zeta_Z^{|\Lambda|_S^*}$.
In view of \eqref{E:zetaXM}, we obtain $B=\zeta^{\mathcal E}_Z(\Phi)$.
This completes the proof of \eqref{E:dp2}.
\end{proof}

\begin{lemma}\label{L:div}
Suppose $\beta\in\Omega^1(|\Lambda|_S^*)$ is a 1-form on the total space of $|\Lambda|_S^*$, such that $\delta_t^*\beta=t\beta$ for all $t\in\mathbb R$ and 
\begin{equation}\label{E:intSbeta}
\int_S\beta\left(\zeta_Z^{|\Lambda|_S^*}\right)=0
\end{equation}
for all $Z\in\mathfrak X(S)$.
\footnote{Note that the integrand is a fiberwise linear function on the total space of $|\Lambda|_S^*$, which may be regarded as a section of $|\Lambda|_S$ and integrated over $S$.} 
Then $\beta=di_R\beta$ where $R=\frac\partial{\partial t}|_{t=1}\delta_t\in\mathfrak X(|\Lambda|_S^*)$ denotes the radial vector field, i.e., the fundamental vector field of the action $\delta_t$.
\end{lemma}

\begin{proof}
We fix a volume density $\mu$ on $S$ and identify $|\Lambda|_S^*\cong S\times\mathbb R$ correspondingly.
The two canonical projections shall be denoted by $p\colon S\times\mathbb R\to S$ and $t\colon S\times\mathbb R\to\mathbb R$, respectively.
The radial vector field becomes $R=t\partial_t\in\mathfrak X(S\times\mathbb R)$.
By homogeneity, $\beta\in\Omega^1(S\times\mathbb R)$ can be written in the form $\beta=tp^*B+(p^*b)dt$ where $B\in\Omega^1(S)$ and $b\in C^\infty(S,\mathbb R)$.
Moreover, for $Z\in\mathfrak X(S)$, we have 
\begin{equation}\label{E:zetaZL}
\zeta_Z^{|\Lambda|_S^*}=p^*Z+(p^*\ddiv(Z))t\partial_t,
\end{equation} 
where $L_Z\mu=:\ddiv(Z)\mu$ and $p^*Z\in\mathfrak X(S\times\mathbb R)$ denotes the vector field which projects to $Z$ on $S$ and $0$ on $\mathbb R$.
Consequently, 
$$
\beta\left(\zeta_Z^{|\Lambda|_S^*}\right)=tp^*\bigl(i_ZB+b\ddiv(Z)\bigr).
$$
Using Stokes' theorem, we obtain
$$
\int_S\beta\left(\zeta_Z^{|\Lambda|_S^*}\right)
=\int_S\bigl(i_ZB+b\ddiv(Z)\bigr)\mu
=\int_S(B-db)\wedge i_Z\mu.
$$
In view of the assumption \eqref{E:intSbeta}, we conclude that $B=db$, whence $di_R\beta=d(tp^*b)=tp^*db+(p^*b)dt=tp^*B+(p^*b)dt=\beta$, the desired relation.
\end{proof}

\begin{lemma}\label{L:Z}
Suppose $W$ is a vector field on the total space of $|\Lambda|_S^*$, such that $\delta_t^*W=W$ for all $t\in\mathbb R$ and such that 
\begin{equation}\label{E:intSdhZ}
\int_Sdh(W)=0,
\end{equation}
for all smooth, fiberwise linear functions $h$ on the total space of $|\Lambda|_S^*$.
\footnote{Note that the integrand is a fiberwise linear function on the total space of $|\Lambda|_S^*$, which can be regarded as a section of $|\Lambda|_S$ and integrated over $S$.} 
Then $W$ is a fundamental vector field of the natural $\Diff(S)$ action on $|\Lambda|_S^*$, i.e., there exists $Z\in\mathfrak X(S)$ such that $W=\zeta^{|\Lambda|_S^*}_Z$.
\end{lemma}

\begin{proof}
As in the proof of the preceding lemma we fix a volume density $\mu$ on $S$, we identify $|\Lambda|_S^*\cong S\times\mathbb R$ correspondingly, and we denote the two canonical projections  by $p\colon S\times\mathbb R\to S$ and $t\colon S\times\mathbb R\to\mathbb R$.
Hence, the vector field $W$ can be written in the form $W=p^*Z+(p^*w)t\partial_t$ where $Z\in\mathfrak X(S)$ and $w\in C^\infty(S)$.
Every function $\bar h\in C^\infty(S)$ gives rise to a fiberwise linear function $h:=tp^*\bar h$ on the total space of $|\Lambda|_S^*$.
Then 
$$
dh(W)=tp^*(\bar hw+d\bar h(Z))
$$ 
and Stokes' theorem yields
$$
\int_Sdh(W)=\int_S\bigl(\bar h w+d\bar h(Z)\bigr)\mu=\int_S\bar h\bigl(w-\ddiv(Z)\bigr)\mu.
$$
Using the assumption \eqref{E:intSdhZ}, we conclude that $w=\ddiv(Z)$.
Consequently, see \eqref{E:zetaZL}, we obtain $W=p^*Z+(p^*w)t\partial_t=p^*Z+(p^*\ddiv(Z))t\partial_t=\zeta^{|\Lambda|_S^*}_Z$.
\end{proof}

\begin{remark}\label{R:JLRalpha}
Let us give a more explicit description of the EPContact dual pair if the contact structure is described by a contact form, $\xi=\ker\alpha$, and a volume density $\mu$ on $S$ has been fixed.
We have already pointed out before, see footnote~\ref{f1}, that these choices provide an identification of the non-linear Stiefel manifold  $\mathcal E$ with $\Emb(S,P)\times C^\infty(S,\mathbb R^\times)$.
Via this identification, the actions of $\Diff(P,\xi)$ from the left and $\Diff(S)$ from the right are
$$
\Psi^{\mathcal E}_g(\varphi,h)=(g\circ\varphi,(\tfrac{g^*\alpha}\alpha\circ\varphi)h)
\qquad\text{and}\qquad 
\psi^{\mathcal E}_f(\varphi,h)=(\varphi\circ f,(h\circ f)\tfrac{f^*\mu}{\mu}),
$$
where $g\in\Diff(P,\xi)$, $f\in\Diff(S)$, and $(\varphi,h)\in\Emb(S,P)\times C^\infty(S,\mathbb R^\times)$.
Using the identification $\mathfrak X(P,\xi)=C^\infty(P)$ provided by the contact form $\alpha$, the EPContact dual pair \eqref{E:JLRE} becomes
\begin{equation}\label{E:JLREchoices}
C^\infty(P)^*\xleftarrow{\quad J_L^{\mathcal E}\quad}\Emb(S,P)\times C^\infty(S,\mathbb R^\times)\xrightarrow{\quad J_R^{\mathcal E}\quad}\Omega^1(S,|\Lambda|_S)\subseteq\mathfrak X(S)^*
\end{equation}
with moment maps
\begin{equation}\label{E:Jalpha}
J_L^{\mathcal E}(\varphi,h)=\varphi_*(h\mu)
\qquad\text{and}\qquad
J_R^{\mathcal E}(\varphi,h)=\varphi^*\alpha\otimes h\mu.
\end{equation}
This follows readily from the formulas provided in Remarks~\ref{R:mutriv} and \ref{R:mutriv2}.
\end{remark}

In view of Theorem~\ref{T:dp} one might expect \cite{BW,GBV17} that the contact group acts locally transitive on the level sets of $J_R^{\mathcal E}$.
This is indeed the case, see Theorem~\ref{T:Erho} in the subsequent section.
Moreover, one might expect that a coadjoint orbit $\mathcal O\subseteq\mathfrak X(S)^*$ gives rise to a reduced symplectic structure on the quotient $(J_R^{\mathcal E})^{-1}(\mathcal O)/\Diff(S)$ which is equivariantly symplectomorphic to a coadjoint orbit of $\Diff_c(P,\xi)$ via the symplectomorphism induced by the moment map $J_L^{\mathcal E}$.
Below we will see that this can be made rigorous for coadjoint orbits corresponding to isotropic embeddings, see Theorem~\ref{T:Giso}.

\section{Level sets of the right moment map}\label{S:hs}

In this section we will show that each level set of the right moment map 
$$
J_R^{\mathcal E}\colon\mathcal E\to\Omega^1(S,|\Lambda|_S)\subseteq\mathfrak X(S)^*
$$ 
is a smooth splitting Fr\'echet submanifold in $\mathcal E$.
Furthermore, we will see that the contact group acts locally transitive on each level set.
More precisely, we will show that this action admits local smooth sections.
Hence, (unions of) connected components of these level sets may be regarded as homogeneous spaces of the contact group.
These results are summarized in Theorem~\ref{T:Erho} below.

A similar transitivity statement has been established in \cite[Proposition~5.5]{GBV12} using methods quite different from the approach presented here.

Let $\pi^{J^1L}\colon J^1L\to P$ denote the 1-jet bundle of sections of $L$.
Recall that each section $h\in\Gamma^\infty(L)$ gives rise to a section $j^1h\in\Gamma^\infty(J^1L)$.
We equip the total space of $J^1L$ with the contact structure uniquely characterized by the following property: 
A section $s\in\Gamma^\infty(J^1L)$ has isotropic image iff there exists $h\in\Gamma^\infty(L)$ such that $s=j^1h$.
\footnote{If $L\cong P\times\mathbb R$ is a trivialization of $L$, then $J^1L\cong T^*P\times\mathbb R$, and the contact structure can be described by the contact form $p^*\theta-dt$, where $\theta$ denotes the canonical 1-form on $T^*P$,
while $p\colon T^*P\times\mathbb R\to T^*P$ and $t\colon T^*P\times\mathbb R\to\mathbb R$ denote the canonical projections.}
In this case $h=\pi^{J^1L}_L\circ s$, where $\pi^{J^1L}_L\colon J^1L\to L$ denotes the natural projection.

Consider the line bundle $p\colon\hom(p_1^*L,p_2^*L)\to P\times P$ where $p_1,p_2\colon P\times P\to P$ denote the two canonical projections.
We let $\mathcal P:=\isomm(p_1^*L,p_2^*L)$ denote the open subset of fiberwise invertible maps.
We equip the total space of $\mathcal P$ with the contact structure
\begin{equation}\label{E:calPxi}
\xi^{\mathcal P}_a:=\left\{A\in T_a\mathcal P\middle|a\left((T_a(p_1\circ p)A)\textrm{ mod }\xi_{(p_1\circ p)(a)}\right)=(T_a(p_2\circ p)A)\textrm{ mod }\xi_{(p_2\circ p)(a)}\right\}
\end{equation}
where $a\in\mathcal P$.
\footnote{If $\xi=\ker\alpha$, and $\mathcal P\cong P\times P\times(\mathbb R^\times)$ denotes the corresponding trivialization, then the contact structure can be described by the contact form $tp_1^*\alpha-p_2^*\alpha$ on $P\times P\times(\mathbb R^\times)$.}
Note that a diffeomorphism $g\in\Diff(P)$ is contact if and only if there exists a smooth map $a\colon P\to\mathcal P$ with isotropic image satisfying $p_1\circ p\circ a=\id$ and $p_2\circ p\circ a=g$.
Moreover, in this case $\Psi^L_{g,x}=a(x)$ in $\hom(L_x,L_{g(x)})$, for all $x\in P$. 
Here $\Psi^L_{g,x}$ denotes the restriction of $\Psi^L_g$ to the fiber $L_x$.

It is well known \cite[Theorem~1]{L98} that there exists a contact diffeomorphism
\begin{equation}\label{E:Xi}
J^1L\supseteq V\xrightarrow{\,\,\Xi\,\,}U\subseteq\mathcal P
\end{equation}
from an open neighborhood $V$ of the zero section $P\subseteq J^1L$ onto an open neighborhood $U$ of the diagonal $P\subseteq\mathcal P$ intertwining the contact structure obtained by restriction from $J^1L$ with the contact structure obtained by restriction from $\mathcal P$.
Moreover, for all $x\in P$, we have
\begin{equation}\label{E:Xizero}
\Xi(0_x)=\id_{L_x}.
\end{equation}

It is also well known, see \cite[Theorem~43.19]{KM97} for the coorientable case, that the map
\begin{equation}\label{E:chart}
\Gamma_c^\infty(L)\supseteq\mathcal W\xrightarrow{F}\Diff_c(P,\xi),
\qquad F(h):=p_2\circ p\circ\Xi\circ j^1h\circ\bigl(p_1\circ p\circ\Xi\circ j^1h\bigr)^{-1},
\end{equation}
provides a chart for the Lie group $\Diff_c(P,\xi)$ at the identity.
Here $\mathcal W$ is a $C^\infty$-open neighborhood of zero such that, for each $h\in\mathcal W$, the image of $j^1h$ is contained in $V$ and $p_1\circ p\circ\Xi\circ j^1h$ as well as $p_2\circ p\circ\Xi\circ j^1h$ are diffeomorphisms of $P$.
Clearly, $F(0)=\id_P$, see \eqref{E:Xizero}.
Moreover, for $h\in\mathcal W$ and $x\in P$, we have
\begin{equation}\label{E:PsiLFh}
\Psi^L_{F(h),x}=\Bigl(\Xi\circ j^1h\circ\bigl(p_1\circ p\circ\Xi\circ j^1h\bigr)^{-1}\Bigr)(x)
\end{equation}
in $\hom(L_x,L_{F(h)(x)})$.
In particular, 
\begin{equation}\label{E:chartfix}
\textrm{$j^1_xh=0\quad\Leftrightarrow\quad F(h)(x)=x$ and $\Psi^L_{F(h),x}=\id_{L_x}$.}
\end{equation}

\begin{lemma}\label{L:isotropy}
For $\Phi\in\mathcal E$, the isotropy subgroup 
$$
\Diff_c(P,\xi;\Phi)=\{g\in\Diff_c(P,\xi):\Psi_g^{\mathcal E}(\Phi)=\Phi\}
$$ 
is a splitting Lie subgroup of $\Diff_c(P,\xi)$.
\end{lemma}

\begin{proof}
Put $\varphi=\pi^{\mathcal E}(\Phi)\in\Emb(S,P)$ and $N:=\varphi(S)$.
For the chart $F$ in \eqref{E:chart} we obtain
$$
F^{-1}\bigl(\Diff_c(P,\xi;\Phi)\bigr)=\left\{h\in\Gamma_c^\infty(L)\bigm|\forall x\in N:j^1_xh=0\right\}\cap\mathcal W,
$$
see \eqref{E:chartfix} and \eqref{boxaction}.
Since $N$ is a closed submanifold in $P$, the linear space on the right hand side admits a linear complement in $\Gamma^\infty_c(L)$.
To construct such a complement, let $\pi^W\colon W\to N$ denote the normal bundle of $N$, where $W=TP|_N/TN$;
fix a tubular neighborhood $W\subseteq P$ of $N$ such that $N$ corresponds to the zero section in $W$; and
choose an isomorphism of line bundles $L|_W\cong(\pi^W)^*L|_N$.
This provides a linear map 
\begin{equation}\label{E:j1K}
\Gamma^\infty(L|_N)\oplus\Gamma^\infty(L|_N\otimes W^*)\to\Gamma^\infty(L|_W),
\end{equation} 
by regarding sections of $L|_N$ as $\pi^W$-fiberwise constant sections of $L|_W$, and by regarding sections of $L|_N\otimes W^*$ as $\pi^W$-fiberwise linear sections of $L|_W$.
Let $\chi\in C^\infty_c(W,\mathbb R)$ be a compactly supported bump function such that $\chi\equiv1$ in a neighborhood of the zero section.
Multiplication with $\chi$ and extension by zero provides a linear map $\Gamma^\infty(L|_W)\to\Gamma^\infty_c(L)$.
Composing this with \eqref{E:j1K}, we obtain a linear map we will denoted by
\begin{equation}\label{E:j1Kc}	
\chi\colon\Gamma^\infty(L|_N)\oplus\Gamma^\infty\bigl(L|_N\otimes W^*\bigr)\to\Gamma^\infty_c(L).
\end{equation}	
	The image of $\chi$ provides a linear complement of $\left\{h\in\Gamma_c^\infty(L)\bigm|\forall x\in N:j^1_xh=0\right\}$ in $\Gamma^\infty_c(L)$.
Hence, $\Diff_c(P,\xi;\Phi)$ is a splitting Lie subgroup of $\Diff_c(P,\xi)$.
\end{proof}

Suppose $\Phi_1,\Phi_2\in\mathcal M$, and write $\varphi_i=\pi^{\mathcal M}(\Phi_i)\in C^\infty(S,P)$.
For $x\in S$ consider the restrictions to the fibers, $\Phi_{1,x}\colon|\Lambda|_{S,x}^*\to L_{\varphi_1(x)}^*$ and $\Phi_{2,x}\colon|\Lambda|_{S,x}^*\to L_{\varphi_2(x)}^*$, and define a smooth map $G(\Phi_1,\Phi_2)\colon S\to\mathcal P$ by 
\begin{equation}\label{E:G}
G(\Phi_1,\Phi_2)(x)
=(\Phi_{1,x}\circ\Phi_{2,x}^{-1})^*\in\hom(L_{\varphi_1(x)},L_{\varphi_2(x)})
\end{equation}
for $x\in S$.
Clearly,
\begin{equation}\label{E:piG}
p_1\circ p\circ G(\Phi_1,\Phi_2)=\varphi_1\qquad\text{and}\qquad p_2\circ p\circ G(\Phi_1,\Phi_2)=\varphi_2.
\end{equation}

\begin{lemma}\label{L:levels}
The map $G(\Phi_1,\Phi_2)\colon S\to\mathcal P$ has isotropic image iff $J_R^{\mathcal M}(\Phi_1)=J_R^{\mathcal M}(\Phi_2)$.
\end{lemma}

\begin{proof}
Suppose $x\in S$, $Z_x\in T_xS$, $0\neq\lambda_x\in|\Lambda|_{S,x}^*$, and write $a:=G(\Phi_1,\Phi_2)(x)$.
Then:
\begin{align*}
T_xG(&\Phi_1,\Phi_2)\cdot Z_x\in\xi_a^{\mathcal P}
\\&\stackrel{\eqref{E:calPxi}}\Leftrightarrow
a\bigl(T_a(p_1\circ p)T_xG(\Phi_1,\Phi_2)\cdot Z_x\textrm{ mod } \xi_{(p_1\circ p)(a)}\bigr)
\\&\qquad\qquad=T_a(p_2\circ p)T_xG(\Phi_1,\Phi_2)\cdot Z_x\textrm{ mod } \xi_{(p_2\circ p)(a)}
\\&\stackrel{\eqref{E:piG}}\Leftrightarrow
a\bigl(T_x\varphi_1\cdot Z_x\textrm{ mod } \xi_{(p_1\circ p)(a)}\bigr)=T_x\varphi_2\cdot Z_x\textrm{ mod } \xi_{(p_2\circ p)(a)}
\\&\stackrel{\eqref{E:G}}\Leftrightarrow
\Phi_{1,x}^*\bigl(T_x\varphi_1\cdot Z_x\textrm{ mod } \xi_{(p_1\circ p)(a)}\bigr)=\Phi_{2,x}^*\bigl(T_x\varphi_2\cdot Z_x\textrm{ mod } \xi_{(p_2\circ p)(a)}\bigr)
\\&\Leftrightarrow
\lambda_x\bigl(\Phi_{1,x}^*\bigl(T_x\varphi_1\cdot Z_x\textrm{ mod } \xi_{(p_1\circ p)(a)}\bigr)\bigr)=\lambda_x\bigl(\Phi_{2,x}^*\bigl(T_x\varphi_2\cdot Z_x\textrm{ mod } \xi_{(p_2\circ p)(a)}\bigr)\bigr)
\\&\Leftrightarrow
\Phi_{1,x}(\lambda_x)(T_x\varphi_1\cdot Z_x)=\Phi_{2,x}(\lambda_x)(T_x\varphi_2\cdot Z_x)
\\&\stackrel{\eqref{E:JRpw}}\Leftrightarrow
\lambda_x\bigl(J_R^{\mathcal M}(\Phi_1)(Z_x)\bigr)=\lambda_x\bigl(J_R^{\mathcal M}(\Phi_2)(Z_x)\bigr)
\\&\Leftrightarrow
J_R^{\mathcal M}(\Phi_1)(Z_x)=J_R^{\mathcal M}(\Phi_2)(Z_x)
\end{align*}
The lemma follows at once.
\end{proof}

For $\rho\in\Omega^1(S,|\Lambda|_S)$ we let 
$$
\mathcal E^\rho:=(J_R^{\mathcal E})^{-1}(\rho)=\bigl\{\Phi\in\mathcal E:J^{\mathcal E}_R(\Phi)=\rho\bigr\}
$$ 
denote the corresponding level set of the moment map $J_R^{\mathcal E}\colon\mathcal E\to\Omega^1(S,|\Lambda|_S)\subseteq\mathfrak X(S)^*$.

\begin{lemma}\label{L:submf}
The level set $\mathcal E^\rho$ is a smooth splitting Fr\'echet submanifold in $\mathcal E$, for each $\rho\in\Omega^1(S,|\Lambda|_S)$.
\end{lemma}

\begin{proof}
Fix $\Phi_1\in\mathcal E^\rho$, put $\varphi_1:=\pi^{\mathcal E}(\Phi_1)\in\Emb(S,P)$, and consider the submanifold $N:=\varphi_1(S)$ of $P$.
Let $\pi^W\colon W\to N$ denote its normal bundle, $W:=TP|_N/TN$.
Choose a tubular neighborhood $W\subseteq P$ of $N$ in $P$ such that the zero section in $W$ corresponds to $N$.
As in the proof of Lemma~\ref{L:isotropy}, we fix an isomorphism of line bundles, 
\begin{equation}\label{E:Ltriv}
L|_W\cong(\pi^W)^*L|_N
\end{equation}
and a compactly supported bump function $\chi\in C^\infty_c(W,\mathbb R)$ such that $\chi\equiv1$ on an open neighborhood $\mathcal X$ of the zero section in $W$.
The corresponding map \eqref{E:j1Kc} extends uniquely to a linear map $\tilde\chi$ such that the following diagram commutes:
\begin{equation}\label{D:pGs}
\vcenter{\xymatrix{
\Gamma^\infty_c(L)\ar[rr]^-{j^1}&&\Gamma^\infty_c(J^1L)\\
\Gamma^\infty(L|_N)\oplus\Gamma^\infty\bigl(L|_N\otimes W^*\bigr)\ar[rr]^-{j^1\oplus\id}\ar[u]^-\chi&&\Gamma^\infty\bigl(J^1(L|_N)\bigr)\oplus\Gamma^\infty\bigl(L|_N\otimes W^*\bigr)\ar[u]_-{\tilde\chi}
}}
\end{equation}
The line bundle isomorphism in \eqref{E:Ltriv} also provides an isomorphism
\begin{equation}\label{E:ts}
\Gamma^\infty\bigl((J^1L)|_N\bigr)\cong\Gamma^\infty\bigl(J^1(L|_N)\bigr)\oplus\Gamma^\infty\bigl(L|_N\otimes W^*\bigr).
\end{equation}
Using this isomorphism to replace the lower right corner in the diagram~\eqref{D:pGs}, we obtain linear maps $\gamma$ and $\Gamma^\infty\bigl((J^1L)|_N\bigr)\to\Gamma^\infty_c(J^1L)$, $s\mapsto\tilde s$, such that the following diagram commutes:
\begin{equation}\label{D:Gs}
\vcenter{\xymatrix{
\Gamma^\infty_c(L)\ar[rr]^-{j^1}&&\Gamma^\infty_c(J^1L)&\tilde s\\
\Gamma^\infty(L|_N)\oplus\Gamma^\infty\bigl(L|_N\otimes W^*\bigr)\ar[rr]^-\gamma\ar[u]^-\chi&&\Gamma^\infty\bigl((J^1L)|_N\bigr)\ar[u]&s\ar@{|->}[u]
}}
\end{equation}

For every $\nu\in\Gamma^\infty(W)$ with $\nu(N)\subseteq\mathcal X$ we obtain a linear isomorphism
\begin{equation}\label{E:tnts}
\tilde\nu\colon\Gamma^\infty\bigl((J^1L)|_N\bigr)\to\Gamma^\infty\bigl(\nu^*(J^1L)\bigr),\qquad\tilde\nu(s):=\tilde s\circ\nu.
\end{equation}
Moreover, $\tilde\nu$ and its inverse $\tilde\nu^{-1}$ are given by first order differential operators depending smoothly on $\nu$.
Furthermore, if $\nu(N)\subseteq\mathcal X$ and $s\in\Gamma^\infty((J^1L)|_N)$, then
\begin{equation}\label{E:tsiso}
\textrm{$\tilde s\circ\nu$ has isotropic image in $J^1L$}\quad\Leftrightarrow\quad s\in\img(\gamma).
\end{equation}
Also note that $\img(\gamma)$ admits a closed complementary subspace in $\Gamma^\infty((J^1L)|_N)$.
Indeed, the space of smooth sections in the kernel of the canonical projection $J^1(L|_N)\to L|_N$ provides a closed complement for the image of $j^1\colon\Gamma^\infty(L|_N)\to\Gamma^\infty(J^1(L|_N))$.
Taking the sum with $\Gamma^\infty(L|_N\otimes W^*)$ and using \eqref{E:ts}, we obtain a complementary subspace of $\img(\gamma)$ in $\Gamma^\infty((J^1L)|_N)$.

Let $\mathcal V$ denote the $C^\infty$-open neighborhood of zero in $\Gamma^\infty((J^1L)|_N)$ consisting of all $s\in\Gamma^\infty((J^1L)|_N)$ with the following five properties:	
\begin{enumerate}[(a)]
\item\label{I:V1} the image of $\tilde s$ is contained in $V$, cf.\ \eqref{E:Xi},
\item\label{I:V2} $p_1\circ p\circ\Xi\circ\tilde s\colon P\to P$ is a diffeomorphism,
\item\label{I:V3} $p_2\circ p\circ\Xi\circ\tilde s\colon P\to P$ is a diffeomorphism,
\item\label{I:V4} the image of $(p_1\circ p\circ\Xi\circ\tilde s)^{-1}\circ\varphi_1\colon S\to P$ is contained in $\mathcal X\subseteq W$, and 
\item\label{I:V5} $\psi_s:=\pi^W\circ(p_1\circ p\circ\Xi\circ\tilde s)^{-1}\circ\varphi_1\colon S\to N$ is a diffeomorphism.
\end{enumerate}
For $s\in\mathcal V$ we define $\nu_s:=(p_1\circ p\circ\Xi\circ\tilde s)^{-1}\circ\varphi_1\circ\psi_s^{-1}\in\Gamma^\infty(W)$. Hence,
\begin{equation}\label{E:nus}
\nu_s\circ\psi_s=(p_1\circ p\circ\Xi\circ\tilde s)^{-1}\circ\varphi_1.
\end{equation}

We will next show that the following map is a diffeomorphism
\begin{align}\notag
\Gamma^\infty\bigl((J^1L)|_N\bigr)\supseteq\mathcal V&\to\mathcal U\subseteq\bigl\{G\in C^\infty(S,\mathcal P):p_1\circ p\circ G=\varphi_1\bigr\}
\\\label{E:Gs}
s&\mapsto G_s:=\Xi\circ\tilde s\circ\bigl(p_1\circ p\circ\Xi\circ\tilde s)^{-1}\circ\varphi_1
\end{align}
from $\mathcal V$ onto the $C^\infty$-open subset $\mathcal U$ in $\bigl\{G\in C^\infty(S,\mathcal P):p_1\circ p\circ G=\varphi_1\bigr\}$ consisting of all $G\in C^\infty(S,\mathcal P)$ with the following five properties:
\begin{enumerate}[(a)]
\setcounter{enumi}{5}
\item\label{I:U1} $p_1\circ p\circ G=\varphi_1$,
\item\label{I:U2} the image of $G$ is contained in $U$, cf.\ \eqref{E:Xi},
\item\label{I:U3} the image of $\pi^{J^1L}\circ\Xi^{-1}\circ G\colon S\to P$ is contained in $\mathcal X\subseteq W$,
\item\label{I:U4} $\psi_G:=\pi^W\circ\pi^{J^1L}\circ\Xi^{-1}\circ G\colon S\to N$ is a diffeomorphism, and 
\item\label{I:U5} $s_G:=\tilde\nu_G^{-1}\bigl(\Xi^{-1}\circ G\circ\psi^{-1}_G\bigr)\in\mathcal V$, where $\nu_G:=\pi^{J^1L}\circ\Xi^{-1}\circ G\circ\psi_G^{-1}\in\Gamma^\infty(W)$.
\end{enumerate}

To see that \eqref{E:Gs} is a diffeomorphism, let $s\in\mathcal V$ and observe that \eqref{E:nus} and \eqref{E:Gs} yield 
\begin{equation}\label{E:Gss}
G_s=\Xi\circ\tilde s\circ\nu_s\circ\psi_s
\end{equation}
as well as $\psi_{G_s}=\psi_s$ and $\nu_{G_s}=\nu_s$.
Hence, $\Xi^{-1}\circ G_s\circ\psi_{G_s}^{-1}=\tilde s\circ\nu_{G_s}$ and \eqref{E:tnts} gives
\begin{equation}\label{E:srec}
s=\tilde\nu_{G_s}^{-1}\bigl(\Xi^{-1}\circ G_s\circ\psi_{G_s}^{-1}\bigr).
\end{equation}
We conclude that $G_s\in\mathcal U$ and $s_{G_s}=s$, for all $s\in\mathcal V$, see \eqref{I:U5}.
This shows that the map $\mathcal U\to\mathcal V$, $G\mapsto s_G$, is left inverse to the map \eqref{E:Gs}.
To show that it is right inverse too, consider $G\in\mathcal U$ and note that \eqref{E:tnts} and \eqref{I:U5} yield $\tilde s_G\circ\nu_G=\Xi^{-1}\circ G\circ\psi_G^{-1}$.
Hence,
$$
\Xi\circ\tilde s_G\circ\nu_G\circ\psi_G=G.
$$
Composing with $p_1\circ p$ and using \eqref{I:V2}, \eqref{I:U1} we obtain,
$$
\nu_G\circ\psi_G=\bigl(p_1\circ p\circ\Xi\circ\tilde s_G\bigr)^{-1}\circ\varphi_1.
$$
Combining the latter two equations, we get
$$
\Xi\circ\tilde s_G\circ\bigl(p_1\circ p\circ\Xi\circ\tilde s_G\bigr)^{-1}\circ\varphi_1=G.
$$
In other words, $G_{s_G}=G$, for all $G\in\mathcal U$, cf.~\eqref{E:Gs}.
This shows that \eqref{E:Gs} is indeed a diffeomorphism.
Using \eqref{E:tsiso}, \eqref{E:Gss}, and the fact that $\Xi$ is a contact diffeomorphism we find
\begin{equation}\label{E:Gsiso}
\textrm{$G_s$ has isotropic image in $\mathcal P$}\quad\Leftrightarrow\quad s\in\img(\gamma).
\end{equation}

The construction in \eqref{E:G}, cf.\ also \eqref{E:piG}, provides a diffeomorphism
$$
\mathcal M\cong\bigl\{G\in C^\infty(S,\mathcal P):p_1\circ p\circ G=\varphi_1\bigr\},\qquad\Phi_2\mapsto G(\Phi_1,\Phi_2).
$$
Combining this with the diffeomorphism in \eqref{E:Gs}, we see that the map
\begin{equation}\label{E:chartErho}
\Gamma^\infty\bigl((J^1L)|_N\bigr)\supseteq\mathcal V\to\mathcal E,\qquad s\mapsto\Phi_s,
\end{equation}
characterized by $G(\Phi_1,\Phi_s)=G_s$, is a diffeomorphism from $\mathcal V$ onto a $C^\infty$-open neighborhood of $\Phi_1$ in $\mathcal E$.
Combining Lemma~\ref{L:levels} with \eqref{E:Gsiso} and $J_R^{\mathcal E}(\Phi_1)=\rho$, we obtain
\begin{equation}\label{E:Phisiso}
J_R^{\mathcal E}(\Phi_s)=\rho\quad\Leftrightarrow\quad s\in\img(\gamma).
\end{equation}
This shows that \eqref{E:chartErho} is a submanifold chart for $\mathcal E^\rho$ in $\mathcal E$, centered a $\Phi_1$.
\end{proof}

\begin{lemma}\label{L:section}
The action of $\Diff_c(P,\xi)$ on the level set $\mathcal E^\rho$ admits local smooth sections, for each $\rho\in\Omega^1(S,|\Lambda|_S)$.
\end{lemma}

\begin{proof}
We continue to use the notation set up in the proof of Lemma~\ref{L:submf}.
Using the commutativity of the diagram~\eqref{D:Gs} we obtain a linear map $\img(\gamma)\to\Gamma^\infty_c(L)$, $s\mapsto h_s$, such that $j^1h_s=\tilde s$, for all $s\in\img(\gamma)$.
Using \eqref{E:PsiLFh}, \eqref{I:V1}, \eqref{I:V2}, and \eqref{I:V3}, see also \eqref{E:chart}, we find $h_s\in\mathcal W$ and
$$
\Psi^L_{F(h_s),\varphi_1(x)}=\Bigl(\Xi\circ\tilde s\circ\bigl(p_1\circ p\circ\Xi\circ\tilde s\bigr)^{-1}\Bigr)(\varphi_1(x))
$$
in $\hom(L_{\varphi_1(x)},L_{F(h_s)(\varphi_1(x))})$, for all $x\in S$ and $s\in\img(\gamma)\cap\mathcal V$.
Hence, see \eqref{E:Gs} and \eqref{E:chartErho},
$$
\Psi^L_{F(h_s),\varphi_1(x)}=G(\Phi_1,\Phi_s)(x).
$$
Using \eqref{E:G} we obtain
$$
\Phi_{s,x}^*\circ\Psi^L_{F(h_s),\varphi_1(x)}=\Phi_{1,x}^*,
$$
and dualizing yields
$$
\Psi^{L^*}_{F(h_s),\varphi_1(x)}\circ\Phi_{1,x}=\Phi_{s,x}
$$
for all $x\in S$ and $s\in\img(\gamma)\cap\mathcal V$.
Hence, in view of \eqref{boxaction}, we get
$$
\Psi^{\mathcal E}_{F(h_s)}(\Phi_1)=\Phi_s,
$$
for all $s\in\img(\gamma)\cap\mathcal V$.
As \eqref{E:chartErho} restricts to a chart, $\img(\gamma)\cap\mathcal V\to\mathcal E^\rho$, for the manifold $\mathcal E^\rho$, the lemma follows.
\end{proof}

Combining Lemmas~\ref{L:isotropy}, \ref{L:submf}, and \ref{L:section} we obtain the following result:

\begin{theorem}\label{T:Erho}
Suppose $\rho\in\Omega^1(S,|\Lambda|_S)$.
Then the level set $\mathcal E^\rho$ is a smooth splitting Fr\'echet submanifold of $\mathcal E$.
For $\Phi\in\mathcal E^\rho$, the isotropy subgroup $\Diff_c(P,\xi;\Phi)$ is a closed Lie subgroup of $\Diff_c(P,\xi)$.
Moreover, the map provided by the action, $\Diff_c(P,\xi)\to\mathcal E^\rho$, $g\mapsto\Psi^{\mathcal E}_g(\Phi)$, admits a local smooth right inverse defined in a neighborhood of $\Phi$ in $\mathcal E^\rho$.
In particular, the group $\Diff_c(P,\xi)$ acts locally and infinitesimally transitive on $\mathcal E^\rho$, and the $\Diff_c(P,\xi)$-orbit of $\Phi$ is open and closed in $\mathcal E^\rho$.
Denoting this orbit by $\mathcal E^\rho_\Phi$, the map $\Diff_c(P,\xi)\to\mathcal E^\rho_\Phi$ is a smooth principal bundle with structure group $\Diff_c(P,\xi;\Phi)$.
Hence, 
$$
\mathcal E^\rho_\Phi=\Diff_c(P,\xi)/\Diff_c(P,\xi;\Phi)
$$ 
may be regarded as a homogeneous space.
\end{theorem}

\section{Weighted non-linear Grassmannians}\label{S:Gr}

We continue to consider a manifold $P$ endowed with a contact structure $\xi$, and a closed manifold $S$.
Recall that the $\Diff(S)$ action is free on the non-linear Stiefel manifold $\mathcal E$ of weighted embeddings.
We will now factor out this action and consider the corresponding space $\mathcal G=\mathcal E/\Diff(S)$ of unparametrized weighted submanifolds of $P$.

\subsection{Principal bundles over non-linear Grassmannians}

Let $\Gr_S(P)$ denote the non-linear Grassmannian of all smooth submanifolds of $P$ which are diffeomorphic to $S$.
It is well know that $\Gr_S(P)$ can be equipped with the structure of a Fr\'echet manifold such that the canonical map $\Emb(S,P)\to\Gr_S(P)$ becomes a principal bundle with structure group $\Diff(S)$.

Consider the space of weighted submanifolds
\begin{equation}\label{gi}
\mathcal G:=\left\{(N,\gamma)\left|\begin{array}{l}\text{$N\in\Gr_S(P)$ and}\\\text{$\gamma\in\Gamma^\infty(|\Lambda|_N\otimes L|_N^*)$ a nowhere vanishing section}\end{array}\right.\right\}.
\end{equation}
The $\Diff(P,\xi)$-actions on $P$ and on $L^*$ induce a left action on $\mathcal G$.
For $g\in\Diff(P,\xi)$ we let $\Psi^{\mathcal G}_g$ denote the corresponding action on $\mathcal G$, that is, $\Psi_g^{\mathcal G}(N,\gamma)=(g(N),g_*\gamma)$.

\begin{remark}\label{R:Galpha}
If $\xi=\ker\alpha$, then the contact form $\alpha$ provides a trivialization $L^*\cong P\times\mathbb R$ which permits to identify $\mathcal G$ with a weighted non-linear Grassmannian,
\begin{equation}\label{E:Galpha}
\mathcal G\cong\Gr^\wt_S(P):=\left\{(N,\nu)\middle| N\in\Gr_S(P)\text{ and }\nu\in\Gamma^\infty(|\Lambda|_N\setminus N)\right\},
\end{equation}
by identifying $(N,\nu)$ with $(N,\nu\otimes\alpha|_N)\in\mathcal G$.
The weighted Grassmannian can be equipped with a smooth structure such that the canonical forgetful map $\Gr_S^\wt(P)\to\Gr_S(P)$ is a smooth fiber bundle.
Indeed, it can be canonically identified with the bundle associated to the principal fiber bundle $\Emb(S,P)\to\Gr_S(P)$ via the $\Diff(S)$-action on the space $\Gamma^\infty(|\Lambda|_S\setminus S)$ of volume densities on $S$.
Note that the induced smooth structure on $\mathcal G$ does not depend on the contact form $\alpha$ for $\xi$.
Via the identification \eqref{E:Galpha}, the $\Diff(P,\xi)$-action becomes
\begin{equation}\label{E:PsiGalpha} 
\Psi^{\mathcal G}_g(N,\nu)=\left(g(N),\frac{g_*\alpha}\alpha\Big|_{g(N)}g_*\nu\right),
\end{equation}
where $g\in\Diff(P,\xi)$ and $(N,\nu)\in\Gr^\wt_S(P)$.
Indeed, $g_*(\nu\otimes\alpha|_N)=\frac{g_*\alpha}\alpha\big|_{g(N)}g_*\nu\otimes\alpha|_{g(N)}$.
\end{remark}

The space $\mathcal G$ in \eqref{gi} can be equipped with the structure of a smooth manifold such that the canonical forgetful map 
$$
\pi^{\mathcal G}\colon\mathcal G\to\Gr_S(P)
$$ 
becomes a smooth fiber bundle with typical fiber $\Gamma^\infty(|\Lambda|_S\setminus S)$.
Indeed, if $(N,\gamma)\in\mathcal G$, then locally around $N$, the contact structure on $P$ is coorientable and can be described by a contact form.
We can therefore use Remark~\ref{R:Galpha} to equip $\mathcal G$ with a smooth structure.
In view of \eqref{E:PsiGalpha} the $\Diff_c(P,\xi)$-action on $\mathcal G$ is smooth.

To an element $\Phi\in\mathcal E=\Emb_\lin(|\Lambda|_S^*,L^*)$ over the embedding $\varphi=\pi^{\mathcal E}(\Phi)\in\Emb(S,P)$ we associate a pair $(N,\gamma)\in\mathcal G$ in the following way: $N=\varphi(S)$ and $\gamma$ is the composition of $\Phi$ (corestricted to $L^*|_N$) with the isomorphism ${|\Lambda|_\varphi^*}\colon|\Lambda|_N^*\to|\Lambda|_S^*$ induced by the diffeomorphism $\varphi\colon S\to N$.
It is easy to see that the map $q\colon\mathcal E\to\mathcal G$, 
given by $q(\Phi)=(N,\gamma)$, is a smooth principal bundle with structure group $\Diff(S)$.
We summarize this in the following $\Diff(P,\xi)$-equivariant commutative diagram:
\begin{equation}\label{E:EGEmbGr}
\vcenter{\xymatrix{
\mathcal E\ar[d]_-q\ar[rr]^-{\pi^{\mathcal E}}&&\Emb(S,P)\ar[d]\\
\mathcal G\ar[rr]^-{\pi^{\mathcal G}}&&\Gr_S(P)
}}
\end{equation}

By $\Diff(S)$ invariance, see Proposition~\ref{P:momentmaps}(a), the moment map $J^{\mathcal E}_L$ descends to a smooth map 
\begin{equation}\label{E:JLG}
J_L^{\mathcal G}\colon\mathcal G\to\mathfrak X(P,\xi)^*,\qquad J_L^{\mathcal G}\circ q=J_L^{\mathcal E}.
\end{equation}
In view of \eqref{E:jlm} we have the explicit formula
\begin{equation}\label{jlg}
\langle J_L^{\mathcal G}(N,\gamma),X\rangle=\int_N\gamma(X|_N),
\end{equation}
where $(N,\gamma)\in\mathcal G$ and $X\in\mathfrak X(P,\xi)$.
On the right hand side $X$ is regarded as a section of $L$, see \eqref{E:XPxi}, 
restricted to $N$ and contracted with $\gamma$ to produce a density on $N$ which can be integrated.
\footnote{Using a contact form $\alpha$ to identify $\mathcal G\cong\Gr^\wt_S(P)$ as in Remark \ref{R:Galpha}, the map \eqref{jlg} is simply
$$
\langle J_L^{\mathcal G}(N,\nu),X\rangle=\int_N\alpha(X)|_N\nu.
$$
}

\begin{proposition}\label{jleg}
The following assertions hold true:
\begin{enumerate}[(a)]
\item The map $J_L^{\mathcal G}\colon\mathcal G\to\mathfrak X(P,\xi)^*$ is a $\Diff(P,\xi)$-equivariant injective immersion.
\item We have $\Diff(P,\xi;(N,\gamma))=\Diff(P,\xi;J_L^{\mathcal G}(N,\gamma))$, where the left hand side denotes the isotropy group of $(N,\gamma)\in\mathcal G$ and the right hand side denotes the isotropy group of $J_L^{\mathcal G}(N,\gamma)\in\mathfrak X(P,\xi)^*$ for the coadjoint action.
\item The group $\Diff(S)$ acts freely and transitively on level sets of $J_L^{\mathcal E}:\mathcal E\to\mathfrak X(P,\xi)^*$.
\end{enumerate} 
\end{proposition}

\begin{proof}
In view of Proposition~\ref{P:momentmaps}(a), the smooth map $J_L^{\mathcal G}$ is $\Diff(P,\xi)$-equivariant.
It follows from the dual pair symplectic orthogonality condition \eqref{E:dp2} that $J_L^{\mathcal G}$ is immersive.
To check injectivity, suppose $(N_1,\gamma_1)$ and $(N_2,\gamma_2)$ are two elements in $\mathcal G$ such that $J_L^{\mathcal G}(N_1,\gamma_1)=J_L^{\mathcal G}(N_2,\gamma_2)$.
Since $\gamma_i$ is nowhere vanishing, we have $\supp(J_L^{\mathcal G}(N_i,\gamma_i))=N_i$, see \eqref{jlg}, whence $N_1=N_2$.
Assume, for the sake of contradiction, $\gamma_1\neq\gamma_2$.
Then there exists $\bar X\in\Gamma^\infty(L|_N)$ such that $\langle\gamma_1,\bar X\rangle\neq\langle\gamma_2,\bar X\rangle$ with respect to the canonical pairing between $\Gamma^\infty(|\Lambda|_N\otimes L|_N^*)$ and $\Gamma^\infty(L|_N)$.
Extending $\bar X$ to a global section $X\in\Gamma^\infty(L)$, we obtain $\langle J_L^{\mathcal G}(N_1,\gamma_1),X\rangle\neq\langle J_L^{\mathcal G}(N_2,\gamma_2),X\rangle$ using \eqref{jlg}.
Since this contradicts our assumption $J_L^{\mathcal G}(N_1,\gamma_1)=J_L^{\mathcal G}(N_2,\gamma_2)$, we must have $\gamma_1=\gamma_2$.
This shows that $J_L^{\mathcal G}$ is injective.

The assertion about the isotropy groups in (b) follows readily from the injectivity and equivariance of $J_L^{\mathcal G}$. 
The assertion in (c) also follows from the injectivity statement in (a), since the $\Diff(S)$-action on the fibers of $q\colon\mathcal E\to\mathcal G$ is free and transitive.
\end{proof}

\subsection{Right leg symplectic reduction}\label{SS:red}

In this section we study the spaces obtained by symplectic reduction for the right moment map $J_R^{\mathcal E}\colon\mathcal E\to\Omega^1(S,|\Lambda|_S)\subseteq\mathfrak X(S)^*$.
For a $1$-form density $\rho\in\Omega^1(S,|\Lambda|_S)$ we put
$$
\mathcal G^\rho:=q(\mathcal E^\rho),
$$
where $\mathcal E^\rho=(J_R^{\mathcal E})^{-1}(\rho)$.
By $\Diff(S)$-equivariance of $J_R^{\mathcal E}$, and since $\Diff(S)$ acts transitively on the fibers of $q\colon\mathcal E\to\mathcal G$, the definition of $\mathcal G^\rho$ may be rephrased equivalently as
\begin{equation}\label{E:qGrho}
q^{-1}(\mathcal G^\rho)=\mathcal E^\rho\cdot\Diff(S)=(J_R^{\mathcal E})^{-1}(\rho\cdot\Diff(S)).
\end{equation}
Here $\rho\cdot\Diff(S)\subseteq\Omega^1(S,|\Lambda|_S)\subseteq\mathfrak X(S)^*$ denotes the coadjoint orbit through $\rho$.
Note that $q$ induces a bijection
\begin{equation}\label{E:Grho}
\mathcal G^\rho=(J_R^{\mathcal E})^{-1}\bigl(\rho\cdot\Diff(S)\bigr)/\Diff(S)=\mathcal E^\rho/\Diff(S,\rho),
\end{equation}
where $\Diff(S,\rho)=\{f\in\Diff(S):f^*\rho=\rho\}$ denotes the isotropy group of $\rho$.
Thus, $\mathcal G^\rho$ is the underlying set of the symplectically reduced space at $\rho$.

We have the following more explicit description of $\mathcal G^\rho$:

\begin{lemma}\label{L:Grho}
For each $\rho\in\Omega^1(S,|\Lambda|_S)$ we have
$$
\mathcal G^\rho=\left\{(N,\gamma)\in\mathcal G\middle|(N,\iota_N^*\gamma)\cong(S,\rho)\right\}.
$$
Here $\iota_N\colon N\to P$ denotes the inclusion and the pull back $\iota_N^*\gamma\in\Omega^1(N,|\Lambda|_N)=\Gamma^\infty(|\Lambda|_N\otimes T^*N)$ is defined as the composition $|\Lambda|_N^*\xrightarrow\gamma L|_N^*\subseteq T^*P|_N\xrightarrow{T^*\iota_N}T^*N$.
\footnote{Because $\gamma$ is nowhere vanishing, the kernel of $\iota_N^*\gamma:TN\to|\Lambda|_N$ coincides with $\xi|_N\cap TN$.}
\end{lemma}

\begin{proof}
Consider $\Phi\in\mathcal E$ over $\varphi:=\pi^{\mathcal E}(\Phi)\in\Emb(S,P)$ and put $(N,\gamma):=q(\Phi)$.
By definition of $q$, we have $\varphi(S)=N$ and the ``triangle'' on the top of the following diagram commutes:
$$
\xymatrix{
|\Lambda|_S^*\ar[d]_-{J_R^{\mathcal E}(\Phi)}\ar[rr]_-\Phi&&L|^*_N\ar[d]&&|\Lambda|_N^*\ar[d]^-{\iota_N^*\gamma}\ar[ll]^-\gamma\ar@/_3ex/[llll]_-{|\Lambda|_\varphi^*}\\
T^*S&&T^*P|_N\ar[ll]_-{T^*\varphi}\ar[rr]^-{T^*\iota_N}&&T^*N\ar@/^3ex/[llll]^-{T^*\varphi}
}
$$
The left rectangle in this diagram commutes in view of the formula for $J_R^{\mathcal E}$ in \eqref{E:jrm};
the right rectangle commutes in view of the definition of $\iota_N^*\gamma$;
and the ``triangle'' at the bottom commutes trivially.
We conclude that $(N,\iota_N^*\gamma)\cong(S,J_R^{\mathcal E}(\Phi))$ via $\varphi$.
Hence, $q(\Phi)\cong(S,\rho)$ iff $(S,J_R^{\mathcal E}(\Phi))\cong(S,\rho)$.
The latter, in turn, holds iff there exists $f\in\Diff(S)$ with $J_R^{\mathcal E}(\Phi)=f^*\rho$, i.e., iff $\Phi\in(J_R^{\mathcal E})^{-1}(\rho\cdot\Diff(S))$.
Using the description \eqref{E:qGrho} of $\mathcal G^\rho$ we obtain the lemma.
\end{proof}

\begin{remark}\label{R:Grho}
We have seen in Remark~\ref{R:Galpha} that the choice of a contact form $\alpha$ on $P$ permits to identify $\mathcal G$ with a weighted Grassmannian.
Under this identification, the reduced space becomes
\begin{equation}
\mathcal G^\rho
\cong\left\{(N,\nu)\in\Gr^\wt_S(P):(N,\iota_N^*\alpha\otimes\nu)\cong(S,\rho)\right\}.
\end{equation}
\end{remark}

\begin{remark}
A general fiber of the forgetful map $\pi^{\mathcal G}\colon\mathcal G\to\Gr_S(P)$ will intersect several of the spaces $\mathcal G^\rho$, for many different $\rho$.
A notable exception are fibers over isotropic submanifolds, cf.\ \eqref{E:Giso} in Section~\ref{SS:Giso} below.
\end{remark}

Since we do not expect $\mathcal G^\rho$ to be a submanifold in $\mathcal G$ for general $\rho$, we will consider $\mathcal G^\rho$ as a Fr\"olicher space with the smooth structure induced from the ambient Fr\'echet manifold $\mathcal G$.

Recall that a \emph{Fr\"olicher space,} see \cite[Section~23]{KM97} and \cite{F80,F81,FK88}, is a set $X$ together with a set $\mathcal C_X$ of curves into $X$ and a set of functions $\mathcal F_X$ on $X$ with the following two properties:
\begin{enumerate}[(a)]
\item A function $f\colon X\to\mathbb R$ is in $\mathcal F_X$ if and only if $f\circ c\in C^\infty(\mathbb R,\mathbb R)$ for all $c\in\mathcal C_X$.
\item A curve $c\colon\mathbb R\to X$ is in $\mathcal C_X$ if and only if $f\circ c\in C^\infty(\mathbb R,\mathbb R)$ for all $f\in\mathcal F_X$.
\end{enumerate}
A map $g\colon X\to Y$ between Fr\"olicher spaces is called \emph{smooth} if $g\circ c\in\mathcal C_Y$ for all $c\in\mathcal C_X$.
Equivalently, smoothness of $g$ can be characterized by $f\circ g\in\mathcal F_X$ for all $f\in\mathcal F_Y$.
Note that $\mathcal C_X$ coincides with the set of smooth curves into $X$, and $\mathcal F_X$ coincides with the set of smooth functions on $X$, provided $\mathbb R$ is equipped with the standard Fr\"olicher structure $\mathcal C_{\mathbb R}=C^\infty(\mathbb R,\mathbb R)=\mathcal F_{\mathbb R}$.
Fr\"olicher spaces and smooth maps between them form a category which is complete, cocomplete and Cartesian closed, see \cite[Theorem~23.2]{KM97}.

Any subset $A$ of a Fr\"olicher space $X$ admits a unique Fr\"olicher structure such that the inclusion $A\subseteq X$ is initial, i.e., a curve into $A$ is smooth iff it is smooth into $X$.
The \emph{$c^\infty$-topology} on a Fr\"olicher space $X$ is the strongest topology such that all smooth curves into $X$ are continuous.
If $\mathcal U$ is a cover of $X$ by $c^\infty$-open subsets, then a function $f$ on $X$ is smooth iff the restriction $f|_U$ is smooth (with respect to the induced Fr\"olicher structure) for all $U\in\mathcal U$.

Any Fr\'echet manifold, together with the usual smooth curves into and smooth functions on it, constitutes a Fr\"olicher space.
More generally, manifolds modeled on $c^\infty$-open subsets of convenient vector spaces \cite[Section~27]{KM97} are Fr\"olicher spaces.
For Fr\'echet manifolds the $c^\infty$-topology coincides with the Fr\'echet topology, see \cite[Theorem~4.11(1)]{KM97}.

We consider $\mathcal G^\rho$ as a Fr\"olicher space with the smooth structure induced from $\mathcal G$.
Hence, a curve in $\mathcal G^\rho$ is smooth iff it is smooth into $\mathcal G$; and a function on $\mathcal G^\rho$ is smooth iff it is smooth along smooth curves.
Moreover, we equip $\mathcal E^\rho/\Diff(S,\rho)$ with the induced Fr\"olicher structure.
Hence, a function on $\mathcal E^\rho/\Diff(S,\rho)$ is smooth iff the corresponding (fiberwise constant) function on $\mathcal E^\rho$ is smooth, with respect to the Fr\"olicher structure on $\mathcal E^\rho$ considered before; and a curve in $\mathcal E^\rho/\Diff(S,\rho)$ is smooth iff it is smooth along smooth functions.
One readily checks that the maps in the commutative diagram
$$
\xymatrix{
\mathcal E^\rho\ar[d]\ar@/^/[dr]^-{q^\rho}\\
\mathcal E^\rho/\Diff(S,\rho)\ar[r]&\mathcal G^\rho
}
$$
are all smooth in the sense of Fr\"olicher spaces, where $q^\rho$ denotes the restriction of $q$.
If smooth curves in $\mathcal G^\rho$ can be lifted to smooth curves in $\mathcal E^\rho$, then the horizontal map provides a diffeomorphism of Fr\"olicher spaces,
$\mathcal E^\rho/\Diff(S,\rho)=\mathcal G^\rho$.

Any subgroup of $\Diff_c(P)$ or $\Diff(S)$ inherits a Fr\"olicher structure from the ambient Lie group, and the group operations are smooth with respect to this Fr\"olicher structure.

\begin{proposition}\label{P:Grho}
If $q^\rho\colon\mathcal E^\rho\to\mathcal G^\rho$ admits local (with respect to the $c^\infty$-topology) smooth sections in the sense of Fr\"olicher spaces, then the following hold true:
\begin{enumerate}[(a)]
\item The map $q^\rho\colon\mathcal E^\rho\to\mathcal G^\rho$ is a locally trivial smooth principal bundle with structure group $\Diff(S,\rho)$ in the sense of Fr\"olicher spaces.
Moreover, the canonical identification $\mathcal E^\rho/\Diff(S,\rho)=\mathcal G^\rho$ is a diffeomorphism of Fr\"olicher spaces.
\item The $\Diff_c(P,\xi)$-action on $\mathcal G^\rho$ admits local (with respect to the $c^\infty$-topology) smooth sections in the sense of Fr\"olicher spaces.
The $\Diff_c(P,\xi)$-orbit through $(N,\gamma)\in\mathcal G^\rho$ is open and closed (with respect to the $c^\infty$-topology) in $\mathcal G^\rho$. 
Denoting this orbit by $\mathcal G^\rho_{(N,\gamma)}$, the map $\Diff_c(P,\xi)\to\mathcal G^\rho_{(N,\gamma)}$ provided by the action is a locally trivial smooth principal bundle with structure group $\Diff_c(P,\xi;(N,\gamma))$ in the sense of Fr\"olicher spaces.
Hence, 
$$
\mathcal G^\rho_{(N,\gamma)}=\Diff_c(P,\xi)/\Diff_c(P,\xi;(N,\gamma))
$$
may be regarded as a homogeneous Fr\"olicher space.
\item The map $J_L^{\mathcal G}$ restricts to a $\Diff_c(P,\xi)$-equivariant smooth bijection 
\begin{equation}\label{E:JLGrho}
J_L^{\mathcal G^\rho}\colon\mathcal G^\rho_{(N,\gamma)}\to\mathfrak X(P,\xi)^*
\end{equation} 
onto the coadjoint orbit of $\Diff_c(P,\xi)$ through $J_L^{\mathcal G}(N,\gamma)$.
\end{enumerate}
\end{proposition}

\begin{proof}
Let $\sigma\colon U\to\mathcal E^\rho$ be a local smooth section of $q^\rho\colon\mathcal E^\rho\to\mathcal G^\rho$, defined on a $c^\infty$-open subset $U$ in $\mathcal G^\rho$.
Putting $\mathcal E^\rho_U:=(q^\rho)^{-1}(U)$, we obtain a local trivialization, 
\begin{equation}\label{E:trivialization}
U\times\Diff(S,\rho)\to\mathcal E^\rho_U,\qquad (z,f)\mapsto\psi^{\mathcal E}_f(\sigma(z)).
\end{equation}
Clearly, this is a $\Diff(S,\rho)$-equivariant smooth bijection.
To see that these are actually diffeomorphisms of Fr\"olicher spaces, we use the fact that $\mathcal E\to\mathcal G$ is a smooth principal bundle.
This implies that the map $\tilde\delta\colon\mathcal E\times_{\mathcal G}\mathcal E\to\Diff(S)$, implicitly characterized by $\psi^{\mathcal E}_{\tilde\delta(\Phi_1,\Phi_2)}(\Phi_1)=\Phi_2$ for all $\Phi_1,\Phi_2\in\mathcal E$ with $q(\Phi_1)=q(\Phi_2)\in\mathcal G$, is smooth.
Restricting $\tilde\delta$, we obtain a smooth map $\delta\colon\mathcal E^\rho\times_{\mathcal G^\rho}\mathcal E^\rho\to\Diff(S,\rho)$, which can be used to express the inverse of \eqref{E:trivialization}: $\mathcal E^\rho_U\to U\times\Diff(S,\rho)$, $\Phi\mapsto(q(\Phi),\delta(\sigma(q(\Phi)),\Phi))$.
This shows that the trivialization \eqref{E:trivialization} is a diffeomorphism, whence $\mathcal E^\rho\to\mathcal G^\rho$ is a locally trivial smooth principal fiber bundle.
The remaining assertions in (a) are now obvious.

Using local sections of $\mathcal E^\rho\to\mathcal G^\rho$ and the fact that the $\Diff_c(P,\xi)$-action on $\mathcal E^\rho$ admits local smooth sections, see Theorem~\ref{T:Erho}, we readily see that the $\Diff_c(P,\xi)$-action on $\mathcal G^\rho$ admits local smooth sections.
The remaining assertions in (b) are then obvious.

Part (c) follows at once, see Proposition~\ref{jleg}(a).
\end{proof}

Note that the assumption in Proposition~\ref{P:Grho} is trivially satisfied for $\rho=0$.
This isotropic case will be discussed in Section~\ref{SS:Giso}; and we will obtain a more precise conclusion than formulated in Proposition~\ref{P:Grho} above, see Theorem~\ref{T:Giso}.
In particular, we will show that in this case $\mathcal G^0$ is a smooth submanifold of $\mathcal G$ which inherits a reduced symplectic form from $\mathcal E^0$. 
Moreover, the map in \eqref{E:JLGrho} is a symplectomorphism onto the coadjoint orbit equipped with the Kostant--Kirillov--Souriau form.

For more general $\rho$ (e.g.\ $\rho$ of contact type) the situation is more delicate.
If we equip $\mathcal G^\rho$ with the trace topology induced from $\mathcal G$, then $q$ restricts to principal fiber bundle $(J_R^{\mathcal E})^{-1}(\rho\cdot\Diff(S))/\Diff(S)\to\mathcal G^\rho$ with structure group $\Diff(S)$, see \eqref{E:Grho}.
However, with respect to this topology, the action of $\Diff(P,\xi)$ on $\mathcal G^\rho$ will in general not admit local sections, see Proposition~\ref{P:contactGrho} below.

The next lemma provides a criterion for the premise in Proposition~\ref{P:Grho} above.

\begin{lemma}\label{L:Grhosec}
Let $\rho\in\Omega^1(S,|\Lambda|_S)$ be a $1$-form density and assume that the $\Diff(S)$-action on the orbit through $\rho$ in $\Omega^1(S,|\Lambda|_S)$ admits local (with respect to the $c^\infty$-topology) smooth sections in the sense of Fr\"olicher spaces.
Then the map $\mathcal E^\rho\to\mathcal G^\rho$ admits local (with respect to the $c^\infty$-topology) smooth sections in the sense of Fr\"olicher spaces.
\end{lemma}

\begin{proof}
Since $\mathcal E\to\mathcal G$ admits local sections, each point in $\mathcal G^\rho$ admits an open neighborhood $\tilde U$ in $\mathcal G$ and a smooth section $\tilde\sigma\colon\tilde U\to\mathcal E$ such that $q\circ\tilde\sigma=\id_{\tilde U}$.
Then $U:=\tilde U\cap\mathcal G^\rho$ is a $c^\infty$-open neighborhood, and the restriction $\bar\sigma:=\tilde\sigma|_U$ is a smooth section mapping $\bar\sigma\colon U\to(J^{\mathcal E}_R)^{-1}(\rho\cdot\Diff(S))$, see \eqref{E:qGrho}.
By assumption, after possibly shrinking $U$, there exists a smooth map $f\colon U\to\Diff(S)$ such that $J_R^{\mathcal E}(\bar\sigma(z))=f(z)^*\rho$ for all $z\in U$.
Hence, $\sigma\colon U\to\mathcal E^\rho$, $\sigma(z):=\psi^{\mathcal E}_{f(z)^{-1}}(\bar\sigma(z))$, is the desired local smooth section of $\mathcal E^\rho\to\mathcal G^\rho$.
\end{proof}

\begin{remark}
For a contact $1$-form density, the Gray stability theorem \cite[Theorem~2.2.2]{G08} permits to reformulate the assumption in Lemma~\ref{L:Grhosec}.
More precisely, if $\rho\in\Omega^1(S,|\Lambda|_S)$ is a $1$-form density such that $\ker\rho$ is a contact distribution on $S$, then the following two statements are equivalent:
\begin{enumerate}[(a)]
\item The $\Diff(S)$-action on the $\Diff(S)$-orbit through $\rho$ in $\Omega^1(S,|\Lambda|_S)$ admits local (with respect to the $c^\infty$-topology) smooth sections in the sense of Fr\"olicher spaces.
\item The $\Diff(S,\ker\rho)$-action on the $\Diff(S,\ker\rho)$-orbit through $\rho$ in $\Omega^1(S,|\Lambda|_S)$ admits local (with respect to the $c^\infty$-topology) smooth sections in the sense of Fr\"olicher spaces. 
\end{enumerate}
We do not know if these (equivalent) properties hold true for all contact $1$-form densities.
\end{remark}

\subsection{Weighted isotropic non-linear Grassmannians}\label{SS:Giso}

We will now specialize to the isotropic case, $\rho=0$.
Let us introduce the notation 
\begin{equation}\label{E:Fiso}
\mathcal E^\iso:=
(\pi^{\mathcal E})^{-1}(\Emb^\iso(S,P))=(J_R^{\mathcal E})^{-1}(0)=\mathcal E^0,
\end{equation}
where $\Emb^\iso(S,P)$ denotes the space of isotropic embeddings, cf.~\eqref{E:JRpwv}, \eqref{E:JRCSL}, or \eqref{E:Jalpha}.
This can equivalently be characterized as the elements in $\mathcal E=\Emb_\lin(|\Lambda|_S^*,L^*)$ which restrict to isotropic embeddings $|\Lambda|_S^*\setminus S\to L^*\setminus P=M$.
\footnote{Using a volume density $\mu$ on $S$ to identify $\mathcal L\cong C^\infty(S,L^*)$ as in Remark~\ref{R:mutriv}, the subset $\mathcal E^\iso$ corresponds to $C^\infty(S,M)\cap(\pi^{\mathcal L})^{-1}(\Emb^\iso(S,P))$.
If moreover $\xi=\ker\alpha$, then the corresponding diffeomorphism $\mathcal L\cong C^\infty(S,P)\times C^\infty(S)$ provides an identification $\mathcal E^\iso\cong\Emb^\iso(S,P)\times C^\infty(S,\mathbb R^\times)$.}

Let $\Gr_S^\iso(P)$ denote the space of isotropic submanifolds of type $S$ and consider the space of all weighted isotropic submanifolds of type $S$,
\begin{align}\label{E:Gisodef}
\mathcal G^\iso&:=(\pi^{\mathcal G})^{-1}(\Gr^\iso_S(P))\\
&=\{(N,\gamma)|N\in\Gr_S^\iso(P),\gamma\in\Gamma^\infty(|\Lambda|_N\otimes L^*|_N)\text{ nowhere vanishing}\}.\nonumber
\end{align}
In view of \eqref{E:EGEmbGr} and \eqref{E:Fiso} we have $q^{-1}(\mathcal G^\iso)=\mathcal E^\iso=(J^{\mathcal E}_R)^{-1}(0)$.
Hence, $\mathcal G^\iso$ coincides with the reduced space $\mathcal G^\rho$ for $\rho=0$, i.e.,
\begin{equation}\label{E:Giso}
\mathcal G^0=(J_R^{\mathcal E})^{-1}(0)/\Diff(S)=\mathcal G^\iso=(\pi^{\mathcal G})^{-1}(\Gr^\iso_S(P)).
\end{equation}

\begin{remark}
If $\alpha$ is a contact form for $\xi$, then isotropic submanifolds $N$ are characterized by $\iota_N^*\alpha=0$ and the identification in Remark~\ref{R:Grho} becomes
\begin{equation}\label{gis}
\mathcal G^0=\mathcal G^\iso\cong\bigl\{(N,\nu):\textrm{$N\in\Gr_S^\iso(P)$ and $\nu\in\Gamma^\infty(|\Lambda|_N\setminus N)$}\bigr\}.
\end{equation}
\end{remark}

\begin{lemma}\label{L:GrSisosubmf}
The subset $\Gr_S^\iso(P)$ is a smooth splitting submanifold of $\Gr_S(P)$.
\end{lemma}

\begin{proof}
This follows from the tubular neighborhood theorem for contact structures near isotropic submanifolds, see \cite[Theorem~2.5.8]{G08} or \cite[Theorem~1]{L98}.
Since we were not able to locate this statement in the literature, we will sketch a proof in the subsequent paragraph.

Suppose $S\cong N\subseteq P$ is an isotropic submanifold, and let $E:=TN^\perp/TN$ denote its conformal symplectic normal bundle, see \cite[Definition~2.5.3]{G08}.
Using the relative Poincar\'e lemma, one easily constructs a $1$-form $\varepsilon$ on the total space of $E$ such that 
(1) $\varepsilon$ vanishes along the zero section; 
(2) $i_Xd\varepsilon=0$ for every vector $X$ tangent to the zero section; and 
(3) such that $(d\varepsilon)|_N$ represents the conformal symplectic structure on each fiber of $E$, cf.\ the proof of \cite[Proposition in Section~4]{L98}.
Hence $\alpha:=p_1^*\varepsilon+p_2^*\theta+dt$ is a contact form in a neighborhood of the zero section of $E\oplus T^*N\times\mathbb R$, where $p_1$, $p_2$, $t$ denote the canonical projections onto the three summands, and $\theta$ denotes the canonical $1$-form on $T^*N$.
Assuming, for simplicity, that the contact structure on $P$ is coorientable near $N$, the tubular neighborhood theorem for isotropic submanifolds asserts that there exists a contact diffeomorphism $\psi$ between an open neighborhood of the zero section in $E\oplus T^*N\times\mathbb R$ and an open neighborhood of $N$ in $P$ which restricts to the identity along $N$.
Using this diffeomorphism, we obtain a manifold chart for $\Gr_S(P)$ centered at $N$ by assigning to a smooth section $\sigma$ of $E\oplus T^*N\times\mathbb R$, which is sufficiently $C^1$-close to the zero section, the submanifold $\psi(\sigma(N))$ in $P$.
As $\psi$ is contact, the part of $\Gr^\iso_S(P)$ covered by this chart corresponds to sections $\sigma\in\Gamma^\infty(E\oplus T^*N\times\mathbb R)$ such that $\sigma^*\alpha=0$.
Identifying $\Gamma^\infty(E\oplus T^*N\times\mathbb R)=\Gamma^\infty(E)\times\Omega^1(N)\times C^\infty(N)$ and writing $\sigma=(s,\beta,f)$ accordingly, the latter condition is equivalent to $s^*\epsilon+\beta+df=0$.
Hence, $\Gr^\iso_S(P)$ corresponds to the part of the chart domain contained in the splitting linear subspace 
\begin{align*}
\Gamma^\infty(E)\times C^\infty(N)&\subseteq\Gamma^\infty(E)\times\Omega^1(N)\times C^\infty(N)=\Gamma^\infty(E\oplus T^*N\times\mathbb R),
\\
(s,f)&\mapsto(s,-s^*\varepsilon-df,f).
\end{align*}
This shows that $\Gr^\iso_S(P)$ is a splitting smooth submanifold of $\Gr_S(P)$.
\end{proof}

\begin{remark}
Lemma~\ref{L:GrSisosubmf} implies that $\Emb^\iso(S,P)$ is a smooth splitting submanifold of $\Emb(S,P)$, because the natural map $\Emb(S,P)\to\Gr_S(P)$ is a (locally trivial) smooth principal bundle with typical fiber $\Diff(S)$.
Since $\pi^{\mathcal E}\colon\mathcal E\to\Emb(S,P)$ is a (locally trivial) smooth fiber bundle, this also implies that $\mathcal E^\iso$ is a smooth submanifold of $\mathcal E$, see \eqref{E:Fiso}.
Using the isotropic isotopy extension theorem for contact manifolds, see \cite[Theorem~2.6.2]{G08} for instance, one can show that the group $\Diff_c(P,\xi)$ acts locally and infinitesimally transitive on $\mathcal E^\iso$.
Hence, for $\rho=0$, Theorem~\ref{T:Erho} is essentially known.
\end{remark}

As mentioned before, one expects that connected components of $\mathcal G^\iso$, endowed with a reduced symplectic form, are symplectomorphic to coadjoint orbits of $\Diff_c(P,\xi)$ via the restriction of $J_L^{\mathcal G}\colon\mathcal G\to\mathfrak X(P,\xi)^*$.
The following theorem makes this precise.

\begin{theorem}\label{T:Giso}
(a) The subset $\mathcal G^\iso$ is a smooth splitting submanifold of $\mathcal G$.
Moreover, the map provided by the action, $\Diff_c(P,\xi)\to\mathcal G^\iso$, $g\mapsto\Psi^{\mathcal G}_g(N,\gamma)$, admits a local smooth right inverse defined in a neighborhood of $(N,\gamma)$ in $\mathcal G^\iso$.
In particular, the group $\Diff_c(P,\xi)$ acts locally and infinitesimally transitive on $\mathcal G^\iso$, and the $\Diff_c(P,\xi)$-orbit of $(N,\gamma)$ is open and closed in $\mathcal G^\iso$.
Denoting this orbit by $\mathcal G^\iso_{(N,\gamma)}$, the map $\Diff_c(P,\xi)\to\mathcal G^\iso_{(N,\gamma)}$ is a smooth principal bundle with structure group $\Diff_c(P,\xi; (N,\gamma))$ in the sense of Fr\"olicher spaces.
Hence, 
$$
\mathcal G^\iso_{(N,\gamma)}=\Diff_c(P,\xi)/\Diff_c(P,\xi;(N,\gamma))
$$ 
may be regarded as a homogeneous space in the sense of Fr\"olicher spaces.

(b) The projection $q$ restricts to a smooth principal bundle $q^\iso\colon\mathcal E^\iso\to\mathcal G^\iso$ with structure group $\Diff(S)$.
The restriction of the symplectic form $\omega^{\mathcal E}$ to $\mathcal E^\iso$ descends to a (reduced) symplectic form $\omega^{\mathcal G^\iso}$ on $\mathcal G^\iso$.
The $\Diff(P,\xi)$-equivariant injective immersion 
$$
J_L^{\mathcal G^\iso}\colon\mathcal G^\iso\to\mathfrak X(P,\xi)^*,\quad
\langle J_L^{\mathcal G^\iso}(N,\gamma),X\rangle=\int_N\gamma(X|_N),
$$ 
provided by restriction of $J_L^{\mathcal G}$ from \eqref{jlg}, identifies $\mathcal G_{(N,\gamma)}^\iso$ with the coadjoint orbit through $J_L^{\mathcal G}(N,\gamma)$ of the contact group $\Diff_c(P,\xi)$, such that  
\begin{equation}\label{E:KKSiso}
(J_L^{\mathcal G^\iso})^*\omega^\KKS=\omega^{\mathcal G^\iso},
\end{equation}
where $\omega^\KKS$ denotes the Kostant--Kirillov--Souriau symplectic form on the coadjoint orbit through $J_L^{\mathcal G}(N,\gamma)$, cf.\ Remark~\ref{R:KKS} below.
\end{theorem}

\begin{remark}\label{R:KKS}
To avoid discussing differential forms on coadjoint orbits, we consider the Kostant--Kirillov--Souriau form on the coadjoint orbit through $J_L^{\mathcal G}(N,\gamma)$ as a formal object only.
We actually work with its pull back along $J^{\mathcal G^\iso}_L$, that is, the well defined smooth $2$-form on $\mathcal G^\iso$ characterized by
\begin{equation}\label{E:JKKS}
((J_L^{\mathcal G^\iso})^*\omega^\KKS)(\zeta^{\mathcal G^\iso}_X(N,\gamma),\zeta^{\mathcal G^\iso}_Y(N,\gamma)):=\langle J^{\mathcal G^\iso}_L(N,\gamma),[X,Y]\rangle,
\end{equation}
where $X,Y\in\mathfrak X(P,\xi)$ and $(N,\gamma)\in\mathcal G^\iso$.
To motivate this definition, recall that for a Lie algebra $\mathfrak g$ the Kostant--Kirillov--Souriau symplectic form on the coadjoint orbit through $\lambda\in\mathfrak g^*$ is (formally) given by
\begin{equation*}\label{kks}
\omega^\KKS(\zeta_X^{\mathfrak g^*}(\lambda),\zeta_Y^{\mathfrak g^*}(\lambda))=\langle\lambda,[X,Y]\rangle,
\end{equation*}
where $X,Y\in\mathfrak g$ and $\zeta_X^{\mathfrak g^*}$ denotes the infinitesimal coadjoint action.
Since $J_L^{\mathcal G^\iso}$ is equivariant, we are being lead to \eqref{E:JKKS}.
\end{remark}

\begin{proof}[Proof of Theorem~\ref{T:Giso}]
We have already observed that $\Gr_S^\iso(P)$ is a smooth submanifold of $\Gr_S(P)$, see Lemma~\ref{L:GrSisosubmf}.
Since the forgetful map $\pi^{\mathcal G}\colon\mathcal G\to\Gr_S(P)$ is a smooth fiber bundle, we conclude that $\mathcal G^\iso$ is a smooth submanifold of $\mathcal G$, see \eqref{E:Gisodef}.
In particular, the map provided by the action $\Diff_c(P,\xi)\to\mathcal G^\iso$, $g\mapsto\Psi^{\mathcal G}_g(N,\gamma)$, is smooth.
The remaining assertions in (a) thus follow from Proposition~\ref{P:Grho}(b).
Note that in the isotropic case the assumption in the latter proposition is trivially satisfied.

In view of $\mathcal E^\iso=q^{-1}(\mathcal G^\iso)$, the smooth principal bundle $q\colon\mathcal E\to\mathcal G$ restricts to a smooth principal bundle $q^\iso\colon\mathcal E^\iso\to\mathcal G^\iso$ with structure group $\Diff(S)$.
By Proposition \ref{jleg} the map $J_L^{\mathcal G^\iso}$ is a $\Diff(P,\xi)$-equivariant injective immersion.
In view of (the trivial inclusion in) Equation~\eqref{E:dp2}, we have $\omega^{\mathcal E}(\zeta^{\mathcal E}_X,\zeta^{\mathcal E}_Z)=0$ for all $X\in\mathfrak X(P,\xi)$ and $Z\in\mathfrak X(S)$.
Since $\Diff_c(P,\xi)$ acts infinitesimally transitive on $\mathcal E^\iso$, the $1$-form $\omega^{\mathcal E}(-,\zeta^{\mathcal E}_Z)$, thus, vanishes when pulled back to $\mathcal E^\iso$.
Hence, the restriction of $\omega^{\mathcal E}$ to $\mathcal E^\iso$ is vertical.
We conclude that there exists a unique $2$-form $\omega^{\mathcal G^\iso}$ on $\mathcal G^\iso$ such that $(q^\iso)^*\omega^{\mathcal G^\iso}$ coincides with the pull back of $\omega^{\mathcal E}$ to $\mathcal E^\iso$.
Clearly, $\omega^{\mathcal G^\iso}$ is closed.
The $2$-form $\omega^{\mathcal G^\iso}$ is (weakly) non-degenerate in view of (the non-trivial inclusion in) Equation~\eqref{E:dp2}.
From \eqref{E:JKKS}, \eqref{E:JLdef}, \eqref{E:bracketM} and the equivariance of $q$ we immediately obtain $(q^\iso)^*(J_L^{\mathcal G^\iso})^*\omega^\KKS=(q^\iso)^*\omega^{\mathcal G^\iso}$, whence \eqref{E:KKSiso}.
The remaining assertions are now obvious.
\end{proof}

\begin{remark}
We expect that the isotropy group $\Diff_c(P,\xi;(N,\gamma))$ in Theorem~\ref{T:Giso}(a) is a closed Lie subgroup in $\Diff_c(P,\xi)$.
If this is the case then $\mathcal G^\iso_{(N,\gamma)}$ may be regarded as a homogeneous space in the category of smooth manifolds.
\end{remark}

\begin{example}
If $S$ is the circle $S^1$ and $P$ is a $3$-dimensional contact manifold, then the weighted non-linear Grassmannian $\mathcal G$ becomes the manifold of weighted (unparametrized) knots in $P$, and $\mathcal G^\iso$ is the (symplectic) manifold of weighted Legendrian knots in $P$.
By Theorem~\ref{T:Giso}, its connected components can be identified with coadjoint orbits of the identity component of the contact group.
\end{example}

\subsection{Weighted contact non-linear Grassmannians}\label{SS:Gcontact}

Let us now consider a $1$-form density $\rho\in\Omega^1(S,|\Lambda|_S)$ of contact type, i.e., $\ker\rho\subseteq TS$ is assumed to be a contact hyperplane distribution.
Then the reduced space $\mathcal G^\rho$, see \eqref{E:Grho}, consists of weighted contact submanifolds. 
More precisely, according to Lemma~\ref{L:Grho} we have
\begin{equation}\label{E:Gcontact}
\mathcal G^\rho\subseteq(\pi^{\mathcal G})^{-1}(\Gr^\contact_{(S,\ker\rho)}(P,\xi)),
\end{equation}
where $\Gr_{(S,\ker\rho)}^\contact(P,\xi)\subseteq\Gr_S(P)$ denotes the subset of contact submanifolds which are of type $(S,\ker\rho)$.
In contrast to the isotropic case, see \eqref{E:Giso}, the inclusion \eqref{E:Gcontact} is strict.

The maps in \eqref{E:EGEmbGr} restrict to a $\Diff(P,\xi)$-equivariant commutative diagram
\begin{equation}\label{E:EGEmbGrrho}
\vcenter{\xymatrix{
\mathcal E^\rho\ar[d]_-{q^\rho}\ar[rr]_-\cong^-{\pi^{\mathcal E^\rho}}&&\Emb^\contact_{(S,\ker\rho)}(P,\xi)\ar[d]\\
\mathcal G^\rho\ar[rr]^-{\pi^{\mathcal G^\rho}}&&\Gr^\contact_{(S,\ker\rho)}(P,\xi)
}}
\end{equation}
where $\Emb^\contact_{(S,\ker\rho)}(P,\xi)\subseteq\Emb(S,P)$ denotes the subset of contact embeddings inducing the contact structure $\ker\rho$ on $S$.

\begin{lemma}\label{L:Econtactsubmf}
If $\rho\in\Omega^1(S,|\Lambda|_S)$ is a contact $1$-form density, then the following hold true:
\begin{enumerate}[(a)]
\item $\Gr^\contact_{(S,\ker\rho)}(P,\xi)$ is an open subset of $\Gr_S(P)$.
\item $\Emb^\contact_{(S,\ker\rho)}(P,\xi)$ is an initial Fr\'echet submanifold of $\Emb(S,P)$.
\item The natural map 
$$
\Emb^\contact_{(S,\ker\rho)}(P,\xi)\to\Gr^\contact_{(S,\ker\rho)}(P,\xi)
$$
is a smooth principal bundle with structure group $\Diff(S,\ker\rho)$. 
\item
The natural map
$$
\mathcal L|_{\Emb^\contact_{(S,\ker\rho)}(P,\xi)}\xrightarrow{(\pi^{\mathcal L},J_R^{\mathcal L})}\Emb^\contact_{(S,\ker\rho)}(P,\xi)\times\Gamma^\infty\bigl((TS/\ker\rho)^*\otimes|\Lambda|_S\bigr)
$$
		is a diffeomorphism of Fr\'echet manifolds, providing a $\Diff(S,\ker\rho)$-equivariant trivialization of the bundle $\pi^{\mathcal L}\colon\mathcal L\to C^\infty(S,P)$ over $\Emb^\contact_{(S,\ker\rho)}(P,\xi)$.
\item The map $\pi^{\mathcal E}\colon\mathcal E\to\Emb(S,P)$ restricts to a diffeomorphism of Fr\'echet manifolds, 
$$
\mathcal E^\rho\cong\Emb^\contact_{(S,\ker\rho)}(P,\xi).
$$
\end{enumerate}
\end{lemma}

\begin{proof}
(a) follows from the Gray stability theorem, see \cite[Theorem~2.2.2]{G08}.
Locally around points in $\Gr^\contact_{(S,\rho)}(P,\xi)$, the Gray stability theorem permits to construct cross sections of the $\Diff(S)$-bundle $\Emb(S,P)\to\Gr_S(P)$ which take values in $\Emb^{\contact}_{(S,\ker\rho)}(P,\xi)$.
Such a local cross section, defined on an open subset $\mathcal U$ in $\Gr_S(P)$, provides a local trivialization of $\Diff(S)$-bundles, $\mathcal U\times\Diff(S)\cong\Emb(S,P)|_{\mathcal U}$, which maps $\mathcal U\times\Diff(S,\ker\rho)$ onto $\Emb^\contact_{(S,\ker\rho)}(P,\xi)|_{\mathcal U}$.
Recall that $\Diff(S,\ker\rho)$ is a Fr\'echet Lie group, and the natural inclusion into $\Diff(S)$ is initial, see \cite[Theorem~43.19]{KM97}.
Whence (b) and (c).

Since $\rho$ is nowhere vanishing, the map in (d) is a bijection.
This map is smooth because the inclusion $\Emb^\contact_{(S,\ker\rho)}(P,\xi)\subseteq\Emb(S,P)$ is initial.
To see that its inverse is smooth too, we fix a vector bundle homomorphism $\sigma\colon TS/\ker\rho\to TS$ splitting the canonical projection $TS\to TS/\ker\rho$.
Let $\mathcal W$ denote the set of embeddings $\varphi\in\Emb(S,P)$ for which the composition
$$
TS/\ker\rho\xrightarrow\sigma TS\xrightarrow{T\varphi}\varphi^*TP\to\varphi^*L
$$
is an isomorphism of line bundles over $S$.
Clearly, $\mathcal W$ is an open neighborhood of $\Emb^\contact_{(S,\ker\rho)}(P,\xi)$ in $\Emb(S,P)$.
We obtain a smooth map 
$$
s\colon\mathcal W\times\Gamma^\infty\bigl((TS/\ker\rho)^*\otimes|\Lambda|_S\bigr)\to\mathcal L,
$$ 
characterized by $\pi^{\mathcal L}(s(\varphi,\beta))=\varphi$ and $J_R^{\mathcal L}(s(\varphi,\beta))\circ\sigma=\beta$, for all $\varphi\in\mathcal W$ and $\beta\in\Gamma^\infty\bigl((TS/\ker\rho)^*\otimes|\Lambda|_S\bigr)$.
Its restriction provides the smooth inverse for the map in (d).

Restricting the diffeomorphism in (d) to the level set $\mathcal E^\rho$, we obtain a diffeomorphism $\mathcal E^\rho\cong\Emb^\contact_{(S,\ker\rho)}(P,\xi)\times\{\rho\}$, whence (e).
\end{proof}

The diffeomorphism in Lemma~\ref{L:Econtactsubmf}(e) induces a natural diffeomorphism of Fr\"olicher spaces:
$$
\mathcal E^\rho/\Diff(S,\rho)\cong\Emb^\contact_{(S,\ker\rho)}(P,\xi)\times_{\Diff(S,\ker\rho)}\frac{\Diff(S,\ker\rho)}{\Diff(S,\rho)}.
$$
Note that the isotropy group $\Diff(S,\rho)$ is akin to the group of strict contact diffeomorphisms.

The diffeomorphism in Lemma~\ref{L:Econtactsubmf}(d) induces a diffeomorphism
$$
\mathcal G|_{\Gr^\contact_{(S,\ker\rho)}(P,\xi)}\cong\Emb^\contact_{(S,\ker\rho)}(P,\xi)\times_{\Diff(S,\ker\rho)}\Gamma^\infty\bigl(((TS/\ker\rho)^*\otimes|\Lambda|_S)\setminus S\bigr)
$$
which restricts to a natural diffeomorphism of Fr\"olicher spaces,
\begin{equation}\label{E:GrhoOrho}
\mathcal G^\rho\cong\Emb^\contact_{(S,\ker\rho)}(P,\xi)\times_{\Diff(S,\ker\rho)}\mathcal O_\rho.
\end{equation}
Here $\mathcal G^\rho$ is equipped with the Fr\"olicher structure induced from $\mathcal G$, and $\mathcal O_\rho$ denotes the $\Diff(S,\ker\rho)$-orbit of $\rho$ equipped with the Fr\"olicher structure induced from $\Omega^1(S,|\Lambda|_S)$ which coincides with the Fr\"olicher structure induced from $\Gamma^\infty((TS/\ker\rho)^*\otimes|\Lambda|_S)$.

\begin{remark}
If $\alpha$ is a contact form for $\xi$, then contact submanifolds $N$ are characterized by the fact that $\iota_N^*\alpha$ is a contact form on $N$, and the identification in Remark~\ref{R:Grho} becomes
$$
\mathcal G^\rho=\bigl\{(N,\nu):\textrm{$N\in\Gr_S^\contact(P)$,
$\nu\in\Gamma^\infty(|\Lambda|_N\setminus N)$, and $(N,\iota_N^*\alpha\otimes\nu)\cong(S,\rho)$}\bigr\}.
$$
If $(N,\nu)\in\mathcal G^\rho$ then any other weight on $N$ allowed in $\mathcal G^\rho$ must be of the form
$$
\frac{f^*\iota_N^*\alpha}{\iota_N^*\alpha}\cdot\frac{f^*\nu}\nu\cdot\nu
$$ 
for a contact diffeomorphism $f\in\Diff(N,\ker\iota_N^*\alpha)$.
Thus, unlike the isotropic case \eqref{gis}, in the contact case not all weights on a contact submanifold $N\in\Gr^\contact_S(P)$ are allowed in $\mathcal G^\rho$, i.e., the inclusion in \eqref{E:Gcontact} is strict.
\end{remark}

\begin{remark}
Let $\rho\in\Omega^1(S,|\Lambda|_S)$ be a contact $1$-form density.
Since $\mathcal G^\rho$ may not be a manifold, we refrain from considering the Kostant--Kirillov--Souriau form on $\mathcal G^\rho$.
However, formally pulling back the Kostant--Kirillov--Souriau form along $J^{\mathcal E^\rho}_L\colon\mathcal E^\rho\to\mathfrak X(P,\xi)^*$, we obtain a well defined smooth $2$-form $(J^{\mathcal E^\rho}_L)^*\omega^\KKS$ on $\mathcal E^\rho$, characterized by
$$
((J^{\mathcal E^\rho}_L)^*\omega^\KKS)(\zeta^{\mathcal E^\rho}_X(\Phi),\zeta^{\mathcal E^\rho}_Y(\Phi)):=\langle J_L^{\mathcal E^\rho}(\Phi),[X,Y]\rangle,
$$
where $\Phi\in\mathcal E^\rho$ and $X,Y\in\mathfrak X(P,\xi)$, cf.\ Remark~\ref{R:KKS} and Theorem~\ref{T:Erho}.
Proceeding exactly as in the proof of Theorem~\ref{T:Giso}, we see that this coincides with $\omega^{\mathcal E^\rho}$, the pull back of the symplectic form $\omega^{\mathcal E}$ to $\mathcal E^\rho$, i.e.,
$$
(J^{\mathcal E^\rho}_L)^*\omega^\KKS=\omega^{\mathcal E^\rho}.
$$
\end{remark}

The discussion in the next example shows that the situation is as nice as one could wish for $1$-dimensional $S$.
Subsequently, we will see that the situation is considerably more delicate in general, see Proposition~\ref{P:contactGrho}.

\begin{example}\label{Ex:S1}
Let us specialize to the circle, $S=S^1$.
In this case, any contact $1$-form density $\rho\in\Omega^1(S,|\Lambda|_S)$ gives rise to an orientation and a Riemannian metric on $S$.
We write $\sqrt{|\rho|}$ for the induced volume density on $S$, and denote the total volume by $\vvol(\rho):=\int_S\sqrt{|\rho|}$.
Using parametrization by arc length it is easy to see that two contact $1$-form densities lie in the same $\Diff(S)$-orbit iff they have the same total volume.
In particular, the $\Diff(S)$-orbits through contact $1$-form densities are closed submanifolds in $\Omega^1(S,|\Lambda|_S)$.
Moreover, parametrization by arc length provides local smooth sections for the $\Diff(S)$-action on these orbits.
In particular, the assumption in Lemma~\ref{L:Grhosec} is satisfied in this case.

Suppose $(P,\xi)$ is a contact manifold and let $\rho\in\Omega^1(S,|\Lambda|_S)$ be a contact $1$-form density on $S=S^1$.
Using \eqref{E:GrhoOrho} we conclude that $\mathcal G^\rho$ is a closed submanifold of $\mathcal G$.
Parametrization by arc length provides local smooth sections of $\mathcal E^\rho\to\mathcal G^\rho$ and the latter is a locally trivial smooth principal bundle.
Note that the structure group $\Diff(S,\rho)\cong\operatorname{SO}(1)$ is a closed Lie subgroup of $\Diff(S)$.
By Proposition~\ref{P:Grho}, the $\Diff_c(P,\xi)$-action on $\mathcal G^\rho$ admits local smooth sections.
	Moreover, its orbits are open and closed subsets in $\mathcal G^\rho$ which may be identified with coadjoint orbits of the contact group via the restriction of $J^{\mathcal G}_L$.
The symplectic form on $\mathcal E$ gives rise to a reduced symplectic form on $\mathcal G^\rho$ which coincides with the pull back of the Kostant--Kirillov--Souriau symplectic form via $J_L^{\mathcal G}$ as in Theorem~\ref{T:Giso}(b).
If $P$ is $3$-dimensional, then $\mathcal G^\rho$ is a (symplectic) manifold of weighted transverse knots.

A slightly more explicit description can be given if the contact structure is admits a contact form, $\xi=\ker\alpha$.
Then, via the identification in Remark~\ref{R:Grho}, we have
$$
\mathcal G^\rho\cong\left\{(N,\nu)\in\Gr^\wt_{S^1}(P)\middle|\iota_N^*\alpha\neq0,\ \vvol(\iota_N^*\alpha\otimes\nu)=\vvol(\rho)\right\},
$$
for every contact $1$-form density $\rho$.
\end{example}

The following result shows that the trace topology on $\mathcal G^\rho$ induced from $\mathcal G$ is not the appropriate topology for general contact $\rho$.

\begin{proposition}\label{P:contactGrho}
There exist a compact contact manifold $(P,\xi)$, a compact manifold $S$, and a contact $1$-form density $\rho\in\Omega^1(S,|\Lambda|_S)$ such that the continuous bijection
\begin{equation}\label{E:cb}
\mathcal E^\rho/\Diff(S,\rho)\to(J_R^{\mathcal E})^{-1}(\rho\cdot\Diff(S))/\Diff(S)
\end{equation}
induced by the natural inclusion is not a homeomorphism with respect to the quotient topologies.
In particular, the continuous bijection $\mathcal E^\rho/\Diff(S,\rho)\to\mathcal G^\rho$ induced by $q$ is not a homeomorphism where the right hand side is equipped with the trace topology induced from $\mathcal G$.
Moreover, for $(N,\gamma)\in\mathcal G^\rho$ the map provided by the action, $\Diff(P,\xi)\to\mathcal G^\rho$, $g\mapsto\Psi^{\mathcal G}_g(N,\gamma)$, does not admit a continuous local (with respect to the trace topology induced from $\mathcal G$) right inverse defined in a neighborhood of $(N,\gamma)$.
\end{proposition}

The following lemma will be crucial in the proof of Proposition~\ref{P:contactGrho}.

\begin{lemma}\label{L:Salpha}
There exists a compact contact manifold $(S,\alpha)$ and a sequence of diffeomorphisms $f_n\in\Diff(S)$ with the following properties:
\begin{enumerate}[(a)]
\item $f_n^*\alpha\to\alpha$ with respect to the $C^\infty$-topology.
\item There does not exist a sequence of diffeomorphisms $g_n\in\Diff(S)$ such that $g_n^*\alpha=f_n^*\alpha$ for all $n$ and $g_n\to\id_S$ with respect to the $C^0$-topology.
\end{enumerate}
\end{lemma}

\begin{proof}
Let $(M,\omega)$ be a connected compact symplectic manifold with integral symplectic form.
Choose a sequence of non-empty open subsets $U,U_1,U_2,U_3,\dots$ of $M$ such that their closures are mutually disjoint.
Choose points $x\in U$ and $x_n\in U_n$.
Assume that the sequences of closures $\bar U_n$ only accumulates at a single point.
Choose Hamiltonian diffeomorphisms $h_n\in\Ham(M,\omega)$ such that for each $n$ we have
\begin{enumerate}[(i)]
\item $h_n(y)=y$ for all $y\in\bigcup_{i\neq n}\bar U_i$, and
\item $h_n(x)=x_n$.
\end{enumerate}
Shrinking $U_n$, we may moreover assume 
\begin{enumerate}[(i)]
\setcounter{enumi}{2}
\item $h_n^{-1}(U_n)\subseteq U$.
\end{enumerate}
For each $n$ let $\lambda_n$ be a compactly supported smooth function on $U_n$ such that $\lambda_n$ is constant and strictly positive in a neighborhood of $x_n$.
Let $\lambda\colon M\to\mathbb R$ denote the function which coincides with $\lambda_n$ on $U_n$ and vanishes outside $\bigcup_nU_n$.
Multiplying $\lambda_n$ with a sufficiently fast decreasing sequence of constants, we may assume that the following hold true:
\begin{enumerate}[(i)]
\setcounter{enumi}{3}
\item $\lambda$ is smooth on $M$, and
\item $h_n^*\lambda\to\lambda$ with respect to the $C^\infty$ topology on $M$.
\end{enumerate}
By construction, we have:
\begin{enumerate}[(i)]
\setcounter{enumi}{5}
\item $\lambda$ is constant and strictly positive on a neighborhood of $x_n$, for each $n$, and
\item $\lambda$ vanishes on $U$.
\end{enumerate}

Let $p\colon S\to M$ be the circle bundle with Chern class $[\omega]$ and let $\tilde\alpha\in\Omega^1(S)$ be a principal connection $1$-form with curvature $\omega$.
Hence, $\tilde\alpha$ is a contact form on $S$.
It is well known that Hamiltonian diffeomorphisms on $M$ can be lifted to strict contact diffeomorphisms on $S$.
Hence, there exist diffeomorphisms $f_n\in\Diff(S)$ such that $f_n^*\tilde\alpha=\tilde\alpha$ and $p\circ f_n=h_n\circ p$.
We consider the contact form $\alpha:=e^{-p^*\lambda}\tilde\alpha$ on $S$.
From (v) we immediately obtain $f_n^*\alpha\to\alpha$, whence (a).

To see (b), let $E$ denote the Reeb vector field of $\alpha$.
From (vii) we see that $\alpha$ coincides with the principal connection $\tilde\alpha$ on $p^{-1}(U)$.
Over $p^{-1}(U)$, the Reeb vector field $E$ thus coincides with the fundamental vector field of the principal circle action.
For each $y\in p^{-1}(U)$ we thus have $\operatorname{Fl}_t^E(y)=y\Leftrightarrow t\in2\pi\mathbb Z$.
Hence, if $g\in\Diff(S)$ is sufficiently close to the identity with respect to the $C^0$-topology, then 
\begin{equation}\label{E:ee1}
\operatorname{Fl}_t^{g^*E}(x)=x\Leftrightarrow t\in2\pi\mathbb Z.
\end{equation}
Note that $g^*E$ is the Reeb vector field of $g^*\alpha$, and $f_n^*E$ is the Reeb vector field of $f_n^*\alpha$.
For each $n$ there exists a constant $0<c_n<1$ such that $f_n^*\alpha$ coincides with $c_n\alpha$ on a neighborhood of $p^{-1}(x)$, see (ii) and (vi).
Hence, $f_n^*E$ coincides with $c_n^{-1}E$ on a neighborhood of $p^{-1}(x)$.
In particular, 
\begin{equation}\label{E:ee2}
\operatorname{Fl}_t^{f_n^*E}(x)=x\Leftrightarrow t\in2\pi c_n\mathbb Z.
\end{equation}
Comparing \eqref{E:ee1} and \eqref{E:ee2} and using $c_n\neq1$, we conclude $g^*E\neq f_n^*E$ and thus $g^*\alpha\neq f_n^*\alpha$.
This shows (b).
\end{proof}

\begin{proof}[Proof of Proposition~\ref{P:contactGrho}]
We consider a closed manifold $S$ of dimension $2k+1$, a contact form $\alpha$ on $S$, and diffeomorphisms $f_n\in\Diff(S)$ as in Lemma~\ref{L:Salpha}.
Using Gray's stability result \cite[Theorem~2.2.2]{G08}, we may w.l.o.g.\ assume that each $f_n$ is a contact diffeomorphism.
Hence, there exist smooth functions $\lambda_n$ on $S$ such that $f_n^*\alpha=\lambda_n\alpha$.
Since $f_n^*\alpha\to\alpha$, we have $\lambda_n\to1$, as $n\to\infty$.
In particular, we may assume $\lambda_n>0$.

We let $\mu:=|\alpha\wedge(d\alpha)^k|$ denote the volume density associated with the volume form $\alpha\wedge(d\alpha)^k$.
Note that $f_n^*\mu=\lambda_n^{k+1}\mu$.
Moreover, we put $\rho:=\alpha\otimes\mu\in\Omega^1(S,|\Lambda|_S)$.
We consider the manifold $P:=S$ equipped with the contact structure $\xi:=\ker(\alpha)$.
Using the volume density $\mu$ on $S$ and the contact form $\alpha$ on $P$, we may identify $\mathcal E=\Emb(S,P)\times C^\infty(S,\mathbb R^\times)$, see Remark~\ref{R:mutriv}.
Using this identification we define a sequence $\Phi_n\in\mathcal E$ by $\Phi_n:=(\id_S,\lambda_n^{k+2})$.
Clearly, $\Phi_n$ converges to $\Phi:=(\id_S,1)\in\mathcal E$.
Using \eqref{E:JRCSL} we find $J_R^{\mathcal E}(\Phi)=\rho$ and 
\begin{equation}\label{E:aaa}
J_R^{\mathcal E}(\Phi_n)=\lambda_n^{k+2}\alpha\otimes\mu=f_n^*(\alpha\otimes\mu)=f_n^*\rho=\lambda_n^{k+2}\rho.
\end{equation}
In particular, we have $\Phi\in\mathcal E^\rho$, $\Phi_n\in(J_R^{\mathcal E})^{-1}(\rho\cdot\Diff(S))$ and $\Phi_n\to\Phi$, as $n\to\infty$.

We will now show that the corresponding sequence in $\mathcal E^\rho/\Diff(S,\rho)$ does not converge, cf.~\eqref{E:cb}.
Suppose, by contradiction, there exists a sequence of diffeomorphisms $g_n\in\Diff(S)$ such that $\psi_{g_n}^{\mathcal E}(\Phi_n)\in\mathcal E^\rho$ and $\psi^{\mathcal E}_{g_n}(\Phi_n)$ converges in $\mathcal E^\rho/\Diff(S,\rho)$.
W.l.o.g.\ we may moreover assume that $\psi^{\mathcal E}_{g_n}(\Phi_n)$ converges in $\mathcal E^\rho$.
In particular, $g_n$ converges to a diffeomorphism $g\in\Diff(S)$.
Using the $\Diff(S)$ equivariance of $J_R^{\mathcal E}$, the relation $\rho=J_R^{\mathcal E}(\psi^{\mathcal E}_{g_n}(\Phi_n))$ yields 
\begin{equation}\label{E:bbb}
g_n^*\rho=J_R^{\mathcal E}(\Phi_n).
\end{equation}
In particular, letting $n\to\infty$, we obtain $g\in\Diff(S,\rho)$.
Replacing $g_n$ with $g_n\circ g^{-1}$ we may w.l.o.g.\ assume that $g_n\to\id_S$.
Combining \eqref{E:aaa} and \eqref{E:bbb} we see that $g_n$ is a contact diffeomorphism.
Hence, there exist smooth functions $\tilde\lambda_n$ with $g_n^*\alpha=\tilde\lambda_n\alpha$.
We obtain $\lambda_n^{k+2}\rho=g_n^*\rho=\tilde\lambda_n^{k+2}\rho$ and thus $\tilde\lambda_n^{k+2}=\lambda_n^{k+2}$.
Since $g_n$ converges to the identity, we may assume $\tilde\lambda_n>0$.
Hence, $\tilde\lambda_n=\lambda_n$ and thus $g_n^*\alpha=f_n^*\alpha$.
This contradicts the choice of $f_n$, see Lemma~\ref{L:Salpha}(b).
Hence, the sequence in $\mathcal E^\rho/\Diff(S,\rho)$ corresponding to $\Phi_n$ does not converge.

This shows that the continuous bijection \eqref{E:cb} is not a homeomorphism.
The remaining statements follow immediately from the fact that the projection $q\colon\mathcal E\to\mathcal G$ admits local (smooth) sections.
\end{proof}

\section{Comparison with the EPDiff dual pair}\label{S:EPDiff}

A pair of moment maps has been introduced by D.~D.~Holm  and J.~E.~Marsden \cite{HM} in relation to the EPDiff equations describing geodesics on the group of all diffeomorphisms. 
The left moment map provides singular solutions of these equations, whereas the right moment map provides a constant of motion for the collective dynamics of these singular solutions. 
It has been shown in \cite{GBV12} that the pair of moment maps, when restricted to an appropriate open subset, do indeed form a symplectic dual pair.
In this section we relate the EPDiff dual pair of a manifold with the EPContact dual pair of its projectivized cotangent bundle.

\subsection{The dual pair for the EPDiff equation}

The (regular) cotangent bundle to the space of smooth maps from a closed manifold $S$ into a manifold $Q$ can be equipped with the canonical symplectic structure. 
Recall that the tangent space at $\eta\in C^\infty(S,Q)$ is $T_\eta C^\infty(S,Q)=\Gamma^\infty(\eta^*TQ)$.
Using the canonical pairing, we regard the space of 1-form densities along $\eta$,
\begin{equation}\label{E:Tsrho}
\Gamma^\infty(|\Lambda|_S\otimes\eta^*T^*Q)=T^*_\eta C^\infty(S,Q)_\reg,
\end{equation}
as the regular cotangent space at $\eta$.
In this way we identify the space of smooth fiberwise linear maps from $|\Lambda|_S^*$ to $T^*Q$ with the regular cotangent bundle:
\begin{equation}\label{identif}
C^\infty_\lin(|\Lambda|_S^*,T^*Q)=T^*C^\infty(S,Q)_\reg.
\end{equation}
Via this identification, the canonical $1$-form on $T^*C^\infty(S,Q)_\reg$ can be written in the form
\begin{equation}\label{E:thetaTQ}
\theta^{T^*C^\infty(S,Q)_\reg}(A)=\int_S\theta^{T^*Q}(A),
\end{equation}
where $A$ is a tangent vector at $\Phi\in T^*C^\infty(S,Q)_\reg$, i.e.,
$$
A\in T_\Phi C^\infty_\lin(|\Lambda|_S^*,T^*Q)
=\left\{A\in C^\infty(|\Lambda|_S^*,TT^*Q)\middle|\begin{array}{l}\pi^{TT^*Q}\circ A=\Phi\quad\textrm{and}\\\forall t\in\mathbb R:A\circ\delta^{|\Lambda|_S^*}_t=T\delta^{T^*Q}_t\circ A\end{array}\right\}
$$
and $\theta^{T^*Q}$ denotes the canonical $1$-form on $T^*Q$.
As before, the integrand $\theta^{T^*Q}(A)$ is a fiberwise linear function on the total space of $|\Lambda|^*_S$, which may be regarded as a section on $|\Lambda|_S$ and integrated over $S$.
The differential $d\theta^{T^*C^\infty(S,Q)_\reg}$ is the canonical (weakly non-degenerate) symplectic form on $T^*C^\infty(S,Q)_\reg$.

The cotangent lifted actions of the groups $\operatorname{Diff}(Q)$ and $\operatorname{Diff}(S)$ on the manifold $C^\infty(S,Q)$ preserve the canonical $1$-form $\theta^{T^*C^\infty(S,Q)_\reg}$.
In particular, these actions are Hamiltonian with equivariant moment maps $J_\Sing\colon T^*C^\infty(S,Q)_\reg\to\mathfrak X(Q)^*$,
\begin{equation}\label{E:JSing}
\langle J_\Sing(\Phi),Y\rangle=\theta^{T^*C^\infty(S,Q)_\reg}\left(\zeta_Y^{T^*C^\infty(S,Q)_\reg}(\Phi)\right),
\end{equation}
and $J_S\colon T^*C^\infty(S,Q)_\reg\to\mathfrak X(S)^*$,
\begin{equation}\label{E:JS}
\langle J_S(\Phi),Z\rangle=\theta^{T^*C^\infty(S,Q)_\reg}\left(\zeta_Z^{T^*C^\infty(S,Q)_\reg}(\Phi)\right),
\end{equation}
respectively, where $\Phi\in T^*C^\infty(S,Q)_\reg$.
Here $\zeta_Y^{T^*C^\infty(S,Q)_\reg}$ and $\zeta_Z^{T^*C^\infty(S,Q)_\reg}$ denote the fundamental vector fields on $T^*C^\infty(S,Q)_\reg$ corresponding to the (infinitesimal) action of $Y\in\mathfrak X(Q)$ and $Z\in\mathfrak X(S)$, respectively.
More explicitly, using the identification \eqref{E:Tsrho}, these cotangent moment maps are 
$$
\langle J_\Sing(\Phi),Y\rangle=\int_S \Phi(Y\circ\eta)
\qquad\text{and}\qquad
\langle J_S(\Phi),{Z}\rangle=\int_S\Phi(T \eta\circ Z),
$$ 
where $\eta\in C^\infty(S,Q)$ and $\Phi\in T^*_\eta C^\infty(S,Q)_\reg=\Gamma^\infty(|\Lambda|_S\otimes\eta^*T^*Q)$.
In particular, the second formula shows that $J_S$ takes values in $\Omega^1(S,|\Lambda|_S)\subseteq\mathfrak X(S)^*$.
More precisely, $J_S(\Phi)$ is the $1$-form density on $S$ corresponding to the $1$-homogeneous vertical $1$-form $\Phi^*\theta^{T^*Q}$ on the total space of $|\Lambda|_S^*$ where we regard $\Phi\colon|\Lambda|_S^*\to T^*Q$, cf.~\eqref{identif}.

We denote by $T^*C^\infty(S,Q)_\reg^\times$ the open subset of \eqref{identif} that corresponds to the space $C^\infty_\isom(|\Lambda|_S^*,T^*Q)$ of smooth maps that are linear and injective on fibers.
Restricting further the actions and moment maps to the open subset $T^*\Emb(S,Q)_\reg^\times$, we obtain the EPDiff symplectic dual pair \cite{GBV12}:
\begin{equation}\label{EPDiff_dualpair} 
\mathfrak X(Q)^*\xleftarrow{\quad J_\Sing\quad}T^*\Emb(S,Q)_\reg^\times\xrightarrow{\quad J_S\quad}\Omega^1(S,|\Lambda|_S)\subseteq\mathfrak{X}(S)^*.
\end{equation}
The left moment map $J_\Sing$ provides the formula for singular solutions of the EPDiff equations, whereas the right moment map $J_S$ provides a Noether conserved quantity for the (collective) Hamiltonian dynamics of these singular solutions in terms of the canonical variable $\Phi\in T^*\Emb(S,Q)_\reg^\times$, see \cite{HM}.

\begin{remark}
Fixing a volume density $\mu$ on $S$, we obtain identifications $T^*C^\infty(S,Q)_\reg\cong C^\infty(S,T^*Q)$ and $T^*C^\infty(S,Q)_\reg^\times\cong C^\infty(S,T^*Q\setminus Q)$, cf.~\eqref{identif}, as well as $\Omega^1(S,|\Lambda|_S)\cong\Omega^1(S)$. 
Using these identifications, the moment maps may be written in the form 
$$
\langle J_\Sing(\phi),Y\rangle=\int_S\phi(Y)\mu\qquad\text{and}\qquad J_S(\phi)=\phi^*\theta^{T^*Q},
$$ 
where $\phi\in C^\infty(S,T^*Q)$ and $Y\in\mathfrak X(Q)$, cf.\ \cite[Section~5]{GBV12}.
\end{remark}

\subsection{The projectivized cotangent bundle}

We will compare  the EPDiff dual pair described in the preceding paragraph with the EPContact dual pair associated with the projectivized cotangent bundle.
Recall that the projectivized cotangent bundle,
$$
P:=\mathbb P(T^*Q)=(T^*Q\setminus Q)/\mathbb R^\times\stackrel{p}{\longrightarrow}Q,
$$
admits a canonical contact structure \cite[Appendix~4]{A89} given by
\begin{equation}\label{purpos}
\xi_\ell=(T_\ell p)^{-1}(\ker\beta),
\end{equation}
where $\ell\in P$ and $\beta\in T^*Q$ is any non-zero element of $\ell$.
As the natural action of $\Diff(Q)$ on $P$ preserves the contact structure $\xi$, we obtain an injective group homomorphism 
$$
\Diff(Q)\to\Diff(P,\xi).
$$

The line bundle $L^*$, see Section~\ref{SS:P}, associated with the projectivized cotangent bundle is naturally isomorphic to the canonical line bundle  over $P$:
$$\gamma=\left\{(\ell,\beta)\middle|\ell\in P,\beta\in\ell\right\}.$$
Indeed, the vector bundle homomorphism $\chi\colon\gamma\to T^*P$ over the identity on $P$,
given by $\chi(\ell,\beta):=\beta\circ T_\ell p$, induces an isomorphism of line bundles, $\chi\colon\gamma\to L^*$.
Furthermore,
\begin{equation}\label{E:thetaTsQ}
\chi^*\theta^{L^*}=\pr_2^*\theta^{T^*Q},
\end{equation}
where $\pr_2\colon\gamma\to T^*Q$ denotes the canonical projection, i.e., the blow-up of the zero section in $T^*Q$.
We consider the map $\kappa\colon L^*\to T^*Q$, $\kappa:=\pr_2\circ\chi^{-1}$.
One readily checks:

\begin{lemma}\label{kap}
The map $\kappa$ is a vector bundle homomorphism over the bundle projection $p$,
$$
\xymatrix{
L^*\ar[rr]^-{\kappa}\ar[d]_{\pi^{L^*}}&&T^*Q\ar[d]^-{\pi^{T^*Q}}\\
P\ar[rr]^-{p}&& Q
}
$$
which has the following properties:
\begin{enumerate}[(a)]
\item $\kappa$ is equivariant over the homomorphism $\Diff(Q)\to\Diff(P,\xi)$.
\item $\kappa$ restricts to a diffeomorphism from $L^*\setminus P$ onto $T^*Q\setminus Q$.
\item $\kappa^*\theta^{T^*Q}=\theta^{L^*}$.
\end{enumerate}
\end{lemma}

Composition with $\kappa$ provides a map, cf.~\eqref{identif}, 
$$
\mathcal L=C^\infty_\lin(|\Lambda|_S^*,L^*)\xrightarrow{\quad\kappa_*\quad}C^\infty_\lin(|\Lambda|_S^*,T^*Q)=T^*C^\infty(S,Q)_\reg
$$
which fits into the following diagram:
\begin{equation}\label{D:PQ}
\vcenter{
\xymatrix{
\mathfrak X(P,\xi)^*\ar[d]^-{i^*}
&&{\mathcal L}\ar[ll]_{J_L^{\mathcal L}}\ar[d]^-{\kappa_*}\ar[rr]^{J_R^{\mathcal L}}
&&\mathfrak X(S)^*\ar@{=}[d]
\\
\mathfrak X(Q)^*
&&{T^*C^\infty(S,Q)_\reg}\ar[ll]_{J_\Sing}\ar[rr]^{J_S}
&&\mathfrak X(S)^*
}}
\end{equation}
Here $i^*$ denotes the dual of the Lie algebra homomorphism $i\colon\mathfrak X(Q)\to\mathfrak X(P,\xi)$ corresponding to the homomorphism of groups $\Diff(Q)\to\Diff(P,\xi)$.  
Clearly, $i^*$ is equivariant over the homomorphism $\Diff(Q)\to\Diff(P,\xi)$.
Note that via \eqref{E:XPxi} and $\kappa$, the Lie algebra $\mathfrak X(P,\xi)=C^\infty_{\lin}(L^*)$ may be regarded as the space of homogeneous functions on $T^*Q\setminus Q$, while the image of $i$ consists of those which extend to fiberwise linear functions on $T^*Q$.

\begin{proposition}\label{T:comp}
The diagram \eqref{D:PQ} commutes.
The map $\kappa_*$ is equivariant over the homomorphism $\Diff(Q)\to\Diff(P,\xi)$ and also $\Diff(S)$-equivariant.
It restricts to a symplectic diffeomorphism from $\mathcal M\subseteq\mathcal L$ onto $T^*C^\infty(S,Q)^\times_\reg$.
\end{proposition}

\begin{proof}
The map $\kappa_*$ is equivariant over the homomorphism $\Diff(Q)\to\Diff(P,\xi)$ since $\kappa$ has the same property, see Lemma~\ref{kap}(a).
Clearly, $\kappa_*$ is $\Diff(S)$-equivariant too.
Hence, the fundamental vector fields are $\kappa_*$-related, that is,
\begin{equation}\label{E:zetaYZ}
T\kappa_*\circ\zeta^{\mathcal L}_{i(Y)}=\zeta_Y^{T^*C^\infty(S,Q)_\reg}\circ\kappa_*
\qquad\text{and}\qquad
T\kappa_*\circ\zeta^{\mathcal L}_Z=\zeta_Z^{T^*C^\infty(S,Q)_\reg}\circ\kappa_*
\end{equation}
for $Y\in\mathfrak X(Q)$ and $Z\in\mathfrak X(S)$.
Using Lemma~\ref{kap}(c), \eqref{E:thetaL}, and \eqref{E:thetaTQ}, we obtain
\begin{equation}\label{E:kappatheta}
(\kappa_*)^*\theta^{T^*C^\infty(S,Q)_\reg}=\theta^{\mathcal L}.
\end{equation}
Combining the latter with the first equation in \eqref{E:zetaYZ}, we see that the square on the left hand side in \eqref{D:PQ} commutes, cf.\ \eqref{E:JSing} and \eqref{E:JLdef}.
Combining \eqref{E:kappatheta} with the second equation in \eqref{E:zetaYZ}, we see that the square on the right hand side in \eqref{D:PQ} commutes, cf.\ \eqref{E:JS} and \eqref{E:JRdef}.
As $\kappa$ restricts to a diffeomorphism from $L^*\setminus P$ onto $T^*Q\setminus Q$, the map $\kappa_*$ restricts to a diffeomorphism from $\mathcal M$ onto $T^*C^\infty(S,Q)^\times_\reg$ which is symplectic in view of \eqref{E:kappatheta}.
\end{proof}

\subsection{Coadjoint orbits of the diffeomorphism group}

The first line in \eqref{D:PQ} becomes a dual pair when restricted to $\mathcal E=\Emb_\lin(|\Lambda|_S^*,L^*)$.
The second line has to be restricted to $T^*\Emb(S,Q)^\times_\reg$ to become a dual pair. 
The latter is a proper open subset of the image $\kappa_*(\mathcal E)$.
To make this more precise, note that 
$$
\mathcal E_Q:=\left\{\Phi\in\mathcal E:p\circ\pi^{\mathcal E}(\Phi)\in\Emb(S,Q)\right\}
$$
is a $\Diff(S)$ invariant open subset of $\mathcal E$.
Since $p\colon P\to Q$ is $\Diff(Q)$ equivariant, $\mathcal E_Q$ is invariant under $\Diff(Q)$ too.
According to Proposition~\ref{T:comp}, the map $\kappa_*$ restricts to a $\Diff(Q)$ and $\Diff(S)$ equivariant symplectomorphism which makes the following diagram commute:
\begin{equation}\label{D:EQ}
\vcenter{
\xymatrix{
&&\mathcal E_Q\ar@/_2ex/[lld]_{J_L^{\mathcal E_Q}}\ar[d]_-\cong^-{\kappa_*}\ar@/^2ex/[rrd]^{J_R^{\mathcal E_Q}}
\\
\mathfrak X(Q)^*
&&T^*\Emb(S,Q)_\reg^\times\ar[ll]_{J_\Sing}\ar[rr]^{J_S}
&&\mathfrak X(S)^*
}}
\end{equation}
Here $J_L^{\mathcal E_Q}$ and $J_R^{\mathcal E_Q}$ denote the restrictions of $i^*\circ J_L^{\mathcal L}$ and $J_R^{\mathcal L}$, respectively, cf.~\eqref{D:PQ}.

This can be used to obtain a geometric interpretation of some coadjoint orbits of $\Diff(Q)$.
To this end consider the open subset
$$
\mathcal G_Q:=\left\{(N,\gamma)\in\mathcal G:p|_N\in\Emb(N,Q)\right\}
$$
of $\mathcal G$.
Since $\mathcal E_Q=q^{-1}(\mathcal G_Q)$, the principal $\Diff(S)$-bundle $q\colon\mathcal E\to\mathcal G$ restricts to a principal $\Diff(S)$ bundle $q_Q\colon\mathcal E_Q\to\mathcal G_Q$.
Restricting $i^*\circ J_L^{\mathcal G}\colon\mathcal G\to\mathfrak X(Q)^*$, see \eqref{E:JLG}, we obtain a smooth map
\begin{equation}\label{E:JLGQ}
J_L^{\mathcal G_Q}\colon\mathcal G_Q\to\mathfrak X(Q)^*,\qquad J_L^{\mathcal G_Q}\circ q_Q=J_L^{\mathcal E_Q}.
\end{equation}
In view of \eqref{jlg} we have the explicit formula
\begin{equation}\label{E:jlgQ}
\langle J_L^{\mathcal G_Q}(N,\gamma),Y\rangle=\int_N\gamma(i(Y)|_N),
\end{equation}
where $(N,\gamma)\in\mathcal G_Q$ and $Y\in\mathfrak X(Q)$.
On the right hand side $i(Y)$ is regarded as a section of $L$, see \eqref{E:XPxi}, restricted to $N$ and contracted with $\gamma\in\Gamma^\infty(|\Lambda|_N\otimes L|_N^*)$ to produce a density on $N$.

\begin{proposition}
The map $J_L^{\mathcal G_Q}$ in \eqref{E:JLGQ} is a $\Diff(Q)$-equivariant injective immersion.
\end{proposition}

\begin{proof}
By Proposition~\eqref{jleg}, the map $J_L^{\mathcal G_Q}$ is smooth and $\Diff(Q)$-equivariant.
It is immersive because the moment maps at the bottom line of \eqref{D:EQ} are mutually completely orthogonal, see \cite[Theorem~5.6]{GBV12}.
The injectivity follows from the transitivity of the $\Diff(S)$-action on level sets of the left moment map $J_\Sing$, which is the content of \cite[Proposition~5.2]{GBV12}.
\end{proof}

Recall from Theorem~\ref{T:Giso} that $\mathcal G^\iso$ is a closed submanifold of $\mathcal G$.
Hence, $\mathcal G^\iso_Q:=\mathcal G_Q\cap\mathcal G^\iso$ is a closed splitting submanifold of $\mathcal G_Q$.
Consequently, $\mathcal E^\iso_Q:=\mathcal E_Q\cap\mathcal E^\iso=q^{-1}(\mathcal G^\iso_Q)$ is a closed splitting submanifold of $\mathcal E_Q$.
The projection $q_Q\colon\mathcal E_Q\to\mathcal G_Q$ restricts to a smooth principal bundle $q^\iso_Q\colon\mathcal E^\iso_Q\to\mathcal G^\iso_Q$ with structure group $\Diff(S)$.
Via $\kappa_*$ the manifold 
$$
\mathcal G_Q^\iso=\mathcal E_Q^\iso/\Diff(S)=\{(N,\gamma)\in\mathcal G:N\in\Gr_S^\iso(P),p|_N\in\Emb(N,Q)\}
$$
identifies with $J_S^{-1}(0)/\Diff(S)$, the reduced space at zero of the right $\Diff(S)$-action on $\Emb(S,Q)^\times_\reg$.

According to \cite[Proposition~5.5]{GBV12}, the group $\Diff_c(Q)$ acts locally transitive on $\mathcal E_Q^\iso$ and, thus, on $\mathcal G_Q^\iso$ too.
In particular, the $\Diff_c(Q)$-orbit $(\mathcal G^\iso_{Q})_{(N,\gamma)}$ through $(N,\gamma)\in\mathcal G^\iso_Q$ is open and closed in $\mathcal G^\iso_Q$.
From Theorem~\ref{T:Giso} we thus obtain:

\begin{corollary}\label{C:GisoQ}
The projection $q$ restricts to a smooth principal bundle $q^\iso_Q\colon\mathcal E^\iso_Q\to\mathcal G^\iso_Q$ with structure group $\Diff(S)$.
The restriction of the symplectic form $\omega^{\mathcal E}$ to $\mathcal E^\iso_Q$ descends to a (reduced) symplectic form $\omega^{\mathcal G^\iso_Q}$ on $\mathcal G^\iso_Q$.
The $\Diff(Q)$-equivariant injective immersion 
$$
J_L^{\mathcal G^\iso_Q}\colon\mathcal G^\iso_Q\to\mathfrak X(Q)^*,\qquad
\langle J_L^{\mathcal G^\iso_Q}(N,\gamma),Y\rangle=\int_N\gamma(i(Y)|_N),
$$ 
provided by restriction of $J_L^{\mathcal G_Q}$ from \eqref{E:JLGQ}, identifies $(\mathcal G^\iso_{Q})_{(N,\gamma)}$ with the coadjoint orbit of $\Diff_c(Q)$ through $J_L^{\mathcal G_Q^\iso}(N,\gamma)$ in such a way that  
\begin{equation}\label{E:KKSisoQ}
\bigl(J_L^{\mathcal G^\iso_{Q}}\bigr)^*\omega^\KKS=\omega^{\mathcal G^\iso_Q}.
\end{equation}
Here $(N,\gamma)\in\mathcal G^\iso_Q$ and $\omega^\KKS$ denotes the Kostant--Kirillov--Souriau symplectic form on the coadjoint orbit through $J_L^{\mathcal G_Q^\iso}(N,\gamma)\in\mathfrak X(Q)^*$.
\end{corollary}

\begin{remark}
In the Legendrian case one has a description of the coadjoint orbit that does not use contact geometry. 
A transverse Legendrian submanifold $N\subseteq \mathbb PT^*Q$ projects to a codimension one submanifold $N_0=p(N)\subseteq Q$, while $N_0$ has a unique Legendrian lift to the projectivized cotangent bundle 
$$ N_0\ni x\mapsto\ann(T_xN_0)\in \mathbb PT_x^*Q.$$
Moreover, the line bundle $L=T(\mathbb PT^*Q)/\xi$ restricted to $N$ is canonically isomorphic to the pull back of the normal line bundle, $p|_N^*TN_0^\perp$, since the contact hyperplane at $y\in N$ is $\xi_y=(T_yp)^{-1}(T_xN_0)$ for $x=p(y)\in N_0$.
Hence, the coadjoint orbit of $\Diff_c(Q)$ described above can be seen as
$$
\left\{(N_0,\gamma_0)\middle|\begin{array}{l}\textrm{$N_0\subseteq Q$ has codimension one and}\\\textrm{$\gamma_0\in\Gamma^\infty(|\Lambda|_{N_0}\otimes (TN_0^\perp)^*)$ is nowhere vanishing}\end{array}\right\},
$$
embedded into $\mathfrak X(Q)^*$ via $Y\mapsto\int_{N_0}\gamma_0(Y|_{N_0}\textrm{ mod } TN_0)$.
Note that we have a canonical identification 
$|\Lambda|_{N_0}\otimes (TN_0^\perp)^*=|\Lambda|_Q\otimes\mathcal O_{N_0}^\perp$, where $\mathcal O_{N_0}^\perp$ denotes the orientation bundle of the normal bundle $TN_0^\perp$.
Hence, disregarding the latter orientation bundle, we may regard points in this coadjoint orbit as codimension one submanifolds $N_0$ in $Q$, weighted by a volume density of the ambient space $Q$ along $N_0$.
\end{remark}

\appendix

\section{Comparison with the dual pair for the Euler equation}\label{SS:Euler}

A dual pair of moment maps associated to the Euler equations
of an ideal fluid has been described by J.~E.~Marsden and A.~Weinstein \cite{MW}; it justifies the existence of Clebsch canonical variables for ideal fluid motion and also explains the Hamiltonian structure of point vortex solutions in a geometric way.
It has been shown in \cite{GBV12} that the pair of moment maps restricted to the open subset of embeddings does indeed form a symplectic dual pair.
In this section we relate this dual pair to the EPContact dual pair, see \eqref{E:JLRE}.

\subsection{The dual pair for the Euler equation}

The space of smooth maps from a closed manifold $S$ into a symplectic manifold $(M,\omega)$ can be equipped with a symplectic structure once a volume density $\mu\in\Gamma^\infty(|\Lambda|_S\setminus S)$ has been fixed.
Recall that the space of maps $C^\infty(S,M)$ is a Fr\'echet manifold in a natural way.
The symplectic form on $C^\infty(S,M)$ can be described by
\begin{equation}\label{E:omegaCSM}
\omega^{C^\infty(S,M)}_\phi(U,V)=\int_S\omega(U,V)\mu,
\end{equation}
where $U,V\in \Gamma^\infty(\phi^*TM)=T_{\phi}C^\infty(S,M)$ are vector fields along $\phi\in C^\infty(S,M)$.
The group of symplectic diffeomorphisms, $\Diff(M,\omega)$, acts on $C^\infty(S,M)$ in a natural way from the left, preserving the symplectic form \eqref{E:omegaCSM}.
Moreover, the group of volume preserving diffeomorphisms, $\Diff(S,\mu)$, acts from the right by reparametrization, also preserving the symplectic form \eqref{E:omegaCSM}.
Clearly, these two actions commute.

Suppose the symplectic form on $M$ is exact, that is $\omega=d\theta$ for some $1$-form $\theta$ on $M$.
In this case the $\Diff(S,\mu)$ action on $C^\infty(S,M)$ is Hamiltonian with equivariant moment map.
This moment map, denoted by $J^{C^\infty(S,M)}_R$, is given by the composition
$$
C^\infty(S,M)\to\Omega^1(S)\overset{\mu}{=}\Omega^1(S,|\Lambda|_S)\subseteq\mathfrak X(S)^*\to\mathfrak X(S,\mu)^*.
$$
Here the first arrow is given by pull back of $\theta$; the second identification is via the volume density $\mu$; the third is the inclusion of smooth sections into distributional sections of $T^*S\otimes|\Lambda|_S$; and the fourth map is the dual of the canonical inclusion $\mathfrak X(S,\mu)\subseteq\mathfrak X(S)$.
We will write this as
\begin{equation}\label{E:JRCSM}
J^{C^\infty(S,M)}_R(\phi)=\phi^*\theta\otimes\mu
\qquad\text{or}\qquad
\langle J^{C^\infty(S,M)}_R(\phi),X\rangle=\int_S(\phi^*\theta)(X)\mu,
\end{equation}
where $\phi\in C^\infty(S,M)$ and $X\in\mathfrak X(S,\mu)$.

Via the Lie algebra homomorphism $C^\infty(M)\to\mathfrak X_\ham(M,\omega)$, the Poisson algebra $C^\infty(M)$ acts on $C^\infty(S,M)$ in a Hamiltonian fashion with infinitesimally equivariant moment map
$$
J_L^{C^\infty(S,M)}\colon C^\infty(S,M)\to C^\infty(M)^*
$$
given by
\begin{equation}\label{E:JLCSM}
J_L^{C^\infty(S,M)}(\phi):=\phi_*\mu\qquad\text{or}\qquad\langle J_L^{C^\infty(S,M)}(\phi),h\rangle=\int_S(\phi^*h)\mu
\end{equation}
where $\phi\in C^\infty(S,M)$ and $h\in C^\infty(M)$.
This moment map is in fact equivariant with respect to the natural action of the full symplectic group, $\Diff(M,\omega)$.
\footnote{Of course this moment map is even $\Diff(M)$-equivariant, but $\Diff(M)$ does not act symplectically on $C^\infty(S,M)$ nor does it act by Poisson maps on $C^\infty(M)$.}

Restricting the actions and moment maps to the open subset $\Emb(S,M)\subseteq C^\infty(S,M)$ of embeddings, we obtain a symplectic dual pair, see \cite{GBV12} and \cite[Section~4.2]{GBV15}:
\begin{equation}\label{idealfluid}
C^\infty(M)^*\xleftarrow{\quad J_L^{\Emb(S,M)}\quad}\Emb(S,M)\xrightarrow{\quad J_R^{\Emb(S,M)}\quad}\mathfrak X(S,\mu)^*
\end{equation}

\subsection{Comparison with the EPContact dual pair}

We will now specialize to the symplectization of a contact manifold $(P,\xi)$, that is, we consider $M=L^*\setminus P$ equipped with the symplectic form $\omega^M$ obtained by restricting the canonical $2$-form $\omega^{L^*}$ on the total space of $L^*$, cf.~Section~\ref{SS:P}.
We will relate the dual pair for the Euler equation \eqref{idealfluid} with the  EPContact dual pair constructed in Theorem~\ref{T:dp}.

Recall the symplectic manifold $\mathcal M=C^\infty_\isom(|\Lambda|_S^*,L^*)$ in \eqref{eme}, with Hamiltonian actions of the groups $\Diff(P,\xi)$ and $\Diff(S)$.
The volume density $\mu$ on $S$ provides an identification 
\begin{equation}\label{imu}
\iota_\mu\colon\mathcal M\to C^\infty(S,M),\quad \iota_\mu(\Phi):=\Phi\circ\hat\mu,
\end{equation}
where $\hat\mu\in\Gamma^\infty(|\Lambda|_S^*)$ denotes the section dual to $\mu$.
Let $j\colon\mathfrak X(P,\xi)\to C^\infty(M)$, $j(X):=h_X^M$, denote the Lie algebra homomorphism provided by \eqref{E:XPxi}, see also \eqref{E:Poisson}.
In view of \eqref{E:hXequi}, $j$ is equivariant over the homomorphism $\Diff(P,\xi)\to\Diff(M,\omega^M)$.
Note that the composition of $j$ with the action $C^\infty(M)\to\mathfrak X_\ham(M,\omega^M)$ yields a Lie algebra homomorphism $\mathfrak X(P,\xi)\to\mathfrak X_\ham(M,\omega^M)\subseteq\mathfrak X(M,\omega^M)$ corresponding to the homomorphism of groups $\Diff(P,\xi)\to\Diff(M,\omega^M)$, see \eqref{E:dhX}.
Finally, let $i\colon\mathfrak X(S,\mu)\to\mathfrak X(S)$ denote the natural inclusion.
Clearly, $i$ is equivariant over the inclusion $\Diff(S,\mu)\subseteq\Diff(S)$.

These maps give rise to the following diagram:
\begin{equation}\label{D:PM}
\vcenter{
\xymatrix{
\mathfrak X(P,\xi)^*&&\mathcal M\ar[ll]_{J_L^{\mathcal M}}\ar[d]_\cong^-{\iota_\mu}\ar[rr]^{J_R^{\mathcal M}}&&\mathfrak X(S)^*\ar[d]^-{i^*}
\\
C^\infty(M)^*\ar[u]^-{j^*}&&C^\infty(S,M)\ar[ll]_{J_L^{C^\infty(S,M)}}\ar[rr]^{J_R^{C^\infty(S,M)}}&&\mathfrak X(S,\mu)^*
}}
\end{equation}
Here $i^*$ and $j^*$ denote the (equivariant) maps dual to the homomorphisms $i$ and $j$, respectively.

\begin{proposition}
The diagram \eqref{D:PM} commutes.
The map $\iota_\mu$ in \eqref{imu} is a symplectomorphism which is equivariant over the inclusion $\Diff(S,\mu)\subseteq\Diff(S)$ and equivariant over the homomorphism $\Diff(P,\xi)\to\Diff(M,\omega^M)$.
Moreover, 
$$
\iota_\mu(\mathcal E)=\{\phi\in C^\infty(S,M):\pi^M\circ\phi\in\Emb(S,P)\},
$$ 
where $\pi^M\colon M\to P$ denotes the restriction of the canonical projection $\pi^{L^*}\colon L^*\to P$.
\end{proposition}

\begin{proof}
Clearly, $\iota_\mu$ is an equivariant diffeomorphism, see Remark~\ref{R:mutriv}.
It is symplectic in view of \eqref{E:omegaCSL} and \eqref{E:omegaCSM}.
The right hand side of the diagram commutes in view of \eqref{E:JRCSL} and \eqref{E:JRCSM}.
The left hand side of the diagram commutes in view of \eqref{E:JLCSL} and \eqref{E:JLCSM}.
\end{proof}

The first line in \eqref{D:PM} becomes a dual pair only when restricted to $\mathcal E$, while the second line needs to be restricted to $\Emb(S,M)$ to become a dual pair.
Note that the image $\iota_\mu(\mathcal E)$ is an open subset (strict, in general) of $\Emb(S,M)$.

\end{document}